\documentclass[10pt,reqno]{amsart}
\usepackage{amssymb, hyperref}
\usepackage{color}

\def\normo#1{\left\|#1\right\|}
\def\normb#1{\Big\|#1\Big\|}
\def\norm#1{\|#1\|}
\def\bra#1{\langle#1\rangle}
\def\wt#1{\widetilde{#1}}
\def\wh#1{\widehat{#1}}
\def\set#1{\{#1\}}

\newcommand{\T}{{\mathbb T}}
\newcommand{\R}{{\mathbb R}}
\newcommand{\Z}{{\mathbb Z}}
\newcommand{\ft}{{\mathcal{F}}}

\newcommand{\N}{{\mathcal{N}}}

\newcommand{\Sch}{{\mathcal{S}}}

\newcommand{\noi}{{\noindent}}

\numberwithin{equation}{section}
\newtheorem{theorem}{Theorem}[section]
\newtheorem{proposition}[theorem]{Proposition}
\newtheorem{lemma}[theorem]{Lemma}
\newtheorem{corollary}[theorem]{Corollary}
\newtheorem{remark}[theorem]{Remark}

\newcommand{\px}{\partial_x}
\newcommand{\pt}{\partial_t}

\begin{document}
\title[Fifth-order KdV equation]{Local well-posedness for the fifth-order KdV equations on $\T$}

\author[C. Kwak]{Chulkwang Kwak}
\email{ckkwak@kaist.ac.kr}
\address{Department of Mathematical Sciences, Korea Advanced Institute of Science and Technology, 291 Daehak-ro, Yuseong-gu, Daejeon 34141, Republic of Korea}

\begin{abstract}

This paper is a continuation of the paper \emph{Low regularity Cauchy problem for the fifth-order modified KdV equations on $\T$} \cite{Kwak2015}. In this paper, we consider the fifth-order equation in the Korteweg-de Vries (KdV) hierarchy as following: 
\begin{equation*}
\begin{cases}
\pt u - \px^5 u - 30u^2\px u + 20 \px u \px^2u + 10u \px^3 u = 0, \hspace{1em} (t,x) \in \R \times \T, \\
u(0,x) = u_0(x) \in H^s(\T)
\end{cases}.
\end{equation*}
 
We prove the local well-posedness of the fifth-order KdV equation for low regularity Sobolev initial data via the energy method. This paper follows almost same idea and argument as in \cite{Kwak2015}. Precisely, we use some conservation laws of the KdV Hamiltonians to observe the direction which the nonlinear solution evolves to. Besides, it is essential to use the short time $X^{s,b}$ spaces to control the nonlinear terms due to \emph{high $\times$ low $\Rightarrow$ high} interaction component in the non-resonant nonlinear term. We also use the localized version of the modified energy in order to obtain the energy estimate. 

As an immediate result from a conservation law in the scaling sub-critical problem, we have the global well-posedness result in the energy space $H^2$.

\end{abstract}

\thanks{}
\thanks{} 
\maketitle


\section{Introduction}\label{sec:intro}
The periodic Korteweg-de Vries (KdV) equation
\[\pt u + \px^3 u + 6u\px u = 0\]
is completely integrable in the sense that the equation admits  \emph{Lax pair} representations. Thanks to the inverse spectral method, it is well known that the KdV equation has a global smooth solution for any smooth initial data. Moreover, from the fact that the integrable Hamiltonian systems have the bi-Hamiltonian structure, there are infinitely many equations and corresponding Hamiltonians (so-called KdV hierarchy), and every equation in the hierarchy enjoys all conservation laws. The following are few conservation laws in the hierarchy: 
\begin{equation}\label{eq:hamiltonian}
M[u] := \int \frac12 u ,\quad  E[u] := \int \frac12 u^2 
\end{equation}
In this paper, we consider the following integrable fifth-order KdV equation:
\begin{equation}\label{eq:5kdv}
\begin{cases}
\pt u - \px^5 u - 30u^2\px u + 20 \px u \px^2u + 10u \px^3 u = 0, \hspace{1em} (t,x) \in \R \times \T, \\
u(0,x) = u_0(x) \in H^s(\T)
\end{cases}
\end{equation}
where $\T = [0,2\pi]$. Even if \eqref{eq:5kdv} and the other equations in the hierarchy have the integrable structure, it is still required the analytic theory of nonlinear dispersive equations to solve the low regularity Cauchy problem. In fact, in previous studies on the low regularity well-posedness problem for nonlinear dispersive equations (especially, under the non-periodic setting), the integrable structures were ignored. This work is a continuation of the paper \emph{Low regularity Cauchy problem for the fifth-order modified KdV equations on $\T$} \cite{Kwak2015} to show that, in the periodic setting, the complete integrability is partly needed to study on the low regularity well-posedness problem.\footnote{In fact, even if the integrability is neglected completely, the same result can be obtained for the integrable and also non-integrable equations. See Theorem \ref{thm:nonintegrable}.}

Generalizing coefficients in the nonlinear terms may break the integrable structure. The following equation generalizes \eqref{eq:5kdv} to non-integrable case:
\begin{equation}\label{eq:5kdv1}
\begin{cases}
\pt u - \px^5 u + a_1u^2\px u + a_2 \px u \px^2 u + a_3 u \px^3 u = 0, \hspace{1em} (t,x) \in \R \times \T, \\
u(0,x) = u_0(x) \in H^s(\T),
\end{cases}
\end{equation}
where $a_i$'s, $i=1,2,3$, are real constants. For studying \eqref{eq:5kdv1}, we can rely no longer on the property of the complete integrability.

Meanwhile, once one observes the Fourier coefficients of both \eqref{eq:5kdv} and \eqref{eq:5kdv1} (see \eqref{eq:5kdv2} below), one can find, in the nonlinear interactions, some resonant terms such as 
\[\int_{\T}u(t,x)\;dx \cdot \px^3 u, \hspace{1em} \norm{u(t)}_{L^2}^2\px u,\]
due to \eqref{eq:resonance function1} and \eqref{eq:resonance function2}. We call those terms the \emph{linear-like} resonant terms. Unfortunately, those terms are unfavorable as perturbations of the linear evolution in the low regularity Sobolev spaces. However, \eqref{eq:5kdv} particularly enjoys the Hamiltonian conservation laws in \eqref{eq:hamiltonian}, so all those terms in \eqref{eq:5kdv} change into
\[c_1 \px^3 u+c_2 \px u\]
for constants $c_1 \in \R, c_2 \ge 0$, and hence the linear part of the equation \eqref{eq:5kdv} can be expressed as 
\begin{equation}\label{eq:modified linear operator}
(\px^5 + \cdot \px^3 +c_2\px) u.
\end{equation}
This is one of different points in contrast with the non-periodic problem, and the reason why we focus on not \eqref{eq:5kdv1} but \eqref{eq:5kdv}.

The following is the main result in this paper:
\begin{theorem}\label{thm:main}
Let $s \ge 2$. For any $u_0 \in H^s(\T)$ specified
\begin{equation}\label{eq:level set}
\int_{\T} u_0(x) \; dx = \gamma_1, \hspace{2em} \int_{\T} (u_0(x))^2 \; dx = \gamma_2
\end{equation}
for some $\gamma_1 \in \R$, $\gamma_2 \ge 0$, there exists $T=T(\norm{u_0}_{H^s})>0$ such that \eqref{eq:5kdv} has a unique solution on $[-T,T]$ satisfying
\[u(t,x) \in C([-T,T];H^s(\T)) \cap F^s(T),\]
where the space $F^s(T)$\footnote{This space also depends on the initial data $u_0$ with \eqref{eq:level set}.} will be defined later. Moreover, the flow map $S_T : H^s \to C([-T,T];H^s(\T))$ is continuous on the level set in $H^s$ satisfying \eqref{eq:level set}.
\end{theorem}

\begin{remark}
The detailed proof of Theorem \ref{thm:main} follows the same argument as in the proof of Theorem 1.1 in \cite{Kwak2015}. Hence, in this paper, we only give the proofs of nonlinear and energy estimates. For the detailed argument, see \cite{Kwak2015}. 
\end{remark}

By simple calculation, we have 
\[a_1u^2\px u + a_2 \px u \px^2 u + a_3 u \px^3 u = \wt{a}_1\px(u^3) + \wt{a}_2\px(u\px^2 u) + \wt{a}_3\px((\px u)^2).\]
This observation gives the conservation of mean so that we do not need to stick to the integrable structure for $\int_{\T}u(t,x)\;dx \cdot\px^3 u$ term. Moreover, if one defines the nonlinear transformation for $\norm{u(t)}_{L^2}^2\px u$ term  by
\[\mathcal{NT}(u)(t,x) = v(t,x) := \frac{1}{\sqrt{2\pi}}\sum_{n \in \Z} e^{i(nx - 30n \int_0^t \norm{u(s)}_{L^2}^2 \; ds)}\wh{u}(t,n)\]
similarly as in \cite{Kwak2015}, $\norm{u(t)}_{L^2}^2\px u$ term can be also controlled, since it has a good property that the transformation is bi-continuous from the ball in $C([-T,T];H^s)$ to itself for $s \ge 0$\footnote{In \cite{Kwak2015}, the nonlinear transformation is bi-continuous for $s \ge 1/4$ due to the Sobolev embedding, which is used for controlling $\norm{u}_{L^4}$ component. But, in this paper, we do not need to use the Sobolev embedding and hence we have the bi-continuity property of nonlinear transformation for $s \ge 0$.}. Thus, we can also get the following corollary for the non-integrable equation \eqref{eq:5kdv1}:

\begin{corollary}\label{thm:nonintegrable}
Let $s \ge 2$. Then, \eqref{eq:5kdv1} is locally well-posed in $H^s(\T)$.\footnote{Similarly as Theorem \ref{thm:main}, local well-posedness result depends on the initial data in the level set satisfying
\[\int_{\T} u_0(x) \; dx = \gamma,\]
for some $\gamma \in \R$.}
\end{corollary}

From the $H^2$-level conservation law in the hierarchy
\[H_3[u](t) = \int \frac12 u_{xx}^2 -5u\px(u^2) + \frac52u^4 \; dx,\]
we can obtain the global well-posedness for \eqref{eq:5kdv}.
\begin{corollary}
The initial value problem \eqref{eq:5kdv} is globally well-posed in the energy space $H^2(\T)$.
\end{corollary}

The fifth-order KdV equation under the non-periodic setting has been widely studied. It was first studied by Ponce \cite{Ponce1993}. Since the strength of the nonlinearity is stronger than the advantage from the dispersive smoothing effect, it is required the energy method to prove the local well-posedness. Ponce used the energy method to prove the local well-posedness for Sobolev initial data $u_0 \in H^s$, $s \ge 4$, and afterward, Kwon \cite{Kwon2008-1} improved Ponce's result for $s > \frac52$. Kwon also used the energy method with corrections in addition to the refined Strichartz estimate, the Maximal function estimate, and the local smoothing estimate. Recently, Guo, Kwon and the author \cite{GKK2013}, and Kenig and Pilod \cite{KP2015} further improved the local result, independently. The method in both \cite{GKK2013} and \cite{KP2015} is the energy method based on the short time $X^{s,b}$ space, while the key energy estimates were shown by using an additional weight and modified energy, respectively. Similarly as the non-periodic setting, the bilinear estimate in the $X^{s,b}$ space
\[\norm{u \px^3 v}_{X^{s,b-1}} \le C \norm{u}_{X^{s,b}}\norm{v}_{X^{s,b}}\]
fails for all $s$ and $b \in \R$ under the periodic boundary condition. As a minor result in this paper, we have the following theorem:
\begin{theorem}\label{thm:bilinear}
For any $s,b \in \R$, the bilinear estimate
\[\norm{u\px^3v}_{X_{\tau-n^5}^{s,b-1}} \le C\norm{u}_{X_{\tau-n^5}^{s,b}}\norm{v}_{X_{\tau-n^5}^{s,b}}\]
fails.
\end{theorem}

The counter-example involves in \emph{high $\times$ low $\Rightarrow$ high} interaction component along the non-resonant phenomenon of the following type:
\[(P_{low}u)\cdot (P_{high}v_{xxx}).\]

The fifth-order KdV evolution provides quite strong modulation effect in the nonlinear interaction, but it is not enough to control three derivatives in the high frequency mode. Hence one cannot obtain the bilinear estimate in the standard $X^{s,b}$-norm. This observation gives a clue that the flow map seems not to be uniformly continuous, that is, the Picard iteration method does not work in this problem. The detailed example will be given in Section \ref{sec:bilinear}, later.

So far, we observe two enemies which disturb obtaining the local well-posedness result for the fifth-order KdV equation : \emph{linear-like} resonant terms and the failure of the bilinear estimate in the standard $X^{s,b}$ space.
The first enemy can be overcome by using the theory of complete integrability. From this, the linear operator of \eqref{eq:5kdv} slightly changes as in \eqref{eq:modified linear operator}, and with this, we use the short time modified $X^{s,b}$ to defeat the second enemy. Indeed, $X^{s,b}$ space taken in a short time interval depending on each frequency mode enables to obtain the bilinear estimate since it prevents the modulation to be low. This type of short time structure was first developed by Ionescu, Kenig and Tataru \cite{IKT2008} in the context of KP-I equation. 

Now we briefly give a sketch of the proof of Theorem \ref{thm:main} for self-containedness. The proof is based on the energy method in addition to Bona-Smith argument. As mentioned before, we first modify the linear propagator that absorbs all resonant interaction components. After then, we show following estimates in suitable functions spaces (which will be defined in Section \ref{sec:preliminary}):
\begin{equation}\label{eq:brief proof}
\begin{cases}
\begin{array}{ll}
\norm{u}_{F^s(T)} \lesssim \norm{u}_{E^s(T)} + \norm{\mathcal{N}(u)}_{N^s(T)} &\mbox{(Linear)}\\
\norm{\mathcal{N}(u)}_{N^s(T)} \lesssim \norm{u}_{F^s(T)}^2 + \norm{u}_{F^s(T)}^3 &\mbox{(Nonlinear)}\\
\norm{u}_{E^s(T)}^2 \lesssim (1+\norm{u_0}_{H^s})\norm{u_0}_{H^s}^2 + (1+\norm{u}_{F^s(T)}+\norm{u}_{F^s(T)}^2)\norm{u}_{F^s(T)}^3 &\mbox{(Energy)}
\end{array}
\end{cases}
\end{equation}
By the continuity argument, one can complete the local well-posedness of \eqref{eq:5kdv}.\footnote{To complete the local well-posedness argument, one needs to obtain similar estimates as in \eqref{eq:brief proof} for the difference of two solutions as well. However, the energy estimate for the difference of two solutions cannot be obtained in only $F^s$ space due to the lack of the symmetry among functions, so Bona-Smith argument (energy estimate in the intersection of the weaker ($F^0$) and the stronger ($F^{2s}$) spaces) is essential to close the energy estimate.}

On the other hand, in the second estimation in \eqref{eq:brief proof}, we can find the other different thing in contrast with the non-periodic problem. In view of, in particular, the $L^2$-block estimates (see Lemma \ref{lem:bi-L2} below) comparing with Lemma 3.1 in \cite{GKK2013}, since there is no dispersive smoothing effect under the periodic setting, we have worse estimates in the $L^2$-block estimates. Nevertheless, the short time length $(\approx 2^{-2k})$ at the $2^k$-frequency piece gives an advantage of the low modulation effect (two derivative gains: $|\tau - \mu(n)| \gtrsim 2^{2k}$), so the short time structure can cover the lack of the dispersive smoothing effect. 

Moreover, similarly as in \cite{Kwak2015} in the context of the fifth-order modified KdV equation on $\T$, we have to use the frequency localized modified energy in order to obtain the last estimation in \eqref{eq:brief proof}. Since the \emph{high-low} interaction component, when three derivatives are in the high frequency mode, is uncontrollable in even short time $F^s$ norm, the modified energy helps move two derivatives from the high frequency mode to the low frequency mode, and hence one can obtain the energy of solutions in $F^s$ space. For the non-periodic problem, the same difficulty appears in the same component only when the low frequency component has the largest modulation since there is dispersive smoothing effect in the non-periodic evolution. In that case, the modified energy still works (see \cite{KP2015}) and an additional weight works as well (see \cite{GKK2013}). We also encounter the technical difficulty to deal with new cubic resonant terms in the energy estimate. Fortunately, thanks to the symmetry among frequencies, all cubic resonant components do not make a difficulty no more (see Remarks \ref{rem:resonant3} and \ref{rem:resonant4} in Section \ref{sec:energy}).

The paper is organized as follows: In Section \ref{sec:preliminary}, we summarize some notations and define function spaces. In Section \ref{sec:bilinear}, we prove Theorem \ref{thm:bilinear} by giving a counter example. In Section \ref{sec:L2 block estimate}, we show the $L^2$ block bi- and trilinear estimates which are useful to obtain nonlinear and energy estimates. In Sections \ref{sec:nonlinear} and \ref{sec:energy}, we prove the nonlinear estimate and energy estimate, respectively.

\textbf{Acknowledgement.} The author would like to thank his advisor Soonsik Kwon for his helpful comments and encouragement through this research problem. Moreover, the author is grateful to Zihua Guo for his helpful advice to understand well the short time $X^{s,b}$ structure under the periodic setting. C.K. is partially supported by NRF(Korea) grant 2015R1D1A1A01058832.

\section{Preliminaries}\label{sec:preliminary}
For $x,y \in \R_+$, $x \lesssim y$ means that there exists $C>0$ such that $x \le Cy$, and $x \sim y$ means $x \lesssim y$ and $y\lesssim x$. We also use $\lesssim_s$ and $\sim_s$ as similarly, where the implicit constants depend on $s$. Let $a_1,a_2,a_3 \in \R$. The quantities $a_{max} \ge a_{med} \ge a_{min}$ can be conveniently defined to be the maximum, medium and minimum values of $a_1, a_2, a_3$ respectively.

For $Z = \R$ or $\Z$, let $\Gamma_k(Z)$ denote $(k-1)$-dimensional hyperplane by 
\[\set{\overline{x} = (x_1,x_2,...,x_k) \in Z^k : x_1 +x_2 + \cdots +x_k= 0}.\] 

For $f \in \Sch '(\R \times \T) $ we denote by $\wt{f}$ or $\ft (f)$ the Fourier transform of $f$ with respect to both spatial and time variables,
\[\wt{f}(\tau,n)=\frac{1}{\sqrt{2\pi}}\int_{\R}\int_{0}^{2\pi} e^{-ixn}e^{-it\tau}f(t,x)\; dxdt .\]
Moreover, we use $\ft_x$ (or $\wh{\;}$ ) and $\ft_t$ to denote the Fourier transform with respect to space and time variable respectively.

From the simple calculation
\[30u^2 u_x = 10(u^3)_x \hspace{1em} \mbox{and} \hspace{1em} 20u_xu_{xx} + 10 uu_{xxx} = 5(u_x^2)_x + 10(uu_{xx})_x,\]
we observe the Fourier coefficient in the spatial variable of \eqref{eq:5kdv} as
\begin{equation}\label{eq:5kdv2}
\begin{split}
\pt\wh{u}(n) - in^5\wh{u}(n) &= 10in \sum_{n_1+n_2+n_3=n} \wh{u}(n_1)\wh{u}(n_2)\wh{u}(n_3)\\
&+5in \sum_{n_1+n_2=n} n_1\wh{u}(n_1)n_2\wh{u}(n_2) \\
&+10in \sum_{n_1+n_2=n} \wh{u}(n_1)n_2^2\wh{u}(n_2).
\end{split}
\end{equation}
We consider the resonant relations for the quadratic and cubic terms in the right-hand side of \eqref{eq:5kdv2}
\begin{align}
H_2 &= H_2(n_1,n_2) := (n_1+n_2)^5 - n_1^5 - n_2^5 = \frac52n_1n_2(n_1+n_2)(n_1^2 + n_2^2 + (n_1+n_2)^2), \label{eq:resonance function1}\\
H_3 &= H_3(n_1,n_2,n_3):= (n_1+n_2+n_3)^5 - n_1^5 - n_2^5 - n_3^5 \nonumber \\
&= \frac52(n_1+n_2)(n_1+n_2)(n_2+n_3)(n_1^2+n_2^2+n_3^2+(n_1+n_2+n_3)^2). \label{eq:resonance function2}
\end{align}
Then we can observe that the resonant phenomenon appears only when $n_1n_2(n_1+n_2) = 0$ and $(n_1+n_2)(n_1+n_2)(n_2+n_3)=0$ in the quadratic and cubic terms, respectively. By using the conservation laws in \eqref{eq:hamiltonian} and gathering resonant terms in right-hand side of \eqref{eq:5kdv2}, we can rewrite \eqref{eq:5kdv2} as following:
\begin{equation}\label{eq:5kdv3}
\begin{split}
\pt\wh{u}(n) - i(n^5 + c_1n^3 + c_2n)\wh{u}(n) &= 30in|\wh{u}(n)|^2\wh{u}(n)\\
&+10in \sum_{\N_{3,n}} \wh{u}(n_1)\wh{u}(n_2)\wh{u}(n_3)\\
&+5in \sum_{\N_{2,n}} n_1\wh{u}(n_1)n_2\wh{u}(n_2) \\
&+10in \sum_{\N_{2,n}} \wh{u}(n_1)n_2^2\wh{u}(n_2)\\
&:= \wh{N}_1(u) + \wh{N}_2(u) + \wh{N}_3(u) + \wh{N}_4(u),
\end{split}
\end{equation}
where $c_1 = 10\wh{u}_0(0)$, $c_2 = 30 \norm{u_0}_{L^2}^2$, 
\[\N_{2,n} = \set{(n_1,n_2) \in \Z^2 : n_1+n_2=n \hspace{1em} \mbox{and} \hspace{1em} n_1n_2(n_1+n_2) \neq 0}\]
and
\[\N_{3,n} = \set{(n_1,n_2,n_3) \in \Z^3 : n_1+n_2+n_3=n \hspace{1em} \mbox{and} \hspace{1em} (n_1+n_2)(n_1+n_2)(n_2+n_3) \neq 0}.\]
We call the first term of the right-hand side of \eqref{eq:5kdv3} the \emph{Resonant} term and the others \emph{Non-resonant} term. We simply generalize $N_i(u)$ as $N_i(u,v)$, $i=3,4$, and $u_i(u,v,w)$, $i=1,2$, for the quadratic and cubic term.

We introduce that $X^{s,b}$-norm associated to \eqref{eq:5kdv3} which is given by 
\[\norm{u}_{X^{s,b}}=\norm{\bra{ \tau - \mu(n)}^b\bra{n}^s \ft(u)}_{L_{\tau}^2(\R;\ell_n^2(\Z))},\]
where 
\begin{equation}\label{eq:evolution}
\mu(n) = n^5 + c_1n^3 + c_2n
\end{equation}
and $\bra{\cdot} = (1+|\cdot|^2)^{1/2}$. The $X^{s,b}$ space turns out to be very useful in the study of low-regularity theory for the dispersive equations. The restricted norm method was first implemented in its current form by Bourgain \cite{Bourgain1993} and further developed by Kenig, Ponce and Vega \cite{KPV1996} and Tao \cite{Tao2001}.

Let $\Z_+ = \Z \cap [0,\infty]$. For $k \in \Z_+$, we set
\[I_0 = \set{n \in \Z : |n| \le 2} \hspace{1em} \mbox{ and } \hspace{1em} I_k = \set{n \in \Z : 2^{k-1} \le |n| \le 2^{k+1}}, \hspace{1em} k \ge 1.\]

Let $\eta_0: \R \to [0,1]$ denote a smooth bump function supported in $ [-2,2]$ and equal to $1$ in $[-1,1]$ with the following property of regularities:
\begin{equation}\label{eq:regularity}
\partial_n^{j} \eta_0(n) = O(\eta_0(n)/\bra{n}^j), \hspace{1em} j=0,1,2.
\end{equation}
For $k \in \Z_+ $, let 
\begin{equation}\label{eq:cut-off1}
\chi_0(n) = \eta_0(n), \hspace{1em} \mbox{and} \hspace{1em} \chi_k(n) = \eta_0(n/2^k) - \eta_0(n/2^{k-1}), \hspace{1em} k \ge 1,
\end{equation}
which is supported in $I_k$, and
\[\chi_{[k_1,k_2]}=\sum_{k=k_1}^{k_2} \chi_k \quad \mbox{ for any} \ k_1 \le k_2 \in \Z_+ .\]
$\{ \chi_k \}_{k \in \Z_+}$ is the inhomogeneous decomposition function sequence to the frequency space. For $k\in \Z_+$ let $P_k$ denote the
operators on $L^2(\T)$ defined by $\widehat{P_kv}(n)=\chi_k(n)\wh{v}(n)$. For $l\in \Z_+$ let
\[P_{\le l}=\sum_{k \le l}P_k, \quad P_{\ge l}=\sum_{k \ge l}P_k.\]
For the time-frequency decomposition, we use the cut-off function $\eta_j$, but the same as $\eta_j(\tau-\mu(n)) = \chi_j(\tau-\mu(n))$.

For $k,j \in \Z_+$ let
\[D_{k,j}=\{(\tau,n) \in \R \times \Z : \tau - \mu(n) \in I_j, n \in I_k \}, \hspace{2em} D_{k,\le j}=\cup_{l\le j}D_{k,l}.\]

For $k \in \Z_+$, we define the $X^{s,\frac12,1}$-type space $X_k$ for frequency localized functions,
\begin{eqnarray}\label{eq:Xk}
X_k=\left\{
\begin{array}{l}
f\in L^2(\R \times \Z): f(\tau,n) \mbox{ is supported in } \R \times I_k   \mbox{ and }\nonumber\\
\norm{f}_{X_k}:=\sum_{j=0}^\infty 2^{j/2}\norm{\eta_j(\tau-\mu(n))\cdot f(\tau,n)}_{L_{\tau}^2\ell_n^2}<\infty
\end{array}
\right\}.
\end{eqnarray}

As in \cite{IKT2008}, at frequency $2^k$ we will use the $X^{s,\frac12,1}$ structure given by the $X_k$-norm, uniformly on the $2^{-2k}$ time scale. For $k\in \Z_+$, we define function spaces
\begin{eqnarray*}
&& F_k=\left\{
\begin{array}{l}
f\in L^2(\R \times \T): \widehat{f}(\tau,n) \mbox{ is supported in } \R \times I_k \mbox{ and } \\
\norm{f}_{F_k}=\sup\limits_{t_k\in \R}\norm{\ft[f\cdot\eta_0(2^{2k}(t-t_k))]}_{X_k}<\infty
\end{array}
\right\},
\\
&&N_k=\left\{
\begin{array}{l}
f\in L^2(\R \times \T): \widehat{f}(\tau,n) \mbox{ is supported in } \R \times I_k \mbox{ and }  \\
\norm{f}_{N_k}=\sup\limits_{t_k\in \R}\norm{(\tau-\mu(n)+i2^{2k})^{-1}\ft[f\cdot\eta_0(2^{2k}(t-t_k))]}_{X_k}<\infty
\end{array}
\right\}.
\end{eqnarray*}
Since the spaces $F_k$ and $N_k$ are defined on the whole line in time variable, we define then local-in-time versions of the spaces in standard ways. For $T\in
(0,1]$ we define the normed spaces
\begin{align*}
F_k(T)=&\{f\in C([-T,T]:L^2): \norm{f}_{F_k(T)}=\inf_{\wt{f}=f \mbox{ in } [-T,T] \times \T }\norm{\wt f}_{F_k}\},\\
N_k(T)=&\{f\in C([-T,T]:L^2): \norm{f}_{N_k(T)}=\inf_{\wt{f}=f \mbox{ in } [-T,T] \times \T }\norm{\wt f}_{N_k}\}.
\end{align*}
We assemble these dyadic spaces in a Littlewood-Paley manner. For $s\geq 0$ and $T\in (0,1]$, we define function spaces solutions and
nonlinear terms:
\begin{eqnarray*}
&&F^{s}(T)=\left\{ u: \norm{u}_{F^{s}(T)}^2=\sum_{k=0}^{\infty}2^{2sk}\norm{P_k(u)}_{F_k(T)}^2<\infty \right\},
\\
&&N^{s}(T)=\left\{ u: \norm{u}_{N^{s}(T)}^2=\sum_{k=0}^{\infty}2^{2sk}\norm{P_k(u)}_{N_k(T)}^2<\infty \right\}.
\end{eqnarray*}

The solution space $F^s(T)$ is well-embedded in the classical solution space $C([-T,T];H^s)$.
\begin{proposition}\label{prop:small data1-1}
Let $s \ge 0$, $T \in (0,1]$ and $v \in F^s(T)$, then
\begin{equation}\label{eq:small data1.1}
\sup_{t \in [-T,T]} \norm{v(t)}_{H^s(\T)} \lesssim \norm{v}_{F^s(T)}. 
\end{equation} 
\end{proposition}

\begin{proof}
See \cite{GKK2013} and references therein.
\end{proof}

We define the dyadic energy space as follows: For $s\geq 0$ and $u\in C([-T,T]:H^\infty)$
\begin{eqnarray*}
\norm{u}_{E^{s}(T)}^2=\norm{P_{0}(u(0))}_{L^2}^2+\sum_{k\geq 1}\sup_{t_k\in [-T,T]}2^{2sk}\norm{P_k(u(t_k))}_{L^2}^2.
\end{eqnarray*}

\begin{lemma}[Properties of $X_k$]\label{lem:prop of Xk}
Let $k, l\in \Z_+$ and $f_k\in X_k$. Then
\begin{equation}\label{eq:prop1}
\begin{split}
&\sum_{j=l+1}^\infty 2^{j/2}\normo{\eta_j(\tau-\mu(n))\int_{\R}|f_k(\tau',n)|2^{-l}(1+2^{-l}|\tau-\tau'|)^{-4}d\tau'}_{L_{\tau}^2\ell_n^2}\\
&+2^{l/2}\normo{\eta_{\leq l}(\tau-\mu(n)) \int_{\R}|f_k(\tau',n)| 2^{-l}(1+2^{-l}|\tau-\tau'|)^{-4}d\tau'}_{L_{\tau}^2\ell_n^2}\lesssim\norm{f_k}_{X_k}.
\end{split}
\end{equation}
In particular, if $t_0\in \R$ and $\gamma\in \Sch(\R)$, then
\begin{eqnarray}\label{eq:prop2}
\norm{\ft[\gamma(2^l(t-t_0))\cdot \ft^{-1}(f_k)]}_{X_k}\lesssim
\norm{f_k}_{X_k}.
\end{eqnarray}
Moreover, from the definition of $X_k$-norm,
\[\normo{\int_{\R}|f_k(\tau',n)|\; d\tau'}_{\ell_n^2} \lesssim \norm{f_k}_{X_k}.\]
\end{lemma}
\begin{proof}
The proof of Lemma \ref{lem:prop of Xk} only depends on the summation over modulations, and there is no difference between the proof in the non-periodic and periodic settings. Hence we omit details and see \cite{GKK2013}.
\end{proof}

\begin{remark}\label{rem:modified space}
To prove Theorem \ref{thm:nonintegrable}, we can also define function spaces $\bar{X}_k$, $\bar{F}_k$, $\bar{N}_k$, $\bar{F}^s$ and $\bar{N}_k$ by using
\[\bar{\mu}(n) = n^5 + c_1n^3\]
instead of \eqref{eq:evolution}. 
\end{remark}

As in \cite{IKT2008}, for any $k\in \Z_+$ we define the set $S_k$ of $k$-\emph{acceptable} time multiplication factors 
\[S_k=\{m_k:\R\rightarrow \R: \norm{m_k}_{S_k}=\sum_{j=0}^{10} 2^{-2jk}\norm{\partial^jm_k}_{L^\infty}< \infty\}.\] 
Direct estimates using the definitions and \eqref{eq:prop2} show that for any $s\geq 0$ and $T\in (0,1]$
\[\begin{cases}
\normb{\sum\limits_{k\in \Z_+} m_k(t)\cdot P_k(u)}_{F^{s}(T)}\lesssim (\sup_{k\in \Z_+}\norm{m_k}_{S_k})\cdot \norm{u}_{F^{s}(T)};\\
\normb{\sum\limits_{k\in \Z_+} m_k(t)\cdot P_k(u)}_{N^{s}(T)}\lesssim (\sup_{k\in \Z_+}\norm{m_k}_{S_k})\cdot \norm{u}_{N^{s}(T)};\\
\normb{\sum\limits_{k\in \Z_+} m_k(t)\cdot P_k(u)}_{E^{s}(T)}\lesssim (\sup_{k\in \Z_+}\norm{m_k}_{S_k})\cdot \norm{u}_{E^{s}(T)}.
\end{cases}\]

\section{Proof of Theorem \ref{thm:bilinear}}\label{sec:bilinear}
In this section, we show the Theorem \ref{thm:bilinear}. The proof basically follows from the section 6 in \cite{KPV1996} associated to the KdV equation. As mentioned in the introduction, we observe the \emph{high $\times$ low $\Rightarrow$ high} interaction component in the non-resonance phenomenon, while, Kenig, Ponce, and Vega focused on the \emph{high $\times$ high $\Rightarrow$ high} interaction component. Actually, our examples of the KdV equation can be easily controlled in $X^{s,\frac12}$, because the size of maximum modulation is comparable to the square of high frequency size ($\approx N^2$) and hence this factor exactly eliminates the one derivative in the nonlinear term. In contrast to this, \eqref{eq:5kdv} has two more derivatives in nonlinear terms, and thus, one cannot control the this component in $X^{s,b}$-norm, although the advantage of the non-resonant effect is better than that of KdV equation. Now, we give examples satisfying
\begin{equation}\label{eq:fail}
\norm{u\px^3v}_{X^{s,b-1}} \nleq C\norm{u}_{X^{s,b}}\norm{v}_{X^{s,b}}.
\end{equation}
In the case of our examples, the bilinear estimate does not depend on the regularity $s$. So, it suffices to show \eqref{eq:fail} for any $b \in \R$. Fix $N \gg 1$. We first consider when $b > \frac14$. Let us define the functions
\[f(\tau,n) = a_n \chi_{\frac12}(\tau-n^5), \hspace{2em} g(\tau,n) = b_n \chi_{\frac12}(\tau-n^5),\]
where
\begin{equation*}
a_n = \begin{cases}1, \hspace{0.5em} n=1\\0, \hspace{0.5em}otherwise\end{cases} \hspace{1em}\mbox{and}\hspace{1em}b_n = \begin{cases}1, \hspace{0.5em}n=N-1\\0, \hspace{0.5em}otherwise\end{cases}.
\end{equation*}
We focus on the case that $|\tau- n^5|$ is the maximum modulation case. We put 
\[\wt{u}(\tau,n) = f(\tau,n) \hspace{2em} \wt{v}(\tau,n) = g(\tau,n),\]
then we need to calculate $\ft[u\px^3v](\tau,n)$. Since $\ft[u\px^3v](\tau,n) = (f \ast g)(\tau,n)$, performing the summation and integration with respect to $n_1, \tau_1$ variables gives
\begin{align*}
(f \ast g)(\tau,n) &= \sum_{n_1}a_{n_1}b_{n-n_1} \int_{\R}\chi_{\frac12}(\tau_1-n_1^5)\chi_{\frac12}(\tau-\tau_1-(n-n_1)^5)\; d\tau_1\\
&\cong c \sum_{n_1}a_{n_1}b_{n-n_1} \chi_{1}(\tau-n^5 + \frac52nn_1(n-n_1)(n^2+n_1^2 + (n-n_1)^2))\\
&\cong c \alpha_{n}\chi_{1}(\tau-n^5 +\frac52N(N-1)(N^2+1+(N-1)^2)),
\end{align*}
where 
\[\alpha_n = \begin{cases}1, \hspace{0.5em}n=N\\0, \hspace{0.5em}otherwise\end{cases}.\]
On the support of $(f \ast g)(\tau,n)$, since we have $|\tau - n^5| \sim N^4$, we finally obtain
\begin{align*}
\norm{u\px^3v}_{X^{s,b-1}} &= \norm{\bra{n}^s\bra{\tau-n^5}^{b-1}\ft[u\px^3v](\tau,n)}_{L_{\tau}^2\ell_n^2}\\
&\sim N^sN^3N^{4(b-1)},
\end{align*}
while
\[\norm{u}_{X^{s,b}}\norm{v}_{X^{s,b}} \sim N^s.\]
This imposes $b \le \frac14$ to succeed the bilinear estimate and hence, we show \eqref{eq:fail} when $b > \frac14$.

We now construct an example when $b \le \frac14$ and focus on the case that $|\tau - n^5|$ is too much smaller than the maximum modulation. In this case, we may assume that $|\tau_1-n_1^5|$ is the maximum modulation by symmetry of modulations. Set
\begin{equation*}
a_n = \begin{cases}1, \hspace{0.5em}n=-(N-1)\\0, \hspace{0.5em}otherwise\end{cases} \hspace{2em}\mbox{and}\hspace{2em}b_n = \begin{cases}1, \hspace{0.5em}n=N\\0, \hspace{0.5em}otherwise\end{cases}
\end{equation*}
and
\[f(\tau,n) = a_n \chi_{\frac11}(\tau-n^5), \hspace{1em} g(\tau,n) = b_n \chi_{\frac12}(\tau-n^5).\]
From the duality and change of variables, it suffices to consider
\[\norm{u\px^3v}_{X_{\tau-n^5}^{-s,-b}} \le C\norm{u}_{X_{\tau-n^5}^{-s,1-b}}\norm{v}_{X_{\tau-n^5}^{s,b}},\]
where 
\[\wt{u}(\tau,n) = f(\tau,n) \hspace{2em} \wt{v}(\tau,n) = g(\tau,n).\]
Similarly as before, we need to calculate $\ft[u\px^3v](\tau,n)$. Since $\ft[u\px^3v](\tau,n) = (f \ast g)(\tau,n)$, performing the summation and integration with respect to $n_1, \tau_1$ variables  gives
\begin{align*}
(f \ast g)(\tau,n) &= \sum_{n_1}a_{n_1}b_{n-n_1} \int_{\R}\chi_{\frac12}(\tau_1-n_1^5)\chi_{\frac12}(\tau-\tau_1-(n-n_1)^5)\; d\tau_1\\
&\cong c \sum_{n_1}a_{n_1}b_{n-n_1} \chi_{1}(\tau_2-n_2^5 + \frac52nn_1(n-n_1)(n^2+n_1^2 + (n-n_1)^2))\\
&\cong c \alpha_{n}\chi_{1}(\tau-n^5 -\frac52N(N-1)(N^2+(N-1)^2+1)),
\end{align*}
where 
\[\alpha_n = \begin{cases}1, \hspace{0.5em}n=1\\0, \hspace{0.5em}otherwise\end{cases}.\]
On the support of $(f \ast g)(\tau,n)$, since we have $|\tau - n^5| \sim N^4$, we finally obtain
\begin{align*}
\norm{u\px^3v}_{X^{-s,-b}} &= \norm{\bra{n}^{-s}\bra{\tau-n^5}^{-b}\ft[u\px^3v](\tau,n)}_{L_{\tau}^2\ell_n^2}\\
&\sim N^3N^{-4b},
\end{align*}
while
\[\norm{u}_{X^{-s,1-b}}\norm{v}_{X^{s,b}} \sim N^{-s}N^s \sim 1.\]
This imposes $b \ge \frac34$ and hence, we show \eqref{eq:fail} when $b \le \frac14$, which complete the proof of Theorem \ref{thm:bilinear}.

\section{$L^2$-block estimates}\label{sec:L2 block estimate}
In this section, we will give $L^2$-block estimates for bilinear estimates. For $n_1,n_2 \in \Z$, let
\[G(n_1,n_2) = \mu(n_1) + \mu(n_2) - \mu(n_1+n_2)\]
be the resonance function, which plays an important role in the bilinear $X^{s,b}$-type estimates. 

Let $\zeta_i = \tau_i - \mu(n_i)$. For compactly supported functions $f_i \in L^2(\R \times \T)$, $i=1,2,3$, we define 
\[J(f_1,f_2,f_3) = \sum_{n_3, \overline{\N}_{2,n_3}}\int_{\overline{\zeta} \in \Gamma_3(\R)}f_1(\zeta_1,n_1)f_2(\zeta_2,n_2)f_3(\zeta_3 + G(n_1,n_2),n_3),\]
where $\overline{\N}_{2,n_3}= \N_{2,-n_3}$ and $\overline{\zeta} = (\zeta_1,\zeta_2,\zeta_3+G(n_1,n_2))$. From the identities
\[n_1+n_2+n_3=0\]
and
\[\zeta_1+\zeta_2+\zeta_3 + G(n_1,n_2) =0\]
on the support of $J(f_1,f_2,f_3)$, we see that $J(f_1,f_2,f_3)$ vanishes unless
\begin{equation}\label{eq:support property}
\begin{array}{c}
2^{k_{max}} \sim 2^{k_{sub}}\\
2^{j_{max}} \sim \max(2^{j_{sub}}, |G|),
\end{array}
\end{equation}
where $|n_i| \sim 2^{k_i}$ and $|\zeta_i| \sim 2^{j_i}$, $i=1,2,3,4$. By simple change of variables in the summation and integration, we have
\[|J(f_1,f_2,f_3)|=|J(f_2,f_1,f_3)|=|J(f_3,f_2,f_1)|=|J(\overline{f}_1,f_2,f_3)|,\]
where $\overline{f}(\tau,n) = f(-\tau,-n)$.

\begin{lemma}\label{lem:bi-L2}
Let $k_i, j_i\in \Z_+$, $i=1,2,3$. Let $f_{k_i,j_i} \in L^2(\T\times\R) $ be nonnegative functions supported in $\widetilde{I}_{k_i} \times \widetilde{I}_{j_i}$.

\noi(a) Let $|k_{max}-k_{min}| \le 5$ and $j_1,j_2,j_3 \in \Z_+$.

\noi(a-1) If $j_{med} \le 3k_{max}$, then we have
\begin{eqnarray}\label{eq:bi-block estimate-a1.1}
   J(f_{k_1,j_1},f_{k_2,j_2},f_{k_3,j_3}) \lesssim 2^{(j_1+j_2+j_3)/2}2^{-(j_{med}+j_{max})/2}\prod_{i=1}^3 \|f_{k_i,j_i}\|_{L^2}.
\end{eqnarray}
\noi(a-2) Otherwise (i.e., if $j_{med} > 3k_{max}$), we have 
\begin{eqnarray}\label{eq:bi-block estimate-a1.2}
   J(f_{k_1,j_1},f_{k_2,j_2},f_{k_3,j_3}) \lesssim 2^{j_{min}/2}2^{j_{med}/4}2^{- \frac34 k_{max}}\prod_{i=1}^3 \|f_{k_i,j_i}\|_{L^2}.
\end{eqnarray}

\noi(b) Let $k_{min} \le k_{max}-10$.

\noi(b-1) If $(k_i,j_i) = (k_{min},j_{max})$ and $j_{med} \le 3k_{max}+k_{min}$, we have
\begin{eqnarray}\label{eq:bi-block estimate-b1.1}
 J(f_{k_1,j_1},f_{k_2,j_2},f_{k_3,j_3}) \lesssim 2^{(j_1+j_2+j_3)/2}2^{-(j_{med}+j_{max})/2}\prod_{i=1}^3 \|f_{k_i,j_i}\|_{L^2}.
\end{eqnarray}
\noi(b-2) If $(k_i,j_i) = (k_{min},j_{max})$ and $j_{med} > 3k_{max}+k_{min}$, we have
\begin{eqnarray}\label{eq:bi-block estimate-b1.2}
 J(f_{k_1,j_1},f_{k_2,j_2},f_{k_3,j_3}) \lesssim 2^{(j_1+j_2+j_3)/2}2^{-3k_{max}/2}2^{-k_{min}/2}2^{-j_{max}/2}\prod_{i=1}^3 \|f_{k_i,j_i}\|_{L^2}.
\end{eqnarray}
\noi(b-3) If $(k_i,j_i) \neq (k_{min},j_{max})$ and $j_{med} \le 4k_{max}$, we have
\begin{eqnarray}\label{eq:bi-block estimate-b1.3}
 J(f_{k_1,j_1},f_{k_2,j_2},f_{k_3,j_3}) \lesssim 2^{(j_1+j_2+j_3)/2}2^{-(j_{med}+j_{max})/2}\prod_{i=1}^3 \|f_{k_i,j_i}\|_{L^2}.
\end{eqnarray}
\noi(b-4) If $(k_i,j_i) \neq (k_{min},j_{max})$ and $j_{med} > 4k_{max}$, we have
\begin{eqnarray}\label{eq:bi-block estimate-b1.4}
 J(f_{k_1,j_1},f_{k_2,j_2},f_{k_3,j_3}) \lesssim 2^{(j_1+j_2+j_3)/2}2^{-2k_{max}}2^{-j_{max}/2}\prod_{i=1}^3 \|f_{k_i,j_i}\|_{L^2}.
\end{eqnarray}

\noi(c) For any $k_1,k_2,k_3,j_1,j_2,j_3 \in \Z_+$, then we
have
\begin{eqnarray}\label{eq:bi-block estimate-c1}
J(f_{k_1,j_1},f_{k_2,j_2},f_{k_3,j_3}) \lesssim 2^{j_{min}/2}2^{k_{min}/2}\prod_{i=1}^3 \|f_{k_i,j_i}\|_{L^2}.
\end{eqnarray}
\end{lemma}

\begin{proof}
The proof is very similar as the proof of Lemma 4.1 in \cite{Kwak2015} associated to the fifth-order modified KdV equation. For the sake of reader's convenience, we will give simple proof here. Let us assume that $j_1 \le j_2 \le j_3$ by the symmetry. In view of the proof of Lemma 4.1 in \cite{Kwak2015}, it suffices to consider
\[\sum_{\substack{n_3, \overline{\N}_{2,n_3}\\ \mu(n_1) + \mu(n_2) = \tau_3 + O(2^{j_2})}}f_{k_1,j_1}(n_1)f_{k_2,j_2}(n_2)f_{k_3,j_3}(n_1+n_2).\]
For (a), since $n_1+n_2+n_3 = 0$, we may assume that $|n_1 - n_2| \ll |n_1|$. Then by using the change of variable ($n_1' = n_1+n_2$), we have
\[\partial_{n_2} (\mu(n_2) + \mu(n_1' -n_2)) = 5n_2^4 - 5(n_1'-n_2) + 3c_1n_2^2 - 3c_1(n_1'-n_2)^2.\]
Thanks to the mean value theorem, since we have  
\[|n_2^4 - (n_1'-n_2)^4| \sim |n_1'|^3(n_2 - \frac{n_1'}{2})\]
and
\[|n_2^2 - (n_1'-n_2)^2| \sim |n_1'|(n_2 - \frac{n_1'}{2}),\]
that implies $n_2$ is contained in two intervals of length $O(2^{-3k_3/2}2^{j_3/2})$, i.e.
\[\mbox{the number of }n_2  \lesssim 2^{-3k_3/2}2^{j_2/2}.\]
Hence we obtain \eqref{eq:bi-block estimate-a1.1} and \eqref{eq:bi-block estimate-a1.2}.

For (b), we first consider $k_3 \neq k_{min}$ and assume that $k_1 \le k_2 \le k_3$ without loss of generality. Similarly as before, by using the change of variable ($n_2' = n_1 + n_2$), we have
\[\partial_{n_1} (\mu(n_1) + \mu(n_2' -n_1)) = 5n_1^4 - 5(n_2'-n_1) + 3c_1n_1^2 - 3c_1(n_2'-n_1)^2.\]
This implies $n_1$ is contained in an interval of length $O(2^{-4k_3}2^{j_2})$, i.e.
\[\mbox{the number of }n_1  \lesssim 2^{-4k_3}2^{j_2}.\]
If $k_3 = k_{min}$, we may assume $k_3 \le k_1 \le k_2$, and the same argument for $k_3 \neq k_{min}$ gives
\[\partial_{n_1} (\mu(n_1) + \mu(n_2' -n_1)) = 5n_1^4 - 5(n_2'-n_1) + 3c_1n_1^2 - 3c_1(n_2'-n_1)^2.\]
But, since $|n_2'| = |n_1+n_2| \sim 2^{k_3}$, $n_1$ is contained in two intervals of length $O(2^{-k_3}2^{-3k_2}2^{j_2})$, i.e.
\[\mbox{the number of }n_1  \lesssim 2^{-k_3}2^{-3k_2}2^{j_2},\]
which completes the proof of \eqref{eq:bi-block estimate-b1.1}, \eqref{eq:bi-block estimate-b1.2}, \eqref{eq:bi-block estimate-b1.3} and \eqref{eq:bi-block estimate-b1.4}.

For (c), we can easily obtain \eqref{eq:bi-block estimate-c1} by using the Cauchy-Schwarz inequality, and hence we complete the proof of Lemma \ref{lem:bi-L2}. 
\end{proof}

As an immediate consequence, we have the following corollary:

\begin{corollary}\label{cor:bi-L2}
Let $k_i, j_i\in \Z_+$, $i=1,2,3$. Let $f_{k_i,j_i} \in L^2(\T\times\R) $ be nonnegative functions supported in $\widetilde{I}_{k_i} \times \widetilde{I}_{j_i}$.

\noi(a) Let $|k_{max}-k_{min}| \le 5$ and $j_1,j_2,j_3 \in \Z_+$.

\noi(a-1) If $j_{med} \le 3k_{max}$, then we have
\begin{eqnarray}\label{eq:bi-block estimate-a2.1}
   \norm{\mathbf{1}_{D_{k_3,j_3}}(n,\tau) (f_{k_1,j_1}\ast f_{k_2,j_2})}_{L^2} \lesssim 2^{(j_1+j_2+j_3)/2}2^{-(j_{med}+j_{max})/2}\prod_{i=1}^2 \|f_{k_i,j_i}\|_{L^2}.
\end{eqnarray}
\noi(a-2) Otherwise (i.e., if $j_{med} > 3k_{max}$), we have 
\[   \norm{\mathbf{1}_{D_{k_3,j_3}}(n,\tau) (f_{k_1,j_1}\ast f_{k_2,j_2})}_{L^2} \lesssim 2^{j_{min}/2}2^{j_{med}/4}2^{- \frac34 k_{max}}\prod_{i=1}^2 \|f_{k_i,j_i}\|_{L^2}.\]

\noi(b) Let $k_{min} \le k_{max}-10$.

\noi(b-1) If $(k_i,j_i) = (k_{min},j_{max})$ and $j_{med} \le 3k_{max}+k_{min}$, we have
\begin{eqnarray}\label{eq:bi-block estimate-b2.1}
 \norm{\mathbf{1}_{D_{k_3,j_3}}(n,\tau) (f_{k_1,j_1}\ast f_{k_2,j_2})}_{L^2} \lesssim 2^{(j_1+j_2+j_3)/2}2^{-(j_{med}+j_{max})/2}\prod_{i=1}^2 \|f_{k_i,j_i}\|_{L^2}.
\end{eqnarray}
\noi(b-2) If $(k_i,j_i) = (k_{min},j_{max})$ and $j_{med} > 3k_{max}+k_{min}$, we have
\[ \norm{\mathbf{1}_{D_{k_3,j_3}}(n,\tau) (f_{k_1,j_1}\ast f_{k_2,j_2})}_{L^2} \lesssim 2^{(j_1+j_2+j_3)/2}2^{-3k_{max}/2}2^{-k_{min}/2}2^{-j_{max}/2}\prod_{i=1}^2 \|f_{k_i,j_i}\|_{L^2}.\]
\noi(b-3) If $(k_i,j_i) \neq (k_{min},j_{max})$ and $j_{med} \le 4k_{max}$, we have
\begin{eqnarray}\label{eq:bi-block estimate-b2.3}
 \norm{\mathbf{1}_{D_{k_3,j_3}}(n,\tau) (f_{k_1,j_1}\ast f_{k_2,j_2})}_{L^2} \lesssim 2^{(j_1+j_2+j_3)/2}2^{-(j_{med}+j_{max})/2}\prod_{i=1}^2 \|f_{k_i,j_i}\|_{L^2}.
\end{eqnarray}
\noi(b-4) If $(k_i,j_i) \neq (k_{min},j_{max})$ and $j_{med} > 4k_{max}$, we have
\begin{eqnarray}\label{eq:bi-block estimate-b2.4}
 \norm{\mathbf{1}_{D_{k_3,j_3}}(n,\tau) (f_{k_1,j_1}\ast f_{k_2,j_2})}_{L^2} \lesssim 2^{(j_1+j_2+j_3)/2}2^{-2k_{max}}2^{-j_{max}/2}\prod_{i=1}^2 \|f_{k_i,j_i}\|_{L^2}.
\end{eqnarray}

\noi(c) For any $k_1,k_2,k_3,j_1,j_2,j_3 \in \Z_+$, then we
have
\[\norm{\mathbf{1}_{D_{k_3,j_3}}(n,\tau) (f_{k_1,j_1}\ast f_{k_2,j_2})}_{L^2} \lesssim 2^{j_{min}/2}2^{k_{min}/2}\prod_{i=1}^2 \|f_{k_i,j_i}\|_{L^2}.\]
\end{corollary}

\section{Nonlinear estimates}\label{sec:nonlinear}
In this section, we prove the quadratic and cubic nonlinear estimates for the fifth-order KdV equation. In the following section, we assume that $|10\wh{u}_0(0)| \le 1$ in order to use
\[|G(n_1,n_2)| \gtrsim |n_1n_2(n_1+n_2)|(n_1^2 + n_2^2 + (n_1+n_2)^2)\]
in the support property \eqref{eq:support property}.

\begin{remark}
The assumption $|10\wh{u}_0(0)| \le 1$ is quite natural for the analysis in this problem, because this problem is scaling sub-critical. Indeed, by the Cauchy-Schwarz inequality, we have
\[|\wh{u}_0(0)| \lesssim \norm{u_0}_{L^2} \le \norm{u_0}_{H^s},\]
for $s \ge 0$. Hence, the smallness of the initial data always guarantees the smallness of mean.  
\end{remark}

\begin{lemma}[Resonance estimate]\label{lem:resonant2}
Let $k \ge 0$. Then, we have
\begin{equation}\label{eq:resonant2-1}
\norm{P_kN_{1}(u,v,w)}_{N_{k}} \lesssim 2^{-k}\norm{P_ku}_{F_{k}}\norm{P_kv}_{F_{k}}\norm{P_kw}_{F_{k}}.
\end{equation} 
\end{lemma}

\begin{proof}
From the definitions of $N_1(u,v,w)$ and $N_k$ norm, the left-hand side of \eqref{eq:resonant2-1} is bounded by
\begin{equation}\label{eq:resonant2-2}
\begin{aligned}
\sup_{t_k \in \R} &\Big\|(\tau - \mu(n)+ i2^{2k})^{-1}2^{k}\mathbf{1}_{I_k}(n)\ft\left[\eta_0\left(2^{2k-2}(t-t_k)\right)P_ku\right] \\
&\hspace{3em}\ast \ft\left[\eta_0\left(2^{2k-2}(t-t_k)\right)P_kv\right] \ast \ft\left[\eta_0\left(2^{2k-2}(t-t_k)\right)P_kw\right]\Big\|_{X_k}
\end{aligned}
\end{equation}
Set $u_k = \ft\left[\eta_0\left(2^{2k-2}(t-t_k)\right)P_ku\right], v_k = \ft\left[\eta_0\left(2^{2k-2}(t-t_k)P_kv\right)\right]$ and $w_k = \ft\left[\eta_0\left(2^{2k-2}(t-t_k)\right)P_kw\right]$. We decompose each of $u_k,v_k$ and $w_k$ into modulation dyadic pieces as $u_{k,j_1}(\tau,n) = u_k(\tau,n)\eta_{j_1}(\tau - \mu(n))$, $v_{k,j_2}(\tau,n) = v_k(\tau,n)\eta_{j_2}(\tau - \mu(n))$ and $w_{k,j_3}(\tau,n) = w_k(\tau,n)\eta_{j_3}(\tau - \mu(n))$, respectively, with usual modification like $f_{\le j}(\tau) = f(\tau)\eta_{\le j}(\tau-\mu(n))$. Then, from the Cauchy-Schwarz inequality, \eqref{eq:resonant2-2} is bounded by
\begin{equation}\label{eq:resonant2-3}
2^{k}\sum_{j_4 \ge 0} \frac{2^{j_4/2}}{\max(2^{j_4},2^{2k})} \sum_{j_1,j_2,j_3 \ge 2k} 2^{(j_{min}+j_{thd})/2}\norm{u_{k,j_1}}_{L_{\tau}^2\ell_n^2}\norm{v_{k,j_2}}_{L_{\tau}^2\ell_n^2}\norm{w_{k,j_3}}_{L_{\tau}^2\ell_n^2}.
\end{equation}
Since $j_1,j_2,j_3 \ge 2k$, if $j_4 \le 2k$, we have $(\max(2^{j_4},2^{2k}))^{-1}2^{(j_{min}+j_{thd})/2} \lesssim 2^{(j_1+j_2+j_3)/2}2^{-3k}$, otherwise, $(\max(2^{j_4},2^{2k}))^{-1}2^{(j_{min}+j_{thd})/2} \lesssim 2^{-j_4}2^{(j_1+j_2+j_3)/2}2^{-k}$, and hence by performing all summations over $j_1, j_2, j_3$ and $j_4$, we have
\begin{align*}
\eqref{eq:resonant2-3} &\lesssim 2^{-k}\sum_{j_1,j_2,j_3\ge 2k} 2^{(j_1+j_2+j_3)/2}\norm{u_{k,j_1}}_{L_{\tau}^2\ell_n^2}\norm{v_{k,j_2}}_{L_{\tau}^2\ell_n^2}\norm{w_{k,j_3}}_{L_{\tau}^2\ell_n^2}\\
&\lesssim 2^{-k} \norm{u_k}_{X_k}\norm{v_k}_{X_k}\norm{w_k}_{X_k},
\end{align*}
which implies \eqref{eq:resonant2-1}. 
\end{proof}

Next, we consider the main nonlinear estimates in the fifth-order KdV equation. The first lemma below is to estimate the \emph{high-low} interaction component. As mentioned in Sections \ref{sec:intro} and \ref{sec:bilinear}, the estimation of the \emph{high-low} interaction component fails in the standard $X^{s,b}$ space because of due to the much more derivatives in high frequency mode and the lack of dispersive smoothing effect. Hence the following lemma shows the choice of short time length ($\approx \mbox{(frequency)}^{-2}$) is well adapted to estimate bilinear terms in the fifth-order KdV equation. 

\begin{lemma}[High-low $\Rightarrow$ high]\label{lem:bi-nonres1}
Let $k_3 \ge 20$, $|k_2-k_3| \le 5$ and $0\le k_1 \le k_3 - 10$. Then, we have 
\begin{equation}\label{eq:bi-nonres1-1}
\begin{aligned}
\norm{P_{k_3}N_3(P_{k_1}u,P_{k_2}v)}_{N_{k_3}} &+ \norm{P_{k_3}N_4(P_{k_1}u,P_{k_2}v)}_{N_{k_3}}\lesssim 2^{-k_1/2}\norm{P_{k_1}u}_{F_{k_1}}\norm{P_{k_2}v}_{F_{k_2}}.
\end{aligned}
\end{equation} 
\end{lemma}

\begin{proof}
We follow the similar argument as in the section 5 in \cite{Kwak2015}. By the definitions of $N_{k}$ and $X_{k}$, the left-hand side of \eqref{eq:bi-nonres1-1} is dominated by
\begin{equation}\label{eq:bi-nonres1-2}
\begin{aligned}
\sup_{t_k \in \R} &\Big\|(\tau_3 - \mu_2(n_3)+ i2^{2k_3})^{-1}2^{3k_3}\mathbf{1}_{I_{k_3}}(n_3) \\
&\cdot\ft\left[\eta_0\left(2^{2k_3-2}(t-t_k)\right)P_{k_1}u\right] \ast \ft\left[\eta_0\left(2^{2k_3-2}(t-t_k)P_{k_2}v\right)\right]\Big\|_{X_{k_3}}.
\end{aligned}
\end{equation}
Set $f_{k_1} = \ft\left[\eta_0\left(2^{2k_3-2}(t-t_k)\right)P_{k_1}u\right]$ and $f_{k_2} = \ft\left[\eta_0\left(2^{2k_3-2}(t-t_k)\right)P_{k_2}v\right]$. We further decompose $f_{k_i}$ into modulation dyadic pieces as $f_{k_i,j_i}(\tau,n) = f_{k_i}(\tau,n)\eta_{j_i}(\tau -\mu_2(n))$, $j=1,2$, with usual modification $f_{k,\le j}(\tau,n) = f_{k}(\tau,n)\eta_{\le j}(\tau-\mu_2(n))$. Then \eqref{eq:bi-nonres1-2} is bounded by
\begin{equation}\label{eq:bi-nonres1-3}
2^{3k_3}\sum_{j_3\ge0}\frac{2^{j_3/2}}{\max(2^{j_4}, 2^{2k_3})}\sum_{j_1,j_2 \ge 2k_3}\norm{\mathbf{1}_{D_{k_3,j_3}}(f_{k_1,j_1} \ast f_{k_2,j_2})}_{L_{\tau_3}^2\ell_{n_2}^2}
\end{equation}
If $j_3 \le 2k_3$, we use \eqref{eq:bi-block estimate-b2.1} -- \eqref{eq:bi-block estimate-b2.4}, separately, to estimate $\norm{\mathbf{1}_{D_{k_3,j_3}}(f_{k_1,j_1} \ast f_{k_2,j_2})}_{L_{\tau_3}^2\ell_{n_2}^2}$, then we have
\[2^{3k_3}\sum_{j_3\le 2k_3}2^{j_3/2}2^{-2k_3}\sum_{\substack{j_1,j_2 \ge 2k_3\\j_1 = j_{max}\\j_{med} \le 3k_3+k_1}}2^{j_{min}/2}\norm{f_{k_1,j_1}}_{L_{\tau}^2\ell_{n}^2}\norm{f_{k_2,j_2}}_{L_{\tau}^2\ell_{n}^2},\]
\[2^{3k_3}\sum_{j_3\le 2k_3}2^{j_3/2}2^{-2k_3}\sum_{\substack{j_1,j_2 \ge 2k_3\\j_1 = j_{max}\\j_{med} > 3k_3+k_1}}2^{(j_1+j_2+j_3)/2}2^{-3k_3/2}2^{-k_1/2}2^{-j_{max}/2}\norm{f_{k_1,j_1}}_{L_{\tau}^2\ell_{n}^2}\norm{f_{k_2,j_2}}_{L_{\tau}^2\ell_{n}^2},\]
\[2^{3k_3}\sum_{j_3\le 2k_3}2^{j_3/2}2^{-2k_3}\sum_{\substack{j_1,j_2 \ge 2k_3\\j_1 \neq j_{max}\\j_{med} \le 4k_3}}2^{j_{min}/2}\norm{f_{k_1,j_1}}_{L_{\tau}^2\ell_{n}^2}\norm{f_{k_2,j_2}}_{L_{\tau}^2\ell_{n}^2},\]
or
\[2^{3k_3}\sum_{j_3\le 2k_3}2^{j_3/2}2^{-2k_3}\sum_{\substack{j_1,j_2 \ge 2k_3\\j_1 \neq j_{max}\\j_{med} > 4k_3}}2^{(j_1+j_2+j_3)/2}2^{-2k_3}2^{-j_{max}/2}\norm{f_{k_1,j_1}}_{L_{\tau}^2\ell_{n}^2}\norm{f_{k_2,j_2}}_{L_{\tau}^2\ell_{n}^2}.\]
By performing the summation over $j_1,j_2$ and $j_3$ for each case with $j_{max} \ge 4k_3+k_1$, we have
\[\begin{aligned}\eqref{eq:bi-nonres1-3} &\lesssim 2^{-k_1/2}\sum_{j_1,j_2}2^{(j_1+j_2)/2}\norm{f_{k_1,j_1}}_{L_{\tau}^2\ell_{n}^2}\norm{f_{k_2,j_2}}_{L_{\tau}^2\ell_{n}^2}\\
&\lesssim 2^{-k_1/2} \norm{P_{k_1}u}_{F_{k_1}}\norm{P_{k_2}v}_{F_{k_2}}.\end{aligned}\]
If $j_3 > 2k_3$, similarly as before, we also have
\[\begin{aligned}\eqref{eq:bi-nonres1-3} &\lesssim 2^{3k_3}\sum_{j_3 > 2k_3}2^{j_3/2}2^{-j_3}\sum_{\substack{j_1,j_2 \ge 2k_3\\j_1 = j_{max}\\j_{med} \le 3k_3+k_1}}2^{j_{min}/2}\norm{f_{k_1,j_1}}_{L_{\tau}^2\ell_{n}^2}\norm{f_{k_2,j_2}}_{L_{\tau}^2\ell_{n}^2} \\
&\lesssim 2^{-k_1/2}\sum_{j_1,j_2}2^{(j_1+j_2)/2}\norm{f_{k_1,j_1}}_{L_{\tau}^2\ell_{n}^2}\norm{f_{k_2,j_2}}_{L_{\tau}^2\ell_{n}^2}\\
&\lesssim 2^{-k_1/2} \norm{P_{k_1}u}_{F_{k_1}}\norm{P_{k_2}v}_{F_{k_2}}.\end{aligned}\]
Remark that one can know that the case when $j_1 = j_{max}$ and $j_{med} \le 3k_3+k_1$ gives the worst bound. Thus, we complete the proof of Lemma \ref{lem:bi-nonres1}.
\end{proof}

\begin{lemma}[High-high $\Rightarrow$ high]\label{lem:bi-nonres2}
Let $k_3 \ge 20$ and $|k_1-k_3|, |k_2-k_3| \le 5$. Then, we have 
\begin{equation}\label{eq:bi-nonres2-1}
\begin{aligned}
\norm{P_{k_3}N_3(P_{k_1}u,P_{k_2}v)}_{N_{k_3}} &+ \norm{P_{k_3}N_{4}(P_{k_1}u,P_{k_2})}_{N_{k_3}}\lesssim 2^{-k_2/2} \norm{P_{k_1}u}_{F_{k_1}}\norm{P_{k_2}v}_{F_{k_2}}
\end{aligned}
\end{equation} 
\end{lemma}

\begin{proof}
In view of the proof of Lemma \ref{lem:bi-nonres1}, \eqref{eq:bi-nonres2-1} is dominated by
\begin{equation}\label{eq:bi-nonres2-2}
2^{3k_3}\sum_{j_3\ge0}\frac{2^{j_3/2}}{\max(2^{j_4}, 2^{2k_3})}\sum_{j_1,j_2 \ge 2k_3}\norm{\mathbf{1}_{D_{k_3,j_3}}(f_{k_1,j_1} \ast f_{k_2,j_2})}_{L_{\tau_3}^2\ell_{n_2}^2}.
\end{equation}
Similarly as above, it is enough to consider the case when $j_3 \ge 2k_3$ and $j_{med} \le 3k_3$. By using \eqref{eq:bi-block estimate-a2.1} to estimate $\norm{\mathbf{1}_{D_{k_3,j_3}}(f_{k_1,j_1} \ast f_{k_2,j_2})}_{L_{\tau_3}^2\ell_{n_2}^2}$, then we have
\begin{align*}
\eqref{eq:bi-nonres2-2} &\lesssim 2^{3k_3}\sum_{j_3\ge 2k_3}2^{-j_3/2}\sum_{j_1,j_2 \ge 2k_3}2^{j_{min}/2}\norm{f_{k_1,j_1}}_{L_{\tau}^2\ell_n^2}\norm{f_{k_2,j_2}}_{L_{\tau}^2\ell_n^2}\\
&\lesssim 2^{3k_3}2^{-5k_3/2}2^{-k_3}\sum_{j_1,j_2 \ge 2k_3}2^{(j_1+j_2)/2}\norm{f_{k_1,j_1}}_{L_{\tau}^2\ell_n^2}\norm{f_{k_2,j_2}}_{L_{\tau}^2\ell_n^2}\\
&\lesssim 2^{-k_2/2} \norm{P_{k_1}u}_{F_{k_1}}\norm{P_{k_2}v}_{F_{k_2}},
\end{align*}
since $j_{max} \ge 5k_3$. Hence, we complete the proof of Lemma \ref{lem:bi-nonres2}.
\end{proof}

\begin{lemma}[High-high $\Rightarrow$ low]\label{lem:bi-nonres3}
Let $k_2 \ge 20$, $|k_1-k_2| \le 5$ and $0\le k_3 \le k_2 -10$. Then, we have 
\begin{equation}\label{eq:bi-nonres3-1}
\begin{aligned}
\norm{P_{k_3}N_3(P_{k_1}u,P_{k_2}v)}_{N_{k_3}} &+ \norm{P_{k_3}N_{4}(P_{k_1}u,P_{k_2})}_{N_{k_3}}\lesssim k_22^{k_2}2^{-3k_3/2}\norm{P_{k_1}u}_{F_{k_1}}\norm{P_{k_2}v}_{F_{k_2}}
\end{aligned}
\end{equation} 
\end{lemma}
\begin{proof}
Since $k_3 \le k_2 -10$, one can observe that the $N_{k_3}$-norm is taken on the time intervals of length $2^{-2k_3}$, while each $F_{k_i}$-norm is taken on shorter time intervals of length $2^{-2k_i}$, $i=1,2$. Thus, we divide the time interval, which is taken in $N_{k_3}$-norm, into $2^{2k_2-2k_3}$ intervals of length $2^{-2^{2k_2}}$ in order to obtain the right-hand side of \eqref{eq:bi-nonres3-1}. Let $\gamma: \R \to [0,1]$ denote a smooth function supported in $[-1,1]$ with $ \sum_{m\in \Z} \gamma^2(x-m) \equiv 1$. From the definition of $N_{k_3}$-norm, the left-hand side of \eqref{eq:bi-nonres3-1} is dominated by
\begin{equation}\label{eq:bi-nonres3-1.1}
\begin{split}
\sup_{t_k\in \R}&2^{k_3}2^{2k_2}\Big\|(\tau_3-\mu(n_3) +i 2^{2k_3})^{-1}\mathbf{1}_{I_{k_3}}\\
&\cdot  \sum_{|m| \le C 2^{2k_2-2k_3}} \ft[\eta_0(2^{2k_3}(t-t_k))\gamma (2^{2k_2}(t-t_k)-m)P_{k_1}u]\\
 &\hspace{6em}\ast \ft[\eta_0(2^{2k_3}(t-t_k))\gamma (2^{2k_2}(t-t_k)-m)P_{k_2}v]\Big\|_{X_{k_3}}.
\end{split}
\end{equation}
As similarly in the proof of above Lemma, \eqref{eq:bi-nonres3-1.1} is bounded by 
\begin{equation}\label{eq:bi-nonres3-2}
2^{4k_2}2^{-k_3}\sum_{j_3\ge0}\frac{2^{j_3/2}}{\max(2^{j_4}, 2^{2k_3})}\sum_{j_1,j_2 \ge 2k_2}\norm{\mathbf{1}_{D_{k_3,j_3}}(f_{k_1,j_1} \ast f_{k_2,j_2})}_{L_{\tau_3}^2\ell_{n_2}^2}.
\end{equation}
If $j_3 < 2k_3$, since $j_3 \neq j_{max}$, we use \eqref{eq:bi-block estimate-b2.3} for $j_{med} \le 4k_2$ case to estimate $\norm{\mathbf{1}_{D_{k_3,j_3}}(f_{k_1,j_1} \ast f_{k_2,j_2})}_{L_{\tau_3}^2\ell_{n_2}^2}$, then we have
\begin{align*}
\eqref{eq:bi-nonres3-2} &\lesssim 2^{4k_2}2^{-k_3}\sum_{j_3 < 2k_3}2^{j_3/2}2^{-2k_3}\sum_{j_1,j_2 \ge 2k_2}2^{j_{min}/2}\norm{f_{k_1,j_1}}_{L_{\tau}^2\ell_n^2}\norm{f_{k_2,j_2}}_{L_{\tau}^2\ell_n^2}\\
&\lesssim 2^{4k_2}2^{-k_3}2^{-2k_2}2^{-k_3/2}2^{-k_2}\sum_{j_1,j_2 \ge 2k_2}2^{(j_1+j_2)/2}\norm{f_{k_1,j_1}}_{L_{\tau}^2\ell_n^2}\norm{f_{k_2,j_2}}_{L_{\tau}^2\ell_n^2}\\
&\lesssim 2^{k_2}2^{-3k_3/2}\norm{P_{k_1}u}_{F_{k_1}}\norm{P_{k_2}v}_{F_{k_2}}.
\end{align*}
If $2k_3 \le j_3 < 2k_2$, similarly as above, we have
\begin{align*}
\eqref{eq:bi-nonres3-2} &\lesssim 2^{4k_2}2^{-k_3}\sum_{2k_3 \le j_3 < 2k_2}2^{-j_3/2}\sum_{j_1,j_2 \ge 2k_2}2^{j_{min}/2}\norm{f_{k_1,j_1}}_{L_{\tau}^2\ell_n^2}\norm{f_{k_2,j_2}}_{L_{\tau}^2\ell_n^2}\\
&\lesssim k_22^{4k_2}2^{-k_3}2^{-2k_2}2^{-k_3/2}2^{-k_2}\sum_{j_1,j_2 \ge 2k_2}2^{(j_1+j_2)/2}\norm{f_{k_1,j_1}}_{L_{\tau}^2\ell_n^2}\norm{f_{k_2,j_2}}_{L_{\tau}^2\ell_n^2}\\
&\lesssim k_22^{k_2}2^{-3k_3/2}\norm{P_{k_1}u}_{F_{k_1}}\norm{P_{k_2}v}_{F_{k_2}}.
\end{align*}
Now, let us assume that $j_3 \ge 2k_2$. If $j_3 \neq j_{max}$, since $2^{j_{min}} \lesssim 2^{j_1+j_2}2^{-j_{max}}$, we use \eqref{eq:bi-block estimate-b2.3} for $j_{med} \le 4k_2$ case to estimate $\norm{\mathbf{1}_{D_{k_3,j_3}}(f_{k_1,j_1} \ast f_{k_2,j_2})}_{L_{\tau_3}^2\ell_{n_2}^2}$, then we have
\begin{align*}
\eqref{eq:bi-nonres3-2} &\lesssim 2^{4k_2}2^{-k_3}\sum_{ j_3 \ge 2k_2}2^{-j_3/2}\sum_{j_1,j_2 \ge 2k_2}2^{j_{min}/2}\norm{f_{k_1,j_1}}_{L_{\tau}^2\ell_n^2}\norm{f_{k_2,j_2}}_{L_{\tau}^2\ell_n^2}\\
&\lesssim 2^{4k_2}2^{-k_3}2^{-2k_2}2^{-k_3/2}2^{-k_2}\sum_{j_1,j_2 \ge 2k_2}2^{(j_1+j_2)/2}\norm{f_{k_1,j_1}}_{L_{\tau}^2\ell_n^2}\norm{f_{k_2,j_2}}_{L_{\tau}^2\ell_n^2}\\
&\lesssim 2^{k_2}2^{-3k_3/2}\norm{P_{k_1}u}_{F_{k_1}}\norm{P_{k_2}v}_{F_{k_2}}.
\end{align*}
Similarly as before, when $j_3 = j_{max}$, since $j_3 \ge 4k_2+k_3$, we use \eqref{eq:bi-block estimate-b2.1} for $j_{med} \le 3k_2+k_3$ case to estimate $\norm{\mathbf{1}_{D_{k_3,j_3}}(f_{k_1,j_1} \ast f_{k_2,j_2})}_{L_{\tau_3}^2\ell_{n_2}^2}$, then we have
\begin{align*}
\eqref{eq:bi-nonres3-2} &\lesssim 2^{4k_2}2^{-k_3}\sum_{ j_3 \ge 4k_2+k_3}2^{-j_3/2}\sum_{j_1,j_2 \ge 2k_2}2^{j_{min}/2}\norm{f_{k_1,j_1}}_{L_{\tau}^2\ell_n^2}\norm{f_{k_2,j_2}}_{L_{\tau}^2\ell_n^2}\\
&\lesssim 2^{4k_2}2^{-k_3}2^{-2k_2}2^{-k_3/2}2^{-k_2}\sum_{j_1,j_2 \ge 2k_2}2^{(j_1+j_2)/2}\norm{f_{k_1,j_1}}_{L_{\tau}^2\ell_n^2}\norm{f_{k_2,j_2}}_{L_{\tau}^2\ell_n^2}\\
&\lesssim 2^{k_2}2^{-3k_3/2}\norm{P_{k_1}u}_{F_{k_1}}\norm{P_{k_2}v}_{F_{k_2}}.
\end{align*}
Thus, we complete the proof of Lemma \ref{lem:bi-nonres3}.
\end{proof}

\begin{lemma}[low-low $\Rightarrow$ low]\label{lem:bi-nonres4}
Let $0 \le k_1,k_2,k_3 \le 200$. Then, we have 
\begin{equation}\label{eq:bi-nonres4-1}
\begin{aligned}
\norm{P_{k_3}N_3(P_{k_1}u,P_{k_2}v)}_{N_{k_3}} &+ \norm{P_{k_3}N_{4}(P_{k_1}u,P_{k_2})}_{N_{k_3}}\lesssim \norm{P_{k_1}u}_{F_{k_1}}\norm{P_{k_2}v}_{F_{k_2}}
\end{aligned}
\end{equation} 
\end{lemma}
\begin{proof}
Similarly as in the proof of Lemma \ref{lem:bi-nonres2}, we can get \eqref{eq:bi-nonres4-1}.
\end{proof}

Now, we focus on the cubic non-resonant interaction component. Here cubic non-resonant interaction terms is weaker than that of the fifth-order mKdV equation due to the loss of two derivatives in the high frequency piece. Similarly as in the section 5 in \cite{Kwak2015}, we can obtain the following result without the detailed proof:
\begin{lemma}\label{lem:cubic}$\;$

(a) \emph{(High - high - high $\Rightarrow$ high)} Let $k_4 \ge 20$ and $|k_1-k_4|, |k_2-k_4|, |k_3-k_4|\le 5$. Then, we have 
\[\norm{P_{k_4}N_2(P_{k_1}u,P_{k_2}v,P_{k_3}w)}_{N_{k_4}}\lesssim 2^{-k_3/2}\norm{P_{k_1}u}_{F_{k_1}}\norm{P_{k_2}v}_{F_{k_2}}\norm{P_{k_3}w}_{F_{k_3}}.\]

(b) \emph{(High - high - low $\Rightarrow$ high)} Let $k_4 \ge 20$, $|k_2-k_4|, |k_3-k_4|\le 5$ and $k_1 \le k_4 - 10$. Then, we have 
\[\norm{P_{k_4}N_2(P_{k_1}u,P_{k_2}v,P_{k_3}w)}_{N_{k_4}}\lesssim 2^{-2k_3}2^{k_1/2}\norm{P_{k_1}u}_{F_{k_1}}\norm{P_{k_2}v}_{F_{k_2}}\norm{P_{k_3}w}_{F_{k_3}}.\]

(c) \emph{(High - high - high $\Rightarrow$ low)} Let $k_3 \ge 20$, $|k_1-k_3|, |k_2-k_3|\le 5$ and $k_4 \le k_3 -10$. Then, we have 
\[\norm{P_{k_4}N_2(P_{k_1}u,P_{k_2}v,P_{k_3}w)}_{N_{k_4}}\lesssim k_32^{-k_3}2^{-k_4/2}\norm{P_{k_1}u}_{F_{k_1}}\norm{P_{k_2}v}_{F_{k_2}}\norm{P_{k_3}w}_{F_{k_3}}.\]

(d) \emph{(High - low - low $\Rightarrow$ high)} Let $k_4 \ge 20$, $|k_3-k_4|\le 5$ and $k_1,k_2 \le k_4 -10$. Then, we have 
\[\norm{P_{k_4}N_2(P_{k_1}u,P_{k_2}v,P_{k_3}w)}_{N_{k_4}}\lesssim 2^{-2k_4}2^{k_{min}/2}\norm{P_{k_1}u}_{F_{k_1}}\norm{P_{k_2}v}_{F_{k_2}}\norm{P_{k_3}w}_{F_{k_3}}.\]

(e) \emph{(High - high - low $\Rightarrow$ low)} Let $k_3 \ge 20$, $|k_2-k_3|\le 5$ and $k_1, k_4 \le k_3 - 10$. Then, we have 
\[\norm{P_{k_4}N_2(P_{k_1}u,P_{k_2}v,P_{k_3}w)}_{N_{k_4}}\lesssim k_32^{-k_3}C(k_1,k_4)\norm{P_{k_1}u}_{F_{k_1}}\norm{P_{k_2}v}_{F_{k_2}}\norm{P_{k_3}w}_{F_{k_3}},\] 
where 
\[C(k_1,k_4)= 
\begin{cases}
2^{-3k_4/2}2^{k_1/2} \hspace{1em}  &, k_1 \le k_4 -10 \\
2^{-k_4} \hspace{1em}  &, k_4 \le k_1 - 10\\
2^{-k_4/2} \hspace{1em}  &,|k_1 -k_4| <10
\end{cases}
.\]
(f) \emph{(low - low - low $\Rightarrow$ low)} Let $0 \le k_1 ,k_2, k_3, k_4 \le 200$. Then, we have 
\[\norm{P_{k_4}N_2(P_{k_1}u,P_{k_2}v,P_{k_3}w)}_{N_{k_4}}\lesssim \norm{P_{k_1}u}_{F_{k_1}}\norm{P_{k_2}v}_{F_{k_2}}\norm{P_{k_3}w}_{F_{k_3}}.\]
\end{lemma}

As a conclusion to this section, we prove the nonlinear estimates for \eqref{eq:5kdv3} by gathering the block estimates obtained above.

\begin{proposition}\label{prop:nonlinear1}
(a) If $s > 1$, $T \in (0,1]$ and $u,v,w, \in F^s(T)$, then
\[\begin{aligned}
&\norm{N_1(u,v,w)}_{N^s(T)}+\norm{N_2(u,v,w)}_{N^s(T)} + \norm{N_3(u,v)}_{N^s(T)} + \norm{N_4(u,v)}_{N^s(T)}\\
&\hspace{20em} \lesssim \norm{u}_{F^s(T)}\norm{v}_{F^s(T)} + \norm{u}_{F^s(T)}\norm{v}_{F^s(T)}\norm{w}_{F^s(T)}.
\end{aligned}\]

(b) 
\[\begin{aligned}
&\norm{N_1(u,v,w)}_{N^0(T)}+\norm{N_2(u,v,w)}_{N^0(T)} + \norm{N_3(u,v)}_{N^0(T)} + \norm{N_4(u,v)}_{N^0(T)}\\
&\hspace{20em}\lesssim \norm{u}_{F^{1+}}\norm{v}_{F^0} + \norm{u}_{F^{\frac12+}(T)}\norm{v}_{F^{\frac12+}(T)}\norm{w}_{F^0(T)}.
\end{aligned}\]
\end{proposition} 
\begin{proof}
The proof follows from the dyadic bilinear and trilinear estimates. See \cite{Guo2012} for similar proof.
\end{proof}

\section{Energy estimates}\label{sec:energy}
In this section, we will control $\norm{u}_{E^s(T)}$ for \eqref{eq:5kdv3} by $\norm{u_0}_{H^s}$ and $\norm{u}_{F^s(T)}$. In the following section, we also assume that $|\wh{u_0}| \le 10$ in order to use
\[|G(n_1,n_2)| \gtrsim |n_1n_2(n_1+n_2)|(n_1^2 + n_2^2 + (n_1+n_2)^2)\]
in the support property \eqref{eq:support property}.

Let us define, for $k \ge 1$, $\psi(n):=n\chi'(n)$ and $\psi_k(n) = \psi(2^{-k}n)$, where $\chi$ is defined in \eqref{eq:cut-off1} and $'$ denote the derivative. Then, we have from the simple observation and the definition of $\chi_k$ that
\[\psi_k(n) = n\chi_k'(n).\]

\begin{remark}
The reason why we define another cut-off function $\psi_k$ is to use the second-order Taylor's theorem for the commutator estimates (see Lemma \ref{lem:commutator2}). But, for the other estimates, it does not need to distinguish between $\psi_k$ and $\chi_k$, since both play a role of frequency support in the other estimates.
\end{remark}

Recall \eqref{eq:5kdv3} by slightly modifying as follows:
\begin{equation}\label{eq:5kdv4}
\begin{aligned}
\pt \wh{u}(n) - i\mu(n)\wh{u}(n) &= -30in|\wh{u}(n)|^2\wh{u}(n) \\
&+10in\sum_{\N_{3,n}}\wh{u}(n_1)\wh{u}(n_2)\wh{u}(n_3) \\
&+10in\sum_{\N_{2,n}}\wh{u}(n_1)n_2^2\wh{u}(n_2) \\
&+10i\sum_{\N_{2,n}}n_1\wh{u}(n_1)n_2^2\wh{u}(n_2) \\
&=: \wh{N}_{1,1}(u) + \wh{N}_{1,2}(u) + \wh{N}_{1,3}(u) + \wh{N}_{1,4}(u),
\end{aligned}
\end{equation}
Denote the last three terms in the right-hand side of \eqref{eq:5kdv4} by $\wh{N}_1(u)(n)$. We perform the following procedure for $k \ge 1$,
\[\sum_{n}\chi_k(n)\eqref{eq:5kdv4} \times \chi_k(-n)\wh{v}(-n) + \overline{\chi_k(n)\eqref{eq:5kdv4}} \times \chi_k(n)\wh{v}(n),\]
where $\overline{\eqref{eq:5kdv4}}$ means to take the complex conjugate on \eqref{eq:5kdv4}, then we have
\begin{align*}
\pt\norm{P_ku}_{L_x^2}^2 &= - \mbox{Re}\left[20i \sum_{n,\overline{\N}_{3,n}}\chi_k(n)n \wh{u}(n_1)\wh{u}(n_2)\wh{u}(n_3)\chi_k(n)\wh{u}(n)\right] \\
&-\mbox{Re}\left[20i \sum_{n,\overline{\N}_{2,n}}\chi_k(n)n \wh{u}(n_1)n_2^2\wh{u}(n_2)\chi_k(n)\wh{u}(n)\right] \\
&-\mbox{Re}\left[20i \sum_{n,\overline{\N}_{2,n}} \chi_k(n)n_1\wh{u}(n_1)n_2^2\wh{u}(n_2)\chi_k(n)\wh{u}(n)\right]\\
&=: E_1 +E_2 + E_3,
\end{align*}
where $\overline{\N}_{2,n}= \N_{2,-n}=\set{(n_1,n_2) \in \Z^2 : n_1+n_2+n=0, nn_1n_2 \neq 0}$. 

For $k \ge 1$, let us define the new localized energy of $u$ by
\begin{equation}\label{eq:new energy2-1}
\begin{aligned}
E_k(u)(t) &= \norm{P_ku(t)}_{L_x^2}^2 + \mbox{Re}\left[\alpha \sum_{n,\overline{\N}_{2,n}}\wh{u}(n_1)\psi_k(n_2)\frac{1}{n_2}\wh{u}(n_2)\chi_k(n)\frac1n\wh{u}(n)\right]\\
&+ \mbox{Re}\left[\beta \sum_{n,\overline{\N}_{2,n}}\wh{u}(n_1)\chi_k(n_2)\frac{1}{n_2}\wh{u}(n_2)\chi_k(n)\frac1n\wh{u}(n)\right],
\end{aligned}
\end{equation}
where $\alpha$ and $\beta$ are real and will be chosen later. By gathering all localized energies, we define the new modified energy for \eqref{eq:5kdv4} by   
\begin{equation}\label{eq:new energy2-2}
E_{T}^s(u) = \norm{P_0u(0)}_{L_x^2}^2 + \sum_{k \ge 1}2^{2sk} \sup_{t_k \in [-T,T]} E_{k}(u)(t_k).
\end{equation}

The following lemma shows that $E_{T}^s(u)$ and $\norm{u}_{E^s(T)}$ are comparable.
\begin{lemma}\label{lem:comparable energy2-1}
Let $s > \frac12$. Then, there exists $0 < \delta \ll 1$ such that  
\[\frac12\norm{u}_{E^s(T)}^2 \le E_{T}^s(u) \le \frac32\norm{u}_{E^s(T)}^2,\]
for all $u \in E^s(T) \cap C([-T,T];H^s(\T))$ satisfying $\norm{u}_{L_T^{\infty}H^s(\T)} \le \delta$.
\end{lemma}
\begin{proof}
The proof follows from the Sobolev embedding $H^s(\T) \hookrightarrow L^{\infty}(\T)$, $s > 1/2$. See Lemma 5.1 in \cite{KP2015} for the details.
\end{proof}

The following lemmas are useful to estimate the modified energy.
\begin{lemma}\label{lem:energy2-1}
Let $T \in (0,1]$, $k_1,k_2,k_3 \in \Z_+$, and $u_i \in F_{k_i}(T)$, $i=1,2,3$. We further assume $k_1 \le k_2 \le k_3$ with $k_3 \ge 10$. Then

(a) For $|k_1 - k_3| \le 5 $, we have
\begin{equation}\label{eq:energy2-1.1}
\left| \sum_{n_3,\overline{\N}_{2,n_3}} \int_0^T  \wh{u}_1(n_1)\wh{u}_2(n_2)\wh{u}_3(n_3) \; dt\right| \lesssim 2^{-3k_3/2}\prod_{i=1}^{3}\norm{u_i}_{F_{k_i}(T)}.
\end{equation}

(b) For $|k_2 - k_3| \le 5 $ and $k_1 \le k_3 - 10$, we have
\begin{equation}\label{eq:energy2-1.2}
\left| \sum_{n_3,\overline{\N}_{2,n_3}} \int_0^T  \wh{u}_1(n_1)\wh{u}_2(n_2)\wh{u}_3(n_3) \; dt\right| \lesssim 2^{-k_3}2^{-k_1/2}\prod_{i=1}^{3}\norm{u_i}_{F_{k_i}(T)}.
\end{equation}
\end{lemma}

\begin{proof}
We fix extensions $\wt{u}_i \in F_{k_i}$ so that $\norm{\wt{u}_{i}}_{F_{k_i}} \le 2 \norm{u_i}_{F_{k_i}(T)}$, $i=1,2,3$. Let $\gamma : \R \to [0,1]$ be a smooth partition of unity function with $\sum_{m \in \Z}\gamma^3(x-m) \equiv 1$, $x \in \R$. Then, we obtain 
\begin{equation}\label{eq:energy2-1.3}
\begin{aligned}
&\left| \sum_{n_3,\overline{\N}_{2,n_3}} \int_0^T  \wh{u}_1(n_1)\wh{u}_2(n_2)\wh{u}_3(n_3) \; dt\right|\\
&\lesssim \sum_{|m| \lesssim 2^{2k_3}} \Big| \sum_{n_3,\overline{\N}_{2,n_3}} \int_{\R}\left(\gamma(2^{2k_3}t-m)\mathbf{1}_{[0,T]}(t)\wh{\wt{u}}_1(n_1)\right) \cdot \left(\gamma(2^{2k_3}t-m)\wh{\wt{u}}_2(n_2)\right) \cdot \left(\gamma(2^{2k_3}t-m)\wh{\wt{u}}_3(n_3)\right) \;dt \Big|
\end{aligned}
\end{equation}
Set
\[A = \set{m:\gamma(2^{2k_3}t-m)\mathbf{1}_{[0,T]}(t) \mbox{ non-zero and } \neq \gamma(2^{2k_3}t-m) }.\]
Then, the summation over $m \lesssim 2^{2k_3}$ in the right-hand side of \eqref{eq:energy2-1.3} is divided into $A$ and $A^c$. Since $|A| \le 4$, we can easily handle (see \cite{Guo2012} for the details) the right-hand side of \eqref{eq:energy2-1.3} on $B$ by showing
\[\sup_{j \in \Z_+}2^{j/2}\norm{\eta_j(\tau-\mu(n)) \cdot \ft[\mathbf{1}_{[0,1]}(t)\gamma(2^{2k_3}t-m)\wt{u}_1]}_{L_{\tau}^2\ell_n^2} \lesssim \norm{\gamma(2^{2k_3}t-m)\wt{u}_1}_{X_{k_1}}.\]
Hence, we only handle the summation on $A^c$ (for $m \in A^c$, $\gamma(2^{2k_3}t-m)\mathbf{1}_{[0,T]}(t)\wh{\wt{u}}_1(n_1) = \gamma(2^{2k_3}t-m)\wh{\wt{u}}_1(n_1)$). Let $f_{k_i} = \ft[\gamma(2^{2k_3}t-m)\wh{\wt{u}}_i(n_i)]$ and $f_{k_i,j_i} = \eta_{j_i}(\tau - \mu(n))f_{k_i}$, $i=1,2,3$. By Parseval's identity and \eqref{eq:prop1}, the right-hand side of \eqref{eq:energy2-1.3} is dominated by
\[\sup_{m \in B^c} 2^{2k_3} \sum_{j_1,j_2,j_3 \ge 2k_3} |J(f_{k_1,j_1},f_{k_2,j_2},f_{k_3,j_3})|.\]

(a) By the support property \eqref{eq:support property}, we know $j_{max} \ge 5k_3$. Then, since the case when $j_{med} \le 3k_3$ is the worst case, we use \eqref{eq:bi-block estimate-a1.1} to estimate $|J(f_{k_1,j_1},f_{k_2,j_2},f_{k_3,j_3})|$, then
\begin{align*}
\eqref{eq:energy2-1.3} &\lesssim 2^{2k_3} \sum_{\substack{j_1,j_2,j_3 \ge 2k_3\\j_{med} \le 3k_3}}2^{(j_1+j_2+j_3)/2}2^{-(j_{max}+j_{med})/2}\prod_{i=1}^{3} \norm{f_{k_i,j_i}}_{L_{\tau}^2\ell_n^2}\\
&\lesssim 2^{2k_3} \sum_{j_1,j_2,j_3 \ge 2k_3}2^{(j_1+j_2+j_3)/2}2^{-7k_3/2}\prod_{i=1}^{3} \norm{f_{k_i,j_i}}_{L_{\tau}^2\ell_n^2}\\
&\lesssim 2^{-3k_3/2}\norm{u_1}_{F_{k_1}(T)}\norm{u_2}_{F_{k_2}(T)}\norm{u_3}_{F_{k_3}(T)}.
\end{align*}

(b) Since the case when $j_{med} \le 3k_3 + k_1$ is also the worst case, we use \eqref{eq:bi-block estimate-b1.1} and argument in (a) with $j_{max} \ge 4k_3 + k_1$, then 
\begin{align*}
\eqref{eq:energy2-1.3} &\lesssim 2^{2k_3} \sum_{\substack{j_1,j_2,j_3 \ge 2k_4\\j_{med} \le 3k_3 + k_1}}2^{(j_1+j_2+j_3)/2}2^{-(j_{max}+j_{med})/2}\prod_{i=1}^{3} \norm{f_{k_i,j_i}}_{L_{\tau}^2\ell_n^2}\\
&\lesssim 2^{2k_3} \sum_{j_1,j_2,j_3 \ge 2k_3}2^{(j_1+j_2+j_3)/2}2^{-3k_3}2^{-k_1/2}\prod_{i=1}^{3} \norm{f_{k_i,j_i}}_{L_{\tau}^2\ell_n^2}\\
&\lesssim 2^{-k_3}2^{-k_1/2}\norm{u_1}_{F_{k_1}(T)}\norm{u_2}_{F_{k_2}(T)}\norm{u_3}_{F_{k_3}(T)}.
\end{align*}
Therefore, we finish the proof of Lemma \ref{lem:energy2-1}.
\end{proof}
In order to estimate the cubic terms, we state the following lemma: 
\begin{lemma}\label{lem:energy-cubic}
Let $T \in (0,1]$, $k_1,k_2,k_3,k_4 \in \Z_+$, and $v_i \in F_{k_i}(T)$, $i=1,2,3,4$. We further assume $k_1 \le k_2 \le k_3 \le k_4$ with $k_4 \ge 10$. Then

(a) For $|k_1 - k_4| \le 5 $, we have
\begin{equation}\label{eq:energy-cubic1}
\left| \sum_{n_4,\overline{\N}_{3,n_4}} \int_0^T  \wh{v}_1(n_1)\wh{v}_2(n_2)\wh{v}_3(n_3)\wh{v}_4(n_4) \; dt\right| \lesssim 2^{k_4/2}\prod_{i=1}^{4}\norm{v_i}_{F_{k_i}(T)}.
\end{equation}

(b) For $|k_2 - k_4| \le 5 $ and $k_1 \le k_4 - 10$, we have
\begin{equation}\label{eq:energy-cubic2}
\left| \sum_{n_4,\overline{\N}_{3,n_4}} \int_0^T \wh{v}_1(n_1)\wh{v}_2(n_2)\wh{v}_3(n_3)\wh{v}_4(n_4) \; dt\right| \lesssim 2^{-k_4}2^{k_1/2}\prod_{i=1}^{4}\norm{v_i}_{F_{k_i}(T)}.
\end{equation}

(c) For $|k_3 - k_4| \le 5$, $k_2 \le k_4 -10$ and $|k_1-k_2| \le 5$, we have
\begin{equation}\label{eq:energy-cubic3}
\left| \sum_{n_4,\overline{\N}_{3,n_4}} \int_0^T \wh{v}_1(n_1)\wh{v}_2(n_2)\wh{v}_3(n_3)\wh{v}_4(n_4) \; dt\right| \lesssim 2^{-k_4}2^{k_1/2}\prod_{i=1}^{4}\norm{v_i}_{F_{k_i}(T)}.
\end{equation}

(d) For $|k_3 - k_4| \le 5$, $k_2 \le k_4 -10$ and $k_1 \le k_2 - 10$, we have
\begin{equation}\label{eq:energy-cubic4}
\left| \sum_{n_4,\overline{\N}_{3,n_4}} \int_0^T\wh{v}_1(n_1)\wh{v}_2(n_2)\wh{v}_3(n_3)\wh{v}_4(n_4) \; dt\right| \lesssim 2^{-k_4}\prod_{i=1}^{4}\norm{v_i}_{F_{k_i}(T)}.
\end{equation}
\end{lemma}

\begin{proof}
See \cite{Kwak2015} for the proof.
\end{proof}

The next lemma is a kind of commutator estimate which will be helpful to handle bad terms $\int_0^T E_{2}$ and $\int_0^T E_{3}$ in the original energy.

\begin{lemma}\label{lem:commutator2}
Let $T \in (0,1]$, $k,k_1\in \Z_+$ satisfying $k_1 \le k -10$, $v \in F_{k_1}(T)$ and $u \in F^0(T)$. Then, we have
\begin{equation}\label{eq:commutator2-1}
\begin{aligned}
\Big|\sum_{n,\overline{\N}_{2,n}}&\int_0^T \chi_k(n)n[\chi_{k_1}(n_1)\wh{v}(n_1)n_2^2\wh{u}(n_2)]\chi_k(n)\wh{u}(n) \;dt\\
&+ \frac12\sum_{n,\overline{\N}_{2,n}}\int_0^T \chi_{k_1}(n_1)n_1\wh{v}(n_1)\chi_k(n_2)n_2\wh{u}(n_2)\chi_k(n)n\wh{u}(n) \;dt\\
&- \sum_{n,\overline{\N}_{2,n}}\int_0^T \chi_{k_1}(n_1)n_1\wh{v}(n_1)\psi_k(n_2)n_2\wh{u}(n_2)\chi_k(n)n\wh{u}(n) \;dt \Big|\\
&\lesssim 2^{3k_1/2} \norm{P_{k_1}v}_{F_{k_1}(T)}\sum_{|k-k'|\le 5} \norm{P_{k'}u}_{F_{k'}(T)}^2,
\end{aligned}
\end{equation}
and
\begin{equation}\label{eq:commutator2-2}
\begin{aligned}
\Big|\sum_{n,\overline{\N}_{2,n}}&\int_0^T \chi_k(n)[\chi_{k_1}(n_1)n_1\wh{v}(n_1)n_2^2\wh{u}(n_2)]\chi_k(n)\wh{u}(n)\; dt\\
&+ \sum_{n,\overline{\N}_{2,n}}\int_0^T \chi_{k_1}(n_1)n_1\wh{v}(n_1)\chi_k(n_2)n_2\wh{u}(n_2)\chi_k(n)n\wh{u}(n)\;dt \Big|\\
&\lesssim 2^{3k_1/2} \norm{P_{k_1}v}_{F_{k_1}(T)}\sum_{|k-k'|\le 5} \norm{P_{k'}u}_{F_{k'}(T)}^2,
\end{aligned}
\end{equation}
\end{lemma}

\begin{proof}
We first consider \eqref{eq:commutator2-1}. From $n_1+n_2+n= 0$ and the symmetry of $n_2,n$, we have
\begin{align*}
\mbox{LHS of }\eqref{eq:commutator2-1} &= \Big|\sum_{n,\overline{\N}_{2,n}}\int_0^T[\chi_k(n)n_2^2 - \chi_k(n_2)n_2^2 - n_1n_2\psi_k(n_2)]\\
&\hspace{11em}\times\chi_{k_1}(n_1)\wh{v}(n_1)\wh{u}(n_2)\chi_k(n)n\wh{u}(n)\;dt \Big|\\
&= \Big|\sum_{n,\overline{\N}_{2,n}}\int_0^T\left[\frac{\chi_k(n) - \chi_k(n_2) - n_1\chi_k'(n_2)}{n_1^2}\cdot n_2^2\right]\\
&\hspace{11em}\times \chi_{k_1}(n_1)n_1^2\wh{v}(n_1)\wh{u}(n_2)\chi_k(n)n\wh{u}(n)\;dt \Big|.
\end{align*}
Since both $\chi_k$ and $\chi_k'$ are even functions, $-n_2 = n + n_1$, $|n|\sim|n_2|$ and $\chi_k''(n) = O(\chi_k(n)/n^2)$ due to \eqref{eq:regularity}, we know from the Taylor's theorem that 
\[\left|\frac{\chi_k(n) - \chi_k(n_2) - n_1\chi_k'(n_2)}{n_1^2}\cdot n_2^2\right| \lesssim 1.\]
Hence by the same way as in the proof of Lemma \ref{lem:energy2-1} (b), we have
\[\mbox{LHS of }\eqref{eq:commutator2-1} \lesssim 2^{3k_1/2}\norm{P_{k_1}v}_{F_{k_1}(T)}\sum_{|k-k'|\le 5} \norm{P_{k'}u}_{F_{k'}(T)}^2.\]

Next, we consider \eqref{eq:commutator2-2}. Since $n = -n_2 - n_1$, we have 
\[\begin{aligned}
&\sum_{n,\overline{\N}_{2,n}}\int_0^T n_1\chi_{k_1}(n_1)\wh{v}(n_1)\chi_k(n_2)n_2\wh{u}(n_2)\chi_k(n)n\wh{u}(n)\;dt\\
=&-\sum_{n,\overline{\N}_{2,n}}\int_0^T n_1^2\chi_{k_1}(n_1)\wh{v}(n_1)\chi_k(n_2)n_2\wh{u}(n_2)\chi_k(n)\wh{u}(n)\;dt\\
&-\sum_{n,\overline{\N}_{2,n}}\int_0^T n_1\chi_{k_1}(n_1)\wh{v}(n_1)\chi_k(n_2)n_2^2\wh{u}(n_2)\chi_k(n)\wh{u}(n)\;dt,
\end{aligned}\]
and similarly as before, we have 
\begin{align*}
&\sum_{n,\overline{\N}_{2,n}}\int_0^T \chi_k(n)[\chi_{k_1}(n_1)n_1\wh{v}(n_1)n_2^2\wh{u}(n_2)]\chi_k(n)\wh{u}(n)\\
&- \sum_{n,\overline{\N}_{2,n}}\int_0^T \chi_{k_1}(n_1)n_1\wh{v}(n_1)\chi_k(n_2)n_2^2\wh{u}(n_2)]\chi_k(n)\wh{u}(n)\\
&= \sum_{n,\overline{\N}_{2,n}}\int_0^T\left[\frac{\chi_k(n) - \chi_k(n_2)}{n_1}\cdot n_2\right]\\
&\hspace{11em}\times \chi_{k_1}(n_1)n_1^2\wh{v}(n_1)n_2\wh{u}(n_2)\chi_k(n)\wh{u}(n)\;dt,
\end{align*}
with
\[\left|\frac{\chi_k(n) - \chi_k(n_2)}{n_1}\cdot n_2\right| \lesssim 1.\]
Again we use \eqref{eq:energy2-1.2} so that
\[\mbox{LHS of }\eqref{eq:commutator2-2} \lesssim 2^{3k_1/2}\norm{P_{k_1}v}_{F_{k_1}(T)}\sum_{|k-k'|\le 5} \norm{P_{k'}u}_{F_{k'}(T)}^2,\]
which completes the proof of Lemma \ref{lem:commutator2}.
\end{proof}

Using above lemmas, we show the energy estimate. 
\begin{proposition}\label{prop:energy2-2}
Let $s \ge 2$ and $T \in (0,1]$, Then, for the solution $u \in C([-T,T];H^{\infty}(\T))$ to \eqref{eq:5kdv4}, we have
\[E_{T}^s(u) \lesssim (1+ \norm{u_0}_{H^s})\norm{u_0}_{H^s}^2 + \left(\norm{u}_{F^{\frac32+}(T)}+\norm{u}_{F^{2}(T)}^2+\norm{u}_{F^{\frac12+}(T)}^3\right)\norm{u}_{F^{s}(T)}^2.\]
\end{proposition}

\begin{proof}
For any $k\in \Z_+$ and $t \in [-T,T]$, recall the localized modified energy \eqref{eq:new energy2-1}
\[\begin{aligned}
E_k(u)(t) &= \norm{P_ku(t)}_{L_x^2}^2 + \mbox{Re}\left[\alpha \sum_{n,\overline{\N}_{2,n}}\wh{u}(n_1)\psi_k(n_2)\frac{1}{n_2}\wh{u}(n_2)\chi_k(n)\frac1n\wh{u}(n)\right]\\
&+ \mbox{Re}\left[\beta \sum_{n,\overline{\N}_{2,n}}\wh{u}(n_1)\chi_k(n_2)\frac{1}{n_2}\wh{u}(n_2)\chi_k(n)\frac1n\wh{u}(n)\right]\\
&=:I(t) + II(t) + III(t)
\end{aligned}\]
and 
\begin{align*}
\pt\norm{P_ku}_{L_x^2}^2 &= - \mbox{Re}\left[20i \sum_{n,\overline{\N}_{3,n}}\chi_k(n)n \wh{u}(n_1)\wh{u}(n_2)\wh{u}(n_3)\chi_k(n)\wh{u}(n)\right] \\
&-\mbox{Re}\left[20i \sum_{n,\overline{\N}_{2,n}}\chi_k(n)n \wh{u}(n_1)n_2^2\wh{u}(n_2)\chi_k(n)\wh{u}(n)\right] \\
&-\mbox{Re}\left[10i \sum_{n,\overline{\N}_{2,n}} \chi_k(n)nn_1\wh{u}(n_1)n_2\wh{u}(n_2)\chi_k(n)\wh{u}(n)\right]\\
&=: E_1.
\end{align*}
We differentiate $II(t)$ with respect to $t$, respectively. Then, we have
\[\begin{aligned}
\frac{d}{dt} II(t) &= \mbox{Re}\Big[\alpha i \sum_{n,\overline{\N}_{2,n}}(\mu_2(n_1)+\mu_2(n_2)+\mu_2(n))\wh{u}(n_1)\psi_k(n_2)\frac{1}{n_2}\wh{u}(n_2)\chi_k(n)\frac1n\wh{u}(n)\Big]\\
&+ \mbox{Re}\Big[\alpha \sum_{n,\overline{\N}_{2,n}}\wh{N}_2(u)(n_1)\psi_k(n_2)\frac{1}{n_2}\wh{u}(n_2)\chi_k(n)\frac1n\wh{u}(n)+\wh{u}(n_1)\psi_k(n_2)\frac{1}{n_2}\wh{N}_2(v)(n_2)\chi_k(n)\frac1n\wh{u}(n)\\
&\hspace{3em}+\wh{u}(n_1)\psi_k(n_2)\frac{1}{n_2}\wh{u}(n_2)\chi_k(n)\frac1n\wh{N}_2(u)(n)\Big]\\
&+ \mbox{Re}\Big[30\alpha i \sum_{n,\overline{\N}_{2,n}}n|\wh{u}(n_1)|^2\wh{u}(n_1)\psi_k(n_2)\frac{1}{n_2}\wh{u}(n_2)\chi_k(n)\frac1n\wh{u}(n)\\
&\hspace{3em}+\wh{u}(n_1)\psi_k(n_2)|\wh{u}(n_2)|^2\wh{u}(n_2)\chi_k(n)\frac1n\wh{u}(n)+\wh{u}(n_1)\psi_k(n_2)\frac{1}{n_2}\wh{u}(n_2)\chi_k(n)|\wh{u}(n)|^2\wh{u}(n)\Big].
\end{aligned}\]

We use the following algebraic laws
\[(a+b)^5 = a^5 + 5(a^4b + ab^4) + 10(a^3b^2 + a^2b^3) + b^5\]
and
\[(a+b)^3 = a^3+b^3+3(a^2b + ab^2)\]
so that we obtain
\[\frac{d}{dt} II(t) = E_{2,1} + E_{2,2} + E_{2,3} +E_{2,4}=: E_2,\]
where
\[\begin{aligned}
E_{2,1} = \mbox{Re}\Big[\alpha i \sum_{n,\overline{\N}_{2,n}}5(n_1^3n_2n - n_1n_2^2n_3^2)\wh{u}(n_1)\psi_k(n_2)\frac{1}{n_2}\wh{u}(n_2)\chi_k(n)\frac1n\wh{u}(n)\Big],
\end{aligned}\]

\[\begin{aligned}
E_{2,2} = \mbox{Re}\Big[\wt{c}_1\alpha i \sum_{n,\overline{\N}_{2,n}}3n_1n_2n \wh{u}(n_1)\psi_k(n_2)\frac{1}{n_2}\wh{u}(n_2)\chi_k(n)\frac1n\wh{u}(n)\Big],
\end{aligned}\]

\[\begin{aligned}
E_{2,3} &= \mbox{Re}\Big[\alpha \sum_{n,\overline{\N}_{2,n}}\Big\{\wh{N}_2(u)(n_1)\psi_k(n_2)\frac{1}{n_2}\wh{u}(n_2)\chi_k(n)\frac1n\wh{u}(n)\\
&\hspace{3em}+\wh{u}(n_1)\psi_k(n_2)\frac{1}{n_2}\wh{N}_2(u)(n_2)\chi_k(n)\frac1n\wh{u}(n)+\wh{u}(n_1)\psi_k(n_2)\frac{1}{n_2}\wh{u}(n_2)\chi_k(n)\frac1n\wh{N}_2(u)(n)\Big\}\Big]
\end{aligned}\]
and
\[\begin{aligned}
E_{2,4} &= \mbox{Re}\Big[30\alpha i \sum_{n,\overline{\N}_{2,n}}\Big\{n_1|\wh{u}(n_1)|^2\wh{u}(n_1)\psi_k(n_2)\frac{1}{n_2}\wh{u}(n_2)\chi_k(n)\frac1n\wh{u}(n)\\
&\hspace{3em}+\wh{u}(n_1)\psi_k(n_2)|\wh{u}(n_2)|^2\wh{u}(n_2)\chi_k(n)\frac1n\wh{u}(n) + \wh{u}(n_1)\psi_k(n_2)\frac{1}{n_2}\wh{u}(n_2)\chi_k(n)|\wh{u}(n)|^2\wh{u}(n)\Big\}\Big].
\end{aligned}\]
Similarly, we get
\[\frac{d}{dt} III(t) = E_{3,1} + E_{3,2} + E_{3,3} +E_{3,4}=: E_3,\]
where
\[\begin{aligned}
E_{3,1} = \mbox{Re}\Big[\beta i \sum_{n,\overline{\N}_{2,n}}5(n_1^3n_2n - n_1n_2^2n_3^2)\wh{u}(n_1)\chi_k(n_2)\frac{1}{n_2}\wh{u}(n_2)\chi_k(n)\frac1n\wh{u}(n)\Big],
\end{aligned}\]

\[\begin{aligned}
E_{3,2} = \mbox{Re}\Big[\wt{c}_1\beta i \sum_{n,\overline{\N}_{2,n}}3n_1n_2n \wh{u}(n_1)\chi_k(n_2)\frac{1}{n_2}\wh{u}(n_2)\chi_k(n)\frac1n\wh{u}(n)\Big],
\end{aligned}\]

\[\begin{aligned}
E_{3,3} &= \mbox{Re}\Big[\beta \sum_{n,\overline{\N}_{2,n}}\Big\{\wh{N}_2(u)(n_1)\chi_k(n_2)\frac{1}{n_2}\wh{u}(n_2)\chi_k(n)\frac1n\wh{u}(n)\\
&\hspace{3em}+\wh{u}(n_1)\chi_k(n_2)\frac{1}{n_2}\wh{N}_2(u)(n_2)\chi_k(n)\frac1n\wh{u}(n) +\wh{u}(n_1)\chi_k(n_2)\frac{1}{n_2}\wh{u}(n_2)\chi_k(n)\frac1n\wh{N}_2(u)(n)\Big\}\Big]
\end{aligned}\]
and
\[\begin{aligned}
E_{3,4} &= \mbox{Re}\Big[30\beta i \sum_{n,\overline{\N}_{2,n}}\Big\{n_1|\wh{u}(n_1)|^2\wh{u}(n_1)\chi_k(n_2)\frac{1}{n_2}\wh{u}(n_2)\chi_k(n)\frac1n\wh{u}(n)\\
&\hspace{3em}+\wh{u}(n_1)\chi_k(n_2)|\wh{u}(n_2)|^2\wh{u}(n_2)\chi_k(n)\frac1n\wh{u}(n) + \wh{u}(n_1)\chi_k(n_2)\frac{1}{n_2}\wh{u}(n_2)\chi_k(n)|\wh{u}(n)|^2\wh{u}(n)\Big\}\Big].
\end{aligned}\]

Fix $t_k \in [0,T]$, by integrating $\pt E_k(u)(t)$ with respect to $t$ from $0$ to $t_k$, then we have
\begin{equation}\label{eq:energy2-2.2}
E_k(u)(t_k) - E_k(u)(0) \le \left|\int_0^{t_k} E_1 + E_2 + E_3 \; dt \right|.
\end{equation}

We estimate the right-hand side of \eqref{eq:energy2-2.2} by dividing it into several cases. First, we choose $\alpha = -4$ and $\beta = 6$ to use Lemma \ref{lem:commutator2}, then for each $k\ge 1$, we have 
\[\left|\int_0^{t_k} E_1 + E_{2,1} + E_{3,1} \; dt \right| \lesssim \sum_{i=1}^{7}B_i(k),\]
where
\[\begin{aligned}
B_1(k) = \sum_{0\le k_1\le k-10}\Big|&\sum_{n,\overline{\N}_{2,n}}\int_0^{t_k} \chi_k(n)n[\chi_{k_1}(n_1)\wh{u}(n_1)n_2^2\wh{u}(n_2)]\chi_k(n)\wh{u}(n) \;dt\\
&+ \frac12\sum_{n,\overline{\N}_{2,n}}\int_0^{t_k} \chi_{k_1}(n_1)n_1\wh{u}(n_1)\chi_k(n_2)n_2\wh{u}(n_2)\chi_k(n)n\wh{u}(n) \;dt\\
&- \sum_{n,\overline{\N}_{2,n}}\int_0^{t_k} \chi_{k_1}(n_1)n_1\wh{u}(n_1)\psi_k(n_2)n_2\wh{u}(n_2)\chi_k(n)n\wh{u}(n) \;dt \Big|,
\end{aligned}\]
\[\begin{aligned}
B_2(k) =\sum_{0\le  k_1 \le k-10}\Big|&\sum_{n,\overline{\N}_{2,n}}\int_0^{t_k} \chi_k(n)[\chi_{k_1}(n_1)n_1\wh{u}(n_1)n_2^2\wh{u}(n_2)]\chi_k(n)\wh{u}(n)\; dt\\
&+ \sum_{n,\overline{\N}_{2,n}}\int_0^{t_k} \chi_{k_1}(n_1)n_1\wh{u}(n_1)\chi_k(n_2)n_2\wh{u}(n_2)\chi_k(n)n\wh{u}(n)\;dt \Big|,
\end{aligned}\]
\[B_3(k) =\sum_{\substack{k_1 \ge k -9\\k_2 \ge 0}}\Big|\sum_{n,\overline{\N}_{2,n}}\int_0^{t_k} \chi_{k_1}(n_1)\wh{u}(n_1)\chi_{k_2}(n_2)n_2^2\wh{u}(n_2)\chi_k^2(n)n\wh{u}(n)\; dt\Big|,\]
\[B_4(k) =\sum_{\substack{k_1 \ge k -9\\k_2 \ge 0}}\Big|\sum_{n,\overline{\N}_{2,n}}\int_0^{t_k} \chi_{k_1}(n_1)n_1\wh{u}(n_1)\chi_{k_2}(n_2)n_2^2\wh{u}(n_2)\chi_k^2(n)\wh{u}(n)\; dt\Big|,\]
\[B_5(k) =\sum_{|k-k_1|\le 5}\Big|\sum_{n,\overline{\N}_{2,n}}\int_0^{t_k} \chi_{k_1}(n_1)n_1\wh{u}(n_1)(\chi_k(n_2) + \psi_k(n_3))n_2\wh{u}(n_2)\chi_k(n)n\wh{u}(n) \;dt\Big|,\]
\[B_6(k) =\sum_{k_1\ge 0}\Big|\sum_{n,\overline{\N}_{2,n}}\int_0^{t_k} n_1^3n_2n\chi_{k_1}(n_1)\wh{u}(n_1)(\chi_k(n_2)+\psi_k(n_3))\frac{1}{n_2}\wh{u}(n_2)\chi_k(n)\frac1n\wh{u}(n) \;dt\Big|\]
and
\[B_7(k) =\sum_{k_1,k_2,k_3 \ge 0}\Big|\sum_{n,\overline{\N}_{3,n}}\int_0^{t_k} \chi_{k_1}(n_1)\wh{u}(n_1)\chi_{k_2}(n_2)\wh{u}(n_2)\chi_{k_3}(n_3)\wh{u}(n_3)\chi_k^2(n)n\wh{u}(n) \;dt\Big|.\]
By using Lemma \ref{lem:commutator2} and the Cauchy-Schwarz inequality, we have
\[\begin{aligned}
B_1(k) + B_2(k) &\lesssim \sum_{0\le  k_1  \le k-10} 2^{3k_1/2} \norm{P_{k_1}u}_{F_{k_1}(T)}\sum_{|k-k'|\le 3} \norm{P_{k'}u}_{F_{k'}(T)}^2\\
&\lesssim \norm{u}_{F^{\frac32+}(T)}\sum_{|k-k'|\le 5} \norm{P_{k'}u}_{F_{k'}(T)}^2.
\end{aligned}\]
For $B_3(k)$ and $B_4(k)$, we divide the summation over $k_1 \ge k -9,k_2 \ge 0$ into
\[\sum_{\substack{|k_1-k| \le 5 \\ |k_2 - k| \le 5}}+\sum_{\substack{k_2 \le k- 10 \\ |k_1-k| \le 5}}+\sum_{\substack{k_1 \ge k+10 \\ |k_1-k_2| \le 5}}.\]
We restrict $B_3(k)$ and $B_4(k)$ to the first summation, we have from \eqref{eq:energy2-1.1} and the Cauchy-Schwarz inequality that
\[\begin{aligned}
\sum_{|k-k'|\le 5} 2^{3k/2} \norm{P_{k'}u}_{F_{k'}(T)}^3 \lesssim \norm{u}_{F^{\frac32}(T)}\sum_{|k-k'|\le 5} \norm{P_{k'}u}_{F_{k'}(T)}^2.
\end{aligned}\]
For the restriction to the second and the third summations, we have from \eqref{eq:energy2-1.2} and the Cauchy-Schwarz inequality that 
\[\begin{aligned}
&\sum_{k_2 \le k- 10}2^{3k_2/2}\norm{P_{k_2}u}_{F_{k_2}(T)}\sum_{|k-k'|\le 5} \norm{P_{k'}u}_{F_{k'}(T)}^2 + 2^{k/2}\norm{P_{k}u}_{F_k(T)} \sum_{\substack{k_1 \ge k+10 \\ |k_1-k'| \le 5}}2^{k_1}\norm{P_{k'}u}_{F_{k'}(T)}^2\\
&\lesssim \norm{u}_{F^{\frac32+}(T)}\sum_{|k-k'|\le 5} \norm{P_{k'}u}_{F_{k'}(T)}^2 + 2^{-(s+\varepsilon)k}\norm{P_{k}u}_{F_k(T)} \norm{u}_{F^{\frac32+}(T)}\norm{u}_{F^s(T)},
\end{aligned}\]
for $s \ge 0$ and $0 < \varepsilon \ll 1$. Hence, we obtain
\[B_3(k) + B_4(k) \lesssim \norm{u}_{F^{\frac32+}(T)}\left(\sum_{|k-k'|\le 5}\norm{P_{k'}u}_{F_{k'}(T)}^2  +\norm{u}_{F^{s}(T)}2^{-sk-\varepsilon k}\norm{P_{k}u}_{F_k(T)}\right).\]

For $B_5(k)$, similarly as the estimate of $B_3(k) + B_4(k)$ over the first summation, we obtain
\[B_5(k) \lesssim \norm{u}_{F^{\frac32}(T)}\sum_{|k-k'|\le 5} \norm{P_{k'}u}_{F_{k'}(T)}^2.\] 
For $B_6(k)$, we use \eqref{eq:energy2-1.1}, \eqref{eq:energy2-1.1} and the Cauchy-Schwarz inequality to obtain
\[\begin{aligned}
B_6(k) &\lesssim \sum_{k_1 \le k- 10} 2^{3k_1/2}\norm{P_{k_1}u}_{F_{1,k_1}(T)} \sum_{|k-k'|\le 5}\norm{P_{k'}u}_{F_{k'}(T)}^2 + \sum_{|k-k'|\le 5}2^{3k/2}\norm{P_{k'}u}_{F_{k'}(T)}^3\\
&\lesssim \norm{u}_{F^{\frac32+}(T)}\sum_{|k-k'|\le 5}\norm{P_{k'}u}_{F_{k'}(T)}^2.
\end{aligned}\]
For $B_7(k)$, without loss of generality, we assume that $k_1 \le k_2 \le k_3$. We first consider the case when $k \sim k_3$. Then from Lemma \ref{lem:energy-cubic},  $B_7(k)$ restricted to $k \sim k_3$ is bounded by
\[\begin{aligned}
&\sum_{|k-k'|\le5} 2^{3k/2}\norm{P_{k'}u}_{F_{k'}(T)}^4 + \sum_{k_1 \le k-10}2^{k_1/2}\norm{P_{k_1}u}_{F_{k_1}(T)}\sum_{|k-k'|\le5} \norm{P_{k'}u}_{F_{k'}(T)}^3\\
&\hspace{3em}+\sum_{\substack{k_2 \le k-10\\|k_1-k_2|\le 5}}2^{k_2/2}\norm{P_{k_1}u}_{F_{k_1}(T)}^2\sum_{|k-k'|\le5}\norm{P_{k'}u}_{F_{k'}(T)}^2\\ 
&\hspace{3em}+\sum_{\substack{k_2 \le k-10\\k_1 \le k_2-10}}\norm{P_{k_1}u}_{F_{k_1}(T)}\norm{P_{k_2}u}_{F_{k_2}(T)}\sum_{|k-k'|\le5} \norm{P_{k'}u}_{F_{k'}(T)}^2\\
&\lesssim \norm{u}_{F^{\frac34}(T)}^2\sum_{|k-k'|\le5} \norm{P_{k'}u}_{F_{k'}(T)}^2.
\end{aligned}\]
Otherwise, by using Lemma \ref{lem:energy-cubic} (c) and (d), we have
\[\begin{aligned}
&2^{3k/2}\norm{P_{k}u}_{F_k(T)}\sum_{\substack{k_3 \ge k+ 10\\|k_3-k'|\le5}}2^{-k_3} \norm{P_{k'}u}_{F_{k'}(T)}^3\\
&\hspace{3em}+\sum_{|k-k'|\le 5}2^{3k/2}\norm{P_{k'}u}_{F_{k'}(T)}^2\sum_{\substack{k_3 \ge k + 10 \\ |k_2 -k_3|\le5}}2^{-k_3}\norm{P_{k_2}u}_{F_{k_2}(T)}^2\\ 
&\hspace{3em}+2^k\norm{P_{k}u}_{F_k(T)}\sum_{k_1 \le k - 10}2^{k_1/2}\norm{P_{k_1}u}_{F_{k_1}(T)}\sum_{\substack{k_3 \ge k + 10 \\ |k_2 -k_3|\le5}}2^{-k_3} \norm{P_{k_2}u}_{F_{k_2}(T)}^2\\
&\hspace{3em}+2^{3k/2}\norm{P_{k}u}_{F_k(T)}\sum_{k+10 \le k_1 \le k_3-10}\norm{P_{k_1}u}_{F_{k_1}(T)}\sum_{\substack{k_3 \ge k_1 + 10 \\ |k_2 -k_3|\le5}}2^{-k_3} \norm{P_{k_2}u}_{F_{k_2}(T)}^2\\
&\lesssim \norm{u}_{F^{\frac14+}(T)}^2\norm{u}_{F^s(T)}2^{-(s+\varepsilon)k}\norm{P_ku}_{F_k(T)} + \norm{u}_{F^{\frac14}(T)}^2\sum_{|k-k'|\le 5}\norm{P_{k'}u}_{F_{k'}(T)}^2,
\end{aligned}\]
for $s \ge 0$ and $0 < \varepsilon \ll 1$. Hence, we get the bound of $B_7(k)$ as
\[B_7(k) \lesssim \norm{u}_{F^{\frac34}(T)}^2\sum_{|k-k'|\le 5}\norm{P_{k'}u}_{F_{k'}(T)}^2 +\norm{u}_{F^{\frac14+}(T)}^2\norm{u}_{F^s(T)}2^{-(s+\varepsilon)k}\norm{P_ku}_{F_k(T)}.\]
Together with all bounds of $B_i(k)$, we obtain
\begin{equation}\label{eq:quad bound1}
\sum_{k \ge 1}2^{2sk} \sup_{t_k \in [0,T]}\left|\int_0^{t_k} E_1 + E_{2,1} + E_{3,1} \; dt \right| \lesssim \left(\norm{u}_{F^{\frac32+}(T)}+\norm{u}_{F^{\frac34}(T)}^2\right)\norm{u}_{F^s(T)}^2.
\end{equation}
Next, for $E_{2,2}$ and $E_{3,2}$ terms, since the total number of derivatives is less than that in $E_{2,1}$ and $E_{3,1}$ terms, we can easily control those terms and obtain
\[\left|\int_0^{t_k} E_{2,2} + E_{3,2} \; dt \right| \lesssim \norm{u}_{F^0(T)}\sum_{|k-k'|\le5}\norm{P_{k'}u}_{F_{k'}(T)}^2,\]
which implies
\begin{equation}\label{eq:quad bound2}
\sum_{k \ge 1}2^{2sk} \sup_{t_k \in [0,T]}\left|\int_0^{t_k} E_{2,2} + E_{3,2} \; dt \right| \lesssim \norm{u}_{F^0(T)}\norm{u}_{F^s(T)}^2,
\end{equation}
For 
\begin{equation}\label{eq:energy2-2.3}
\left|\int_0^{t_k} E_{2,4} + E_{3,4} \; dt \right|,
\end{equation}
it is enough to consider
\begin{equation}\label{eq:energy2-2.4}
\sum_{k_1\ge0}\left|\int_0^{t_k}\sum_{n,\overline{\N}_{2,n}}\chi_{k_1}(n_1)n_1|\wh{u}(n_1)|^2\wh{u}(n_1)\chi_k(n_2)\frac{1}{n_2}\wh{u}(n_2)\chi_k(n)\frac1n\wh{u}(n)\; dt\right|
\end{equation}
and
\begin{equation}\label{eq:energy2-2.5}
\sum_{k_1\ge0}\left|\int_0^{t_k}\sum_{n,\overline{\N}_{2,n}}\chi_{k_1}(n_1)\wh{u}(n_1)\chi_k(n_2)\frac{1}{n_2}\wh{u}(n_2)\chi_k(n)|\wh{u}(n)|^2\wh{u}(n)\; dt\right|,
\end{equation}
due to the symmetry of $n_2$ and $n$ variables. Since we only consider the cases when $k_1 \le k- 10$ and $|k-k_1|\le 5$, both \eqref{eq:energy2-2.4} and \eqref{eq:energy2-2.5} are reduced to
\[\norm{u}_{L_{t_k}^{\infty}L_x^2}^2\sum_{k_1 \ge 0} 2^{-k}\left|\int_0^{t_k}\sum_{n,\overline{\N}_{2,n}}\chi_{k_1}(n_1)\wh{u}(n_1)\chi_k(n_2)\wh{u}(n_2)\chi_k(n)\wh{u}(n)\; dt\right|.\]
By Lemma \ref{lem:energy2-1} and $F^0(T) \hookrightarrow C_TL^2$ \eqref{eq:small data1.1}, we obtain that
\begin{equation}\label{eq:quad bound3}
\sum_{k \ge 1}2^{2sk} \sup_{t_k \in [0,T]}\eqref{eq:energy2-2.3} \lesssim \norm{u}_{F^0(T)}^3\norm{u}_{F^s(T)}^2.
\end{equation}
Lastly, we estimate cubic and quartic terms as
\begin{equation}\label{eq:energy2-2.6}
\left|\int_0^{t_k} E_{2,3} + E_{3,3} \; dt \right|.
\end{equation}
\begin{remark}\label{rem:resonant3}
In order to control \eqref{eq:energy2-2.6}, we need to check carefully the cubic resonant case in $E_{2,3}$ and $E_{3,3}$. The only worst terms are of the form of 
\begin{align}
&\mbox{Re}\Big[\alpha \sum_{n,\overline{\N}_{2,n}}\wh{u}(n_1)\psi_k(n_2)\frac{1}{n_2}\left\{10in_2\sum_{\N_{2,n_2}}\wh{u}(n_{2,1})n_{2,2}^2\wh{u}(n_{2,2})\right\}\chi_k(n)\frac1n\wh{u}(n)\Big] \nonumber \\
=&\mbox{Re}\Big[10\alpha i \sum_{n,\overline{\N}_{2,n},\N_{2,n_2}}\wh{u}(n_1)\psi_k(n_2)\wh{u}(n_{2,1})n_{2,2}^2\wh{u}(n_{2,2})\chi_k(n)\frac1n\wh{u}(n)\Big], \label{eq:cubic resonant1}
\end{align}
and
\begin{equation}\label{eq:cubic resonant2}
\mbox{Re}\Big[10\beta i \sum_{n,\overline{\N}_{2,n},\N_{2,n_2}}\wh{u}(n_1)\chi_k(n_2)\wh{u}(n_{2,1})n_{2,2}^2\wh{u}(n_{2,2})\chi_k(n)\frac1n\wh{u}(n)\Big],
\end{equation}
where $\N_{2,n_2}$ is the same set as $\N_{2,n}$ of $n_{2,1}$ and $n_{2,2}$ variables. Especially, if $n_{2.2} = -n$ (exact cubic resonant case), we cannot use the maximum modulation effect to attack the derivative in the high frequency mode. But, since $\psi_k$ and $\chi_k$ are real-valued even functions and $n_1+n_{2,1} = 0$, we observe that
\[\wh{u}(n_1)\psi_k(n_2)\wh{u}(n_{2,1})\chi_k(n)n|\wh{u}(n)|^2 = \psi_k(n_2)|\wh{u}(n_1)|^2\chi_k(n)n|\wh{u}(n)|^2\]
and
\[\wh{u}(n_1)\chi_k(n_2)\wh{u}(n_{2,1})\chi_k(n)n|\wh{u}(n)|^2 = \chi_k(n_2)|\wh{u}(n_1)|^2\chi_k(n)n|\wh{u}(n)|^2\]
Those observations show that both \eqref{eq:cubic resonant1} and \eqref{eq:cubic resonant2} are vanishing since 
\[\psi_k(n_2)|\wh{u}(n_1)|^2\chi_k(n)n|\wh{u}(n)|^2\]
and
\[\chi_k(n_2)|\wh{u}(n_1)|^2\chi_k(n)n|\wh{u}(n)|^2\] 
are real numbers. Moreover, for the other cubic resonant case, by applying the same argument as above, we can observe that those are vanishing. And to conclude, we do not need to consider the cubic resonant case any more. 
\end{remark}
We first consider the cubic term in \eqref{eq:energy2-2.6}. For
\[\sum_{n,\overline{\N}_{2,n}}\wh{N}_2(u)(n_1)\chi_{k}(n_2)\frac1{n_2}\wh{u}(n_2)\chi_{k}(n)\frac1n\wh{u}(n),\]
if the frequency support of $n$ ($\sim2^k$) is the widest among the other frequency supports, it suffices to estimate
\begin{equation}\label{eq:energy2-2.7}
\sum_{0 \le k_1 \le k_2 \le k} 2^{k_2} \left|\sum_{n,\overline{\N}_{3,n}}\int_0^{t_k}\chi_{k_1}(n_1)\wh{u}(n_1)\chi_{k_2}(n_2)\wh{u}(n_2)\chi_{k}(n_3)\wh{u}(n_3)\chi_{k}(n)\wh{u}(n) \; dt\right|.
\end{equation}
We use Lemma \ref{lem:energy-cubic} so that we obtain
\begin{equation}\label{eq:energy2-2.8}
\begin{aligned}
\eqref{eq:energy2-2.7} &\lesssim \sum_{|k-k'|\le 5} 2^{3k/2}\norm{P_{k'}u}_{F_{k'}(T)}^4\\
&+\sum_{k_1 \le k_2 -10}2^{k_1/2}\norm{P_{k_1}u}_{F_{k_1}(T)}\sum_{|k-k'|\le 5}\norm{P_{k'}u}_{F_{k'}(T)}^3\\
&+\sum_{\substack{k_2 \le k -10\\|k_1-k_2|\le 5}}2^{3k_2/2}\norm{P_{k_1}u}_{F_{k_1}(T)}^2\sum_{|k-k'|\le 5}2^{-k}\norm{P_{k'}u}_{F_{k'}(T)}^2\\
&+\sum_{\substack{k_2 \le k -10\\k_1 \le k_2 -10}}2^{k_2}\norm{P_{k_1}u}_{F_{k_1}(T)}\norm{P_{k_2}u}_{F_{k_2}(T)}\sum_{|k-k'|\le 5}2^{-k}\norm{P_{k'}u}_{F_{k'}(T)}^2\\
&\lesssim \norm{u}_{F^{\frac34}(T)}^2\sum_{|k-k'|\le 5}\norm{P_{k'}u}_{F_{k'}(T)}^2.
\end{aligned}
\end{equation}
Otherwise, we only need to consider 
\begin{equation}\label{eq:energy2-2.9}
\sum_{\substack{k_1 \ge k + 10\\|k_1-k_2|\le 5}} 2^{3k_2}2^{-2k} \left|\sum_{n,\overline{\N}_{3,n}}\int_0^{t_k}\chi_{k_1}(n_1)\wh{u}(n_1)\chi_{k_2}(n_2)\wh{u}(n_2)\chi_{k}(n_3)\wh{u}(n_3)\chi_{k}(n)\wh{u}(n) \; dt\right|.
\end{equation}
By using \eqref{eq:energy-cubic3}, we get
\[\begin{aligned}
\eqref{eq:energy2-2.9} &\lesssim \sum_{\substack{k_1 \ge k + 10\\|k_1-k_2|\le 5}}2^{2k_2}\norm{P_{k_1}u}_{F_{k_1}(T)}^2 \sum_{|k-k'|\le 5}2^{-3k/2}\norm{P_{k'}u}_{F_{k'}(T)}^2\\
&\lesssim \norm{u}_{F^1(T)}^2\sum_{|k-k'|\le 5}\norm{P_{k'}u}_{F_{k'}(T)}^2.
\end{aligned}\]
For 
\[\sum_{n,\overline{\N}_{2,n}}\wh{u}(n_1)\chi_{k}(n_2)\frac1{n_2}\wh{N}_2(u)(n_2)\chi_{k}(n)\frac1n\wh{u}(n),\]
the following case is dominant among all cases:
\begin{equation}\label{eq:energy2-2.10}
\sum_{0 \le k_1 \le k_2 \le k_3} 2^{2k_3}2^{-k}\left|\sum_{n,\overline{\N}_{3,n}} \int_0^{t_k}\chi_{k_1}(n_1)\wh{u}(n_1)\chi_{k_2}(n_2)\wh{u}(n_2)\chi_{k_3}(n_3)\wh{u}(n_3)\chi_{k}^2(n)\wh{u}(n) \; dt \right|.
\end{equation}
If $|k-k_3| \le 5$, similarly as \eqref{eq:energy2-2.8}, we obtain
\[\eqref{eq:energy2-2.10} \lesssim \norm{u}_{F^{\frac34}(T)}^2\sum_{|k-k'|\le 5}\norm{P_{k'}u}_{F_{k'}(T)}^2.\]
For the case when $k \le k_2 - 10$, we use \eqref{eq:energy-cubic2}, \eqref{eq:energy-cubic3} and \eqref{eq:energy-cubic4} to estimate \eqref{eq:energy2-2.10}, then we have
\[\begin{aligned}
\eqref{eq:energy2-2.10} &\lesssim 2^{-3k/2}\norm{P_{k}u}_{F_k(T)}\sum_{|k_3-k'|\le 5}2^{2k_3}\norm{P_{k'}u}_{F_{k'}(T)}^3\\
&+\sum_{\substack{k_1 \le k_2 -10\\|k-k_1|\le 5}}2^{-3k/2}\norm{P_{k_1}u}_{F_{k_1}(T)}^2\sum_{|k_3-k'|\le 5}2^{2k_3}\norm{P_{k'}u}_{F_{k'}(T)}^2\\
&+\sum_{\substack{k_1 \le k_2 -10\\k_1 \le k -10}}2^{-2k}\norm{P_{k_1}u}_{F_{k_1}(T)}\norm{P_{k}u}_{F_k(T)}\sum_{|k_3-k'|\le 5}2^{2k}\norm{P_{k'}u}_{F_{k'}(T)}^2\\
&+\sum_{\substack{k_1 \le k_2 -10\\k \le k_1 -10}}2^{-2k}\norm{P_{k_1}u}_{F_{k_1}(T)}\norm{P_{k}u}_{F_k(T)}\sum_{|k_3-k'|\le 5}2^{2k}\norm{P_{k'}u}_{F_{k'}(T)}^2\\
&\lesssim \norm{u}_{F^{1}(T)}^2\sum_{|k-k'|\le 5}\norm{P_{2,k'}u}_{F_{k'}(T)}^2 + 2^{-(s+3/2)k}\norm{P_ku}_{F_k(T)}\norm{u}_{F^2(T)}^2\norm{u}_{F^s(T)},
\end{aligned}\]
for $s \ge 0$. 

For the estimation of the quartic terms in \eqref{eq:energy2-2.6}, by using the similar argument as in the proof of Lemma \ref{lem:energy2-1} and the Cauchy-Schwarz inequality, we use the following estimate:
\begin{equation}\label{eq:energy2-2.11}
\begin{aligned}
&\left|\int_{\T \times [0,T]}u_1u_2u_3u_4u_5 \; dxdt \right|\\
&\hspace{3em}\lesssim 2^{2k_5}\sum_{j_i \ge 2k_6} \left|\sum_{\overline{n} \in \Gamma_5(\Z)}\int_{\overline{\tau}\in\Gamma_5(\R)}\prod_{i=1}^{5}\ft[\gamma(2^{2k_5}t-m)u_i](\tau_i,n_i) \right|\\
&\hspace{3em}\lesssim 2^{2k_5}\prod_{l=1}^{3}2^{k_l/2}\sum_{j_i \ge 2k_5}2^{-(j_{max}+j_{sub})/2}\prod_{i=1}^{5}2^{j_i/2}\norm{\eta_{j_i}(\tau_i-\mu(n_i))\ft[\gamma(2^{2k_5}t-m)u_i]}_{L_{\tau_i}^2\ell_{n_i}^2}\\
&\hspace{3em}\lesssim 2^{(k_1+k_2+k_3)/2}\prod_{i=1}^{5}\norm{u_i}_{F_{k_i}(T)},
\end{aligned}
\end{equation}
where $u_i = P_{k_i}u \in F_{k_i}(T)$, $i=1,2,3,4,5$ and assuming that $k_1 \le k_2 \le k_3 \le k_4 \le k_5$.

Since the cubic term in $\wh{N}_2(u)$ has the one total derivative, it suffices to estimate the following two terms:
\begin{equation}\label{eq:energy2-2.12}
\sum_{\substack{0 \le k_1 \le k_2 \le k_3\\k \le k_2 - 10}}2^{k_3}2^{-2k}\left|\sum_{\overline{n} \in \Gamma_5(\Z)}\int_0^{t_k}\chi_{k_1}(n_1)\wh{u}(n_1)\chi_{k_2}(n_2)\wh{u}(n_2)\chi_{k_3}(n_3)\wh{u}(n_3)\chi_{k}(n_4)\wh{u}(n_4)\chi_k(n)\wh{u}(n)\;dt\right|
\end{equation} 
and
\begin{equation}\label{eq:energy2-2.13}
\sum_{0 \le k_1 \le k_2 \le k_3 \le k_4}2^{-k}\left|\sum_{\overline{n} \in \Gamma_5(\Z)}\int_0^{t_k}\chi_{k_1}(n_1)\wh{u}(n_1)\chi_{k_2}(n_2)\wh{u}(n_2)\chi_{k_3}(n_3)\wh{u}(n_3)\chi_{k}(n_4)\wh{u}(n_4)\chi_k(n)\wh{u}(n)\;dt\right|.
\end{equation} 
By using \eqref{eq:energy2-2.11}, we can easily have
\[\eqref{eq:energy2-2.12} + \eqref{eq:energy2-2.13} \lesssim \norm{u}_{F^{\frac12+}(T)}^3\sum_{|k-k'|\le5}\norm{P_{k'}u}_{F_{k'}(T)}^2 + 2^{-(s+1/2)k}\norm{P_ku}_{F_k(T)}\norm{u}_{F^{\frac12+}(T)}^3\norm{u}_{F^s(T)},\]
for $s \ge 0$.

Together with all bounds of the cubic and quartic terms, we conclude that
\begin{equation}\label{eq:cubic,quartic bound}
\sum_{k \ge 1}2^{2sk}\sup_{t_k \in [0,T]}\left|\int_0^{t_k} E_{2,3} + E_{3,3} \; dt \right| \lesssim \left(\norm{u}_{F^2(T)}^2 + \norm{u}_{F^{\frac12+}(T)}^3\right)\norm{u}_{F^s(T)}^2.
\end{equation}
Therefore, we complete the proof of Proposition \ref{prop:energy2-2} by recalling the definition of the modified energy \eqref{eq:new energy2-2} and gathering \eqref{eq:quad bound1}, \eqref{eq:quad bound2}, \eqref{eq:quad bound3} and \eqref{eq:cubic,quartic bound}.
\end{proof}

As a Corollary to Lemma \ref{lem:comparable energy2-1} and Proposition \ref{prop:energy2-2}, we obtain \emph{a priori bound} of $\norm{u}_{E^s(T)}$ for a smooth solution $u$ to the equation \eqref{eq:5kdv4}.
\begin{corollary}\label{cor:energy2-2}
Let $s \ge 2$ and $T \in (0,1]$. Then, there exists $0 < \delta \ll 1$ such that
\[\norm{u}_{E^s(T)}^2 \lesssim (1+ \norm{u_0}_{H^s})\norm{u_0}_{H^s}^2 + \left(\norm{u}_{F^{\frac32+}(T)}+\norm{u}_{F^{2}(T)}^2+\norm{u}_{F^{\frac12+}(T)}^3\right)\norm{u}_{F^{s}(T)}^2,\]
for the solution $u \in C([-T,T];H^{\infty}(\T))$ to \eqref{eq:5kdv4} with $\norm{u}_{L_T^{\infty}H_x^{\frac12+}} \le \delta$.
\end{corollary}

Next, we consider the energy estimate for the difference of two solutions $u_1$ and $u_2$ to the equation in \eqref{eq:5kdv4}. Let $w = u_1 - u_2$, then $w$ satisfies
\begin{equation}\label{eq:5kdv5}
\pt\wh{w}(n) - i\mu_2(n)\wh{w}(n)=\wh{N}_{1,1}(u_1,u_2,w)+\wh{N}_{1,2}(u_1,u_2,w)+\wh{N}_{1,3}(u_1,u_2,w)+\wh{N}_{1,4}(u_1,u_2,w),
\end{equation}
with $w(0,x) = w_0(x) = u_{1,0}(x) - u_{2,0}(x)$ and where
\begin{equation}\label{eq:energy-nonlinear2-1}
\wh{N}_{1,1}(u_1,u_2,w) = -30in(|\wh{u}_1(n)|^2\wh{w}(n) + \wh{u}_1(n)\wh{u}_2(n)\wh{w}(-n) + |\wh{u}_2(n)|^2\wh{w}(n)),
\end{equation}

\begin{equation}\label{eq:energy-nonlinear2-2}
\begin{aligned}
\wh{N}_{1,2}(u_1,u_2,w) &= 10in \sum_{\N_{3,n}} \wh{w}(n_1)\wh{u}_1(n_2)\wh{u}_1(n_3)\\ 
&+ 10in \sum_{\N_{3,n}} \wh{u}_2(n_1)\wh{w}(n_2)\wh{u}_1(n_3)\\
&+ 10in \sum_{\N_{3,n}} \wh{u}_2(n_1)\wh{u}_2(n_2)\wh{w}(n_3)
\end{aligned}
\end{equation}

\begin{equation}\label{eq:energy-nonlinear2-3}
\wh{N}_{1,3}(u_1,u_2,w) = 10in \sum_{\N_{2,n}} n_2^2(\wh{w}(n_1)\wh{u}_1(n_2) + \wh{u}_2(n_1)\wh{w}(n_2)) 
\end{equation}
and
\begin{equation}\label{eq:energy-nonlinear2-4}
\wh{N}_{1,4}(u_1,u_2,w) = 10i \sum_{\N_{2,n}} n_1n_2^2(\wh{w}(n_1)\wh{u}_1(n_2) + \wh{u}_2(n_1)\wh{w}(n_2))
\end{equation}
We denote $\wh{N}_{1,1}(u_1,u_2,w)+\wh{N}_{1,2}(u_1,u_2,w)+\wh{N}_{1,3}(u_1,u_2,w)+\wh{N}_{1,4}(u_1,u_2,w)$ by $\wh{N}_2(u_1,u_2,w)$. 

For $k \ge 1$, we define the localized modified energy for the difference of two solutions by
\[\begin{aligned}
\wt{E}_{k}(w)(t) &= \norm{P_kw(t)}_{L_x^2}^2 + \mbox{Re}\left[\wt{\alpha} \sum_{n,\overline{\N}_{2,n}}\wh{u}_2(n_1)\psi_k(n_2)\frac{1}{n_2}\wh{w}(n_2)\chi_k(n)\frac1n\wh{w}(n)\right]\\
&+ \mbox{Re}\left[\wt{\beta} \sum_{n,\overline{\N}_{2,n}}\wh{u}_2(n_1)\chi_k(n_2)\frac{1}{n_2}\wh{w}(n_2)\chi_k(n)\frac1n\wh{w}(n)\right]
\end{aligned}\]
and
\[\wt{E}_{T}^s(w) = \norm{P_0w(0)}_{L_x^2}^2 + \sum_{k \ge 1}2^{2sk} \sup_{t_k \in [-T,T]} \wt{E}_{2,k}(w)(t_k),\]
where $\wt{\alpha}$ and $\wt{\beta}$ are real and will be chosen later.

Similarly as in Lemma \ref{lem:comparable energy2-1}, we can show that $\wt{E}_{T}^s(w)$ and $\norm{w}_{E^s(T)}$ are comparable.
\begin{lemma}\label{lem:comparable energy2-2}
Let $s > \frac12$. Then, there exists $0 < \delta \ll 1$ such that  
\[\frac12\norm{w}_{E^s(T)}^2 \le \wt{E}_{T}^s(w) \le \frac32\norm{w}_{E^s(T)}^2,\]
for all $u_2 \in E^s(T) \cap C([-T,T];H^s(\T))$ satisfying $\norm{u_2}_{L_T^{\infty}H^s(\T)} \le \delta$.
\end{lemma}

\begin{proposition}\label{prop:energy2-3}
Let $s \ge 2$ and $T \in (0,1]$, Then, for solutions $w \in C([-T,T];H^{\infty}(\T))$ to \eqref{eq:5kdv5} and $u_1, u_2 \in C([-T,T];H^{\infty}(\T))$ to \eqref{eq:5kdv4}, we have
\begin{equation}\label{eq:energy2-3.1}
\begin{aligned}
\wt{E}_{T}^0(w) &\lesssim (1+ \norm{u_{1,0}}_{H^{\frac12+}}+\norm{u_{2,0}}_{H^{\frac12+}})\norm{w_0}_{L_x^2}^2\\
&+(1+ \norm{u_1}_{F^2(T)}+\norm{u_2}_{F^2(T)})(\norm{u_1}_{F^2(T)}+\norm{u_2}_{F^2(T)})\norm{w}_{F^0(T)}^2\\
&+\left(\sum_{1 \le i \le j \le k \le 2}\norm{u_i}_{F^2(T)}\norm{u_j}_{F^2(T)}\norm{u_k}_{F^2(T)}\right)\norm{w}_{F^0(T)}^2.
\end{aligned} 
\end{equation}
and
\begin{equation}\label{eq:energy2-3.2}
\begin{aligned}
\wt{E}_{T}^s(w) &\lesssim (1+ \norm{u_{2,0}}_{H^{\frac12+}})\norm{w_0}_{H^s}^2\\
&+(\norm{u_1}_{F^{2s}(T)}+\norm{u_2}_{F^{2s}(T)})\norm{w}_{F^0(T)}\norm{w}_{F^s(T)}\\
&+(\norm{u_1}_{F^s(T)}+\norm{u_2}_{F^s(T)})\norm{w}_{F^s(T)}^2\\
&+(\norm{u_1}_{F^s(T)}+\norm{u_2}_{F^s(T)})(\norm{u_1}_{F^{2s}(T)}+\norm{u_2}_{F^{2s}(T)})\norm{w}_{F^0(T)}\norm{w}_{F^s(T)}\\
&+\left(\sum_{1 \le i \le j \le 2}\norm{u_i}_{F^s(T)}\norm{u_j}_{F^s(T)}\right)\norm{w}_{F^s(T)}^2\\
&+\left(\sum_{1 \le i \le j \le k \le 2}\norm{u_i}_{F^s(T)}\norm{u_j}_{F^s(T)}\norm{u_k}_{F^s(T)}\right)\norm{w}_{F^s(T)}^2.
\end{aligned}
\end{equation} 
\end{proposition}

\begin{remark}
In fact, in the energy estimates for the difference of two solutions, since the symmetry of functions breaks down, one can obtain Proposition \ref{prop:energy2-3} by defining the localized modified energy by
\[\begin{aligned}
\wt{E}_{k}(w)(t) &= \norm{P_kw(t)}_{L_x^2}^2\\ 
&+ \mbox{Re}\left[\wt{\alpha}_{1} \sum_{n,\overline{\N}_{2,n}}\wh{u}_1(n_1)\psi_k(n_2)\frac{1}{n_2}\wh{w}(n_2)\chi_k(n)\frac1n\wh{w}(n)\right]\\
&+\mbox{Re}\left[\wt{\alpha}_{2} \sum_{n,\overline{\N}_{2,n}}\wh{u}_2(n_1)\psi_k(n_2)\frac{1}{n_2}\wh{w}(n_2)\chi_k(n)\frac1n\wh{w}(n)\right]\\
&+ \mbox{Re}\left[\wt{\beta}_{1} \sum_{n,\overline{\N}_{2,n}}\wh{u}_1(n_1)\chi_k(n_2)\frac{1}{n_2}\wh{w}(n_2)\chi_k(n)\frac1n\wh{w}(n)\right]\\
&+ \mbox{Re}\left[\wt{\beta}_{2} \sum_{n,\overline{\N}_{2,n}}\wh{u}_2(n_1)\chi_k(n_2)\frac{1}{n_2}\wh{w}(n_2)\chi_k(n)\frac1n\wh{w}(n)\right]
\end{aligned}\]
and using another forms of \eqref{eq:energy-nonlinear2-1}, \eqref{eq:energy-nonlinear2-2}, \eqref{eq:energy-nonlinear2-3} and \eqref{eq:energy-nonlinear2-4}, by the symmetry of $u_1$ and $u_2$. But, for the simplicity, we do not distinguish between $u_1$ and $u_2$ in the following proof of Proposition.
\end{remark}
\begin{proof}
We use similar argument as in the proof of Proposition \ref{prop:energy2-2}. For any $k \in \Z_+$ and $t \in [-T,T]$, we differentiate $\wt{E}_{k}(w)$ with respect to $t$ and deduce that 
\[\frac{d}{dt}\wt{E}_k(w) = \frac{d}{dt}\wt{I}(t) + \frac{d}{dt}\wt{II}(t) + \frac{d}{dt}\wt{III}(t),\]
where 
\[\begin{aligned}
\frac{d}{dt}\wt{I}(t) &= \frac{d}{dt}\norm{P_kw}_{L_x^2}^2\\
&= -30 i \sum_{n} \chi_k(n) n \wh{u}_1(-n)\wh{u}_2(-n)\wh{w}(n)\chi_k(n)\wh{w}(n)\\
&+ 2\mbox{Re}\left[\sum_{n}\chi_k(n)\left(\overline{\wh{N}}_{2,2}(u_1,u_2,w)+\overline{\wh{N}}_{2,3}(u_1,u_2,w)+\overline{\wh{N}}_{2,4}(u_1,u_2,w)\right)\chi_k(n)\wt{w}(n)\right]\\
&=: \wt{E}_{1,1},
\end{aligned}\]
\begin{align*}
\frac{d}{dt} \wt{II}(t) &= \mbox{Re}\Big[\wt{\alpha} i \sum_{n,\overline{\N}_{2,n}}5(n_1^3n_2n - n_1n_2^2n_3^2)\wh{u}_2(n_1)\psi_k(n_2)\frac{1}{n_2}\wh{w}(n_2)\chi_k(n)\frac1n\wh{w}(n)\Big]\\
&+ \mbox{Re}\Big[\wt{c}_1\wt{\alpha} i \sum_{n,\overline{\N}_{2,n}}3n_1n_2n \wh{u}_2(n_1)\psi_k(n_2)\frac{1}{n_2}\wh{w}(n_2)\chi_k(n)\frac1n\wh{w}(n)\Big]\\
&+ \mbox{Re}\Big[\wt{\alpha} \sum_{n,\overline{\N}_{2,n}}\wh{N}_2(u_2)(n_1)\psi_k(n_2)\frac{1}{n_2}\wh{w}(n_2)\chi_k(n)\frac1n\wh{w}(n)+\wh{u}_2(n_1)\psi_k(n_2)\frac{1}{n_2}\wh{N}_2(w)(n_2)\chi_k(n)\frac1n\wh{w}(n)\\
&\hspace{3em}+\wh{u}_2(n_1)\psi_k(n_2)\frac{1}{n_2}\wh{w}(n_2)\chi_k(n)\frac1n\wh{N}_2(w)(n)\Big]\\
&=:\wt{E}_{2,1} + \wt{E}_{2,2} + \wt{E}_{2,3} =: \wt{E}_2
\end{align*}
and
\begin{align*}
\frac{d}{dt} \wt{III}(t) &= \mbox{Re}\Big[\wt{\beta}i \sum_{n,\overline{\N}_{2,n}}5(n_1^3n_2n - n_1n_2^2n_3^2)\wh{u}_2(n_1)\chi_k(n_2)\frac{1}{n_2}\wh{w}(n_2)\chi_k(n)\frac1n\wh{w}(n)\Big]\\
&+ \mbox{Re}\Big[\wt{c}_1\wt{\beta}i \sum_{n,\overline{\N}_{2,n}}3n_1n_2n \wh{u}_2(n_1)\chi_k(n_2)\frac{1}{n_2}\wh{w}(n_2)\chi_k(n)\frac1n\wh{w}(n)\Big]\\
&+ \mbox{Re}\Big[\wt{\beta} \sum_{n,\overline{\N}_{2,n}}\wh{N}_2(u_2)(n_1)\chi_k(n_2)\frac{1}{n_2}\wh{w}(n_2)\chi_k(n)\frac1n\wh{w}(n)+\wh{u}_2(n_1)\chi_k(n_2)\frac{1}{n_2}\wh{N}_2(w)(n_2)\chi_k(n)\frac1n\wh{w}(n)\\
&\hspace{3em}+\wh{u}_2(n_1)\chi_k(n_2)\frac{1}{n_2}\wh{w}(n_2)\chi_k(n)\frac1n\wh{N}_2(w)(n)\Big]\\
&=:\wt{E}_{3,1} + \wt{E}_{3,2} + \wt{E}_{3,3} =: \wt{E}_3
\end{align*}
In order to prove Proposition \ref{prop:energy2-3}, we need to control
\[\left|\int_0^{t_k} \wt{E}_1 + \wt{E}_2 + \wt{E}_3 \; dt \right|.\]
By choosing $\wt{\alpha} = -4$ and $\wt{\beta} = 6$, we have, for each $k \ge 1$, that
\[\left|\int_0^{t_k} \wt{E}_1 + \wt{E}_{2,1} + \wt{E}_{3,1} \; dt \right| \lesssim \sum_{i=1}^{10} \wt{B}_i(k),\]
where
\[\begin{aligned}
\wt{B}_1(k) = \sum_{0\le k_1\le k-10}\Big|&\sum_{n,\overline{\N}_{2,n}}\int_0^{t_k} \chi_k(n)n[\chi_{k_1}(n_1)\wh{u}_2(n_1)n_2^2\wh{w}(n_2)]\chi_k(n)\wh{w}(n) \;dt\\
&+ \frac12\sum_{n,\overline{\N}_{2,n}}\int_0^{t_k} \chi_{k_1}(n_1)n_1\wh{u}_2(n_1)\chi_k(n_2)n_2\wh{w}(n_2)\chi_k(n)n\wh{w}(n) \;dt\\
&- \sum_{n,\overline{\N}_{2,n}}\int_0^{t_k} \chi_{k_1}(n_1)n_1\wh{u}_2(n_1)\psi_k(n_2)n_2\wh{w}(n_2)\chi_k(n)n\wh{w}(n) \;dt \Big|,
\end{aligned}\]
\[\begin{aligned}
\wt{B}_2(k) =\sum_{0\le  k_1 \le k-10}\Big|&\sum_{n,\overline{\N}_{2,n}}\int_0^{t_k} \chi_k(n)[\chi_{k_1}(n_1)n_1\wh{u}_2(n_1)n_2^2\wh{w}(n_2)]\chi_k(n)\wh{w}(n)\; dt\\
&+ \sum_{n,\overline{\N}_{2,n}}\int_0^{t_k} \chi_{k_1}(n_1)n_1\wh{u}_2(n_1)\chi_k(n_2)n_2\wh{w}(n_2)\chi_k(n)n\wh{w}(n)\;dt \Big|,
\end{aligned}\]
\[\wt{B}_3(k) =\sum_{\substack{k_1 \ge k -9\\k_2 \ge 0}}\Big|\sum_{n,\overline{\N}_{2,n}}\int_0^{t_k} \chi_{k_1}(n_1)\wh{u}_2(n_1)\chi_{k_2}(n_2)n_2^2\wh{w}(n_2)\chi_k^2(n)n\wh{w}(n)\; dt\Big|,\]
\[\wt{B}_4(k) =\sum_{\substack{k_1 \ge k -9\\k_2 \ge 0}}\Big|\sum_{n,\overline{\N}_{2,n}}\int_0^{t_k} \chi_{k_1}(n_1)n_1\wh{u}_2(n_1)\chi_{k_2}(n_2)n_2^2\wh{w}(n_2)\chi_k^2(n)\wh{w}(n)\; dt\Big|,\]
\[\wt{B}_5(k) =\sum_{|k-k_1|\le 5}\Big|\sum_{n,\overline{\N}_{2,n}}\int_0^{t_k} \chi_{k_1}(n_1)n_1\wh{u}_2(n_1)(\chi_k(n_2) + \psi_k(n_3))n_2\wh{w}(n_2)\chi_k(n)n\wh{w}(n) \;dt\Big|,\]
\[\wt{B}_6(k) =\sum_{k_1\ge 0}\Big|\sum_{n,\overline{\N}_{2,n}}\int_0^{t_k} n_1^3n_2n\chi_{k_1}(n_1)\wh{u}_2(n_1)(\chi_k(n_2)+\psi_k(n_3))\frac{1}{n_2}\wh{w}(n_2)\chi_k(n)\frac1n\wh{w}(n) \;dt\Big|,\]
\[\wt{B}_7(k) =\sum_{k_1, k_2 \ge 0}\Big|\sum_{n,\overline{\N}_{2,n}}\int_0^{t_k} \chi_{k_1}(n_1)\wh{w}(n_1)\chi_{k_2}(n_2)n_2^2\wh{u}_1(n_2)\chi_k^2(n)n\wh{w}(n)\; dt\Big|,\]
\[\wt{B}_8(k) =\sum_{k_1, k_2 \ge 0}\Big|\sum_{n,\overline{\N}_{2,n}}\int_0^{t_k} \chi_{k_1}(n_1)n_1\wh{w}(n_1)\chi_{k_2}(n_2)n_2^2\wh{u}_2(n_2)\chi_k^2(n)\wh{w}(n)\; dt\Big|,\]
\[\wt{B}_9(k) =\Big|\sum_{n}\int_0^{t_k} \chi_{k}(n)n\wh{u}_1(-n)\wh{u}_2(-n)\wh{w}(n)\chi_k(n)\wh{w}(n) \;dt\Big|\]
and
\[\wt{B}_{10}(k) =\Big|\sum_{n}\int_0^{t_k} \chi_k(n)\overline{\wh{N}}_{2,2}(u_1,u_2,w) \chi_k(n)\wh{w}(n) \;dt\Big|.\]
Similarly as the estimation of $B_1(k) + B_2(k)$ in the proof of Proposition \ref{prop:energy2-2}, we have
\[\wt{B}_1(k)+\wt{B}_2(k) \lesssim \norm{u_2}_{F^{\frac32+}(T)}\sum_{|k-k'| \le 5} \norm{P_{k'}w}_{F_{k'}(T)}^2.\]
For $\wt{B}_3(k)$ and $\wt{B}_4(k)$, we divide the summation range into
\[\sum_{\substack{|k_1 -k|\le 5 \\ |k_2 -k|\le 5}} + \sum_{\substack{k_2 \le k - 10 \\ |k_1 -k|\le 5}}+\sum_{\substack{k_1 \ge k + 10 \\ |k_1 -k_2|\le 5}}.\]
On the first summation, $\wt{B}_3(k)$ and $\wt{B}_4(k)$ are bounded by
\[\norm{u_2}_{F^{\frac32}(T)}\sum_{|k-k'| \le 5} \norm{P_{k'}w}_{F_{k'}(T)}^2,\]
by using the same way to the estimation of $B_3(k)$ and $B_4(k)$ in the proof of Proposition \ref{prop:energy2-2}. On the rest summations, we have from \eqref{eq:energy2-1.2} and the Cauchy-Schwarz inequality that
\[\begin{aligned}
&\sum_{k_2 \le k- 10} 2^{3k_2/2}\norm{P_{k_2}w}_{F_{k_2}(T)}\sum_{|k_1 - k| \le 5}\norm{P_{k_1}u_2}_{F_{k_1}(T)}\norm{P_{k}w}_{F_k(T)}\\
+&2^{k/2}\norm{P_{k}w}_{F_k(T)}\sum_{\substack{k_1 \ge k+ 10 \\|k_1 - k_2| \le 5}}2^{k_1}\norm{P_{k_1}u_2}_{F_{k_1}(T)}\norm{P_{k_2}w}_{F_{k_2}(T)}\\
\lesssim&\norm{w}_{F^{\frac32+}(T)}\sum_{|k_1 - k| \le 5}\norm{P_{k_1}u_2}_{F_{k_1}(T)}\norm{P_{k}w}_{F_k(T)} \\
+& 2^{-(s+\varepsilon)k}\norm{P_kw}_{F_k(T)}\norm{u_2}_{F^{\frac32+}(T)}\norm{w}_{F^s(T)},
\end{aligned}\]
for $s \ge 0$ and $0 < \varepsilon \ll 1$, and hence we obtain
\[\sum_{k \ge 1} 2^{2sk} \sup_{t_k \in [0,T]} (\wt{B}_3(k) + \wt{B}_4(k)) \lesssim \norm{u_2}_{F^{\frac32+}(T)}\norm{w}_{F^s(T)}^2 + \norm{u_2}_{F^s(T)}\norm{w}_{F^s(T)}^2,\]
whenever $s > \frac32$, and
\[\sum_{k \ge 1}  \sup_{t_k \in [0,T]} (\wt{B}_3(k) + \wt{B}_4(k)) \lesssim \norm{u_2}_{F^{\frac32+}(T)}\norm{w}_{F^0(T)}^2,\]
at $L^2$-level.

For $\wt{B}_5(k)$, by using \eqref{eq:energy2-1.1} and the Cauchy-Schwarz inequality that
\[\wt{B}_5(k) \lesssim \norm{u_2}_{F^{\frac32}(T)}\sum_{|k-k'| \le 5}\norm{P_{k'}w}_{F_{k'}(T)}^2.\]
For $\wt{B}_6(k)$, we use \eqref{eq:energy2-1.1} and \eqref{eq:energy2-1.2}, respectively, to obtain
\[\begin{aligned}
\wt{B}_6(k) & \lesssim \sum_{|k-k'|\le 5}2^{3k/2}\norm{P_{k'}u_2}_{F_{k'}(T)}\norm{P_{k'}w}_{F_{k'}(T)}^2\\
&+ \sum_{k_1 \le k-10}2^{5k_1/2}\norm{P_{k'}u_2}_{F_{k'}(T)}\sum_{|k-k'|\le 5}2^{-k}\norm{P_{k'}w}_{F_{k'}(T)}^2\\
&\lesssim \norm{u_2}_{F^{\frac32+}(T)}\sum_{|k-k'|\le 5}2^{-k}\norm{P_{k'}w}_{F_{k'}(T)}^2.
\end{aligned}\]
For $\wt{B}_7(k)$ and $\wt{B}_8(k)$, since much more derivatives are taken on $P_{k_2}u_1$ and $P_{k_2}u_2$ than $P_{k_1}w$ and $P_kw$, we may assume $k_2 = \max(k_1,k_2,k)$. We use Lemma \ref{lem:energy2-1} and the Cauchy-Schwarz inequality to obtain that\footnote{For simplicity, we estimate the dominant term for each case.}
\[\begin{aligned}
\wt{B}_7(k) + \wt{B}_8(k) &\lesssim \sum_{k_1 \le k -10}2^{-k_1/2}\norm{P_{k_1}w}_{F_{k_1}(T)}\sum_{|k-k'|\le 5}2^{2k}\norm{P_{k'}u_1}_{F_{k'}(T)}\norm{P_{k'}w}_{F_{k'}(T)}\\
&+2^{-k/2}\norm{P_{k}w}_{F_k(T)}\sum_{|k_1-k_2|\le 5}2^{2k_2}\norm{P_{k_1}w}_{F_{k_1}(T)}\norm{P_{k_2}u_2}_{F_{k'}(T)}\\
&+\sum_{|k-k'|\le 5}2^{3k/2}\norm{P_{k'}u_2}_{F_{k'}(T)}\norm{P_{k'}w}_{F_{k'}(T)}^2\\
&\lesssim \norm{w}_{F^0(T)}\sum_{|k-k'|\le 5}2^{2k}\norm{P_{k'}u_1}_{F_{k'}(T)}\norm{P_{k'}w}_{F_{k'}(T)} \\
&+2^{-(s+1/2)k}\norm{P_kw}_{F_k(T)}\norm{u_2}_{F^2(T)}\norm{w}_{F^s(T)} + \norm{u_2}_{F^{\frac32}(T)}\sum_{|k-k'|\le 5}\norm{P_{k'}w}_{F_{k'}(T)}^2,
\end{aligned}\]
which implies
\[\sum_{k \ge 1} 2^{2sk} \sup_{t_k \in [0,T]} (\wt{B}_7(k) + \wt{B}_8(k)) \lesssim \norm{u_1}_{F^{2s}(T)}\norm{w}_{F^0(T)}\norm{w}_{F^s(T)} + \norm{u_2}_{F^s(T)}\norm{w}_{F^s(T)}^2,\]
whenever $s \ge 2$, and
\[\sum_{k \ge 1}  \sup_{t_k \in [0,T]} (\wt{B}_7(k) + \wt{B}_8(k)) \lesssim \norm{u_1}_{F^{2}(T)}\norm{w}_{F^0(T)}^2,\]
at $L^2$-level.

For $\wt{B}_9(k)$, since
\[\left|\sum_{n} \chi_{k}(n)n\wh{u}_1(-n)\wh{u}_2(-n)\wh{w}(n)\chi_k(n)\wh{w}(n) \right| \lesssim \norm{u_1(t)}_{H^{\frac12}}\norm{u_1(t)}_{H^{\frac12}} \norm{P_{k}w}_{L_x^2}^2,\]
by embedding property \eqref{eq:small data1.1}, we obtain
\[\wt{B}_9(k) \lesssim \norm{u_1}_{F^{\frac12}(T)}\norm{u_2}_{F^{\frac12}(T)}\norm{P_kw}_{F_k(T)}^2.\]
For $\wt{B}_10(k)$, it suffices to consider
\begin{equation}\label{eq:energy2-3.3}
\sum_{k_1,k_2,k_3 \ge 0}2^k\Big|\sum_{n,\overline{\N}_{3,n}}\int_0^{t_k} \chi_{k_1}(n_1)\wh{u}_2(n_1)\chi_{k_2}(n_2)\wh{u}_2(n_2)\chi_{k_3}(n_3)\wh{w}(n_3)\chi_k^2(n)\wh{w}(n) \;dt\Big|.
\end{equation}
Without loss of generality, we assume that $k_1 \le k_2$. If $k = \max(k_1,k_2,k_3,k)$, by using \eqref{eq:energy-cubic1} and \eqref{eq:energy-cubic2}, we first have
\[\begin{aligned}
\eqref{eq:energy2-3.3} &\lesssim \sum_{|k-k'| \le 5}2^{3k/2}\norm{P_{k'}u_2}_{F_{k'}(T)}^2\norm{P_{k'}w}_{F_{k'}(T)}^2\\
&+ \sum_{k_1 \le k_2 - 10}2^{k_1/2}\norm{P_{k_1}u_2}_{F_{k_1}(T)} \sum_{|k-k'|\le 5}\norm{P_{k'}u_2}_{F_{k'}(T)}\norm{P_{k'}w}_{F_{k'}(T)}^2\\
&+ \sum_{k_3 \le k_1 - 10}2^{k_3/2}\norm{P_{k_3}w}_{F_{k_3}(T)} \sum_{|k-k'|\le 5}\norm{P_{k'}u_2}_{F_{k'}(T)}^2\norm{P_{k'}w}_{F_{k'}(T)}\\
&\lesssim \norm{u_2}_{F^{\frac34}(T)}^2\sum_{|k-k'|\le 5}\norm{P_{k'}w}_{F_{k'}(T)}^2 \\
&+ \norm{w}_{F^0(T)}\norm{u_2}_{F^{\frac12}(T)}\sum_{|k-k'|\le 5}\norm{P_{k'}u_2}_{F_{k'}(T)}\norm{P_{k'}w}_{F_{k'}(T)}.
\end{aligned}\]
Moreover, by using \eqref{eq:energy-cubic3} and \eqref{eq:energy-cubic4}, we also obtain
\[\begin{aligned}
\eqref{eq:energy2-3.3} &\lesssim \sum_{\substack{k_2 \le k- 10\\k_1 \le k_2 - 10}}\norm{P_{k_1}u_2}_{F_{k_1}(T)}\norm{P_{k_2}u_2}_{F_{k_2}(T)}\sum_{|k-k'| \le 5}\norm{P_{k'}w}_{F_{k'}(T)}^2\\
&+\sum_{\substack{k_2 \le k- 10\\|k_1 -k_2|\le 5}}2^{k_2/2}\norm{P_{k_1}u_2}_{F_{k_1}(T)}\norm{P_{k_2}u_2}_{F_{k_2}(T)}\sum_{|k-k'| \le 5}\norm{P_{k'}w}_{F_{k'}(T)}^2\\
&+\sum_{\substack{k_1 \le k- 10\\k_3 \le k- 10\\|k_1-k_3|\le 5}}2^{k_1/2}\norm{P_{k_1}u_2}_{F_{k_1}(T)}\norm{P_{k_3}w}_{F_{k_2}(T)}\sum_{|k-k'| \le 5}\norm{P_{k'}u_2}_{F_{k'}(T)}\norm{P_{k'}w}_{F_{k'}(T)}\\
&+\sum_{\substack{k_1 \le k- 10\\k_3 \le k- 10\\k_1\le k_3 -10}}\norm{P_{k_1}u_2}_{F_{k_1}(T)}\norm{P_{k_3}w}_{F_{k_2}(T)}\sum_{|k-k'| \le 5}\norm{P_{k'}u_2}_{F_{k'}(T)}\norm{P_{k'}w}_{F_{k'}(T)}\\
&+\sum_{\substack{k_1 \le k- 10\\k_3 \le k- 10\\k_3\le k_1 -10}}\norm{P_{k_1}u_2}_{F_{k_1}(T)}\norm{P_{k_3}w}_{F_{k_2}(T)}\sum_{|k-k'| \le 5}\norm{P_{k'}u_2}_{F_{k'}(T)}\norm{P_{k'}w}_{F_{k'}(T)}\\
&\lesssim \norm{u_2}_{F^{\frac14+}(T)}^2\sum_{|k-k'|\le 5}\norm{P_{k'}w}_{F_{k'}(T)}^2 \\
&+ \norm{w}_{F^{\frac14}(T)}\norm{u_2}_{F^{\frac14}(T)}\sum_{|k-k'|\le 5}\norm{P_{k'}u_2}_{F_{k'}(T)}\norm{P_{k'}w}_{F_{k'}(T)}.
\end{aligned}\]
If $k \neq \max(k_1,k_2,k_3,k)$, we use \eqref{eq:energy-cubic2}, \eqref{eq:energy-cubic3} and \eqref{eq:energy-cubic4} to obtain that
\[\begin{aligned}
\eqref{eq:energy2-3.3} &\lesssim 2^{3k/2}\norm{P_kw}_{F_k(T)}\sum_{\substack{k_3 \ge k + 10\\|k_3-k'| \le 5}}2^{-k_3}\norm{P_{k'}u_2}_{F_{k'}(T)}^2\norm{P_{k'}w}_{F_{k'}(T)}\\
&+\sum_{\substack{k_1 \le k_2- 10\\|k_1 -k|\le 5}}2^{3k/2}\norm{P_{k_1}u_2}_{F_{k_1}(T)}\norm{P_{k}w}_{F_k(T)}\sum_{|k_2-k'| \le 5}2^{-k_2}\norm{P_{k'}u_2}_{F_{k'}(T)}\norm{P_{k'}w}_{F_{k'}(T)}\\
&+\sum_{k_1 \le k - 10}2^{k}\norm{P_{k_1}u_2}_{F_{k_1}(T)}\norm{P_{k}w}_{F_{k_2}(T)}\sum_{\substack{ k_2 \ge k+10 \\|k_2-k'| \le 5}}2^{-k_2}\norm{P_{k'}u_2}_{F_{k'}(T)}\norm{P_{k'}w}_{F_{k'}(T)}\\
&+\sum_{\substack{k_1 \le k_2- 10\\k \le k_1 - 10}}2^{k}\norm{P_{k_1}u_2}_{F_{k_1}(T)}\norm{P_{k}w}_{F_{k_2}(T)}\sum_{|k_2-k'| \le 5}2^{-k_2}\norm{P_{k'}u_2}_{F_{k'}(T)}\norm{P_{k'}w}_{F_{k'}(T)}\\
&+\sum_{\substack{k_3 \le k_2- 10\\|k_3 -k|\le 5}}2^{3k/2}\norm{P_{k_3}w}_{F_{k_3}(T)}^2\sum_{|k_2-k'| \le 5}2^{-k_2}\norm{P_{k'}u_2}_{F_{k'}(T)}^2\\
&+\sum_{k_3 \le k - 10}2^{k}\norm{P_{k_3}w}_{F_{k_2}(T)}\norm{P_{k}w}_{F_{k_3}(T)}\sum_{\substack{k_2 \ge k + 10 \\ |k_2-k'| \le 5}}2^{-k_2}\norm{P_{k'}u_2}_{F_{k'}(T)}^2\\
&+\sum_{\substack{k_3 \le k_2- 10\\k \le k_3 -10}}2^k\norm{P_{k_3}w}_{F_{k_2}(T)}\norm{P_{k}w}_{F_{k_3}(T)}\sum_{|k_2-k'| \le 5}2^{-k_2}\norm{P_{k'}u_2}_{F_{k'}(T)}^2\\
&\lesssim \norm{u_2}_{F^{\frac14}(T)}^2\sum_{|k-k'|\le 5}\norm{P_{k'}w}_{F_{k'}(T)}^2 \\
&+\norm{u_2}_{F^{\frac14}(T)}\norm{w}_{F^{\frac14}(T)}\sum_{|k-k'|\le 5}\norm{P_{k'}u_2}_{F_{k'}(T)}\norm{P_{k'}w}_{F_{k'}(T)} \\
&+2^{-(s+\varepsilon)k}\norm{P_kw}_{F_k(T)}\norm{u_2}_{F^{0+}(T)}\norm{u_2}_{F^s(T)}\norm{w}_{F^0(T)},
\end{aligned}\]
for $s \ge 0$ and $0 < \varepsilon \ll 1$. Hence we conclude that
\[\sum_{k \ge 1} 2^{2sk} \sup_{t_k \in [0,T]} (\wt{B}_9(k) + \wt{B}_{10}(k)) \lesssim \norm{u_2}_{F^{\frac34}(T)}^2\norm{w}_{F^s(T)}^2,\]
for $s \ge 0$.

Together with all bounds of $\wt{B}_i(k)$, we obtain
\begin{equation}\label{eq:energy2-3.4}
\begin{aligned}
\sum_{k \ge 1} 2^{2sk} \sup_{t_k \in [0,T]} \left|\int_0^{t_k} \wt{E}_1 + \wt{E}_{2,1} + \wt{E}_{3,1} \; dt \right| &\lesssim \norm{u_2}_{F^{s}(T)}^2\norm{w}_{F^s(T)}^2 +\norm{u_2}_{F^s(T)}\norm{w}_{F^s(T)}^2\\ &+\norm{u_2}_{F^{2s}(T)}\norm{w}_{F^0(T)}\norm{w}_{F^s(T)}
\end{aligned}
\end{equation}
for $s \ge 2$ and
\begin{equation}\label{eq:energy2-3.5}
\begin{aligned}
\sum_{k \ge 1} \sup_{t_k \in [0,T]} \left|\int_0^{t_k} \wt{E}_1 + \wt{E}_{2,1} + \wt{E}_{3,1} \; dt \right| &\lesssim \norm{u_2}_{F^{\frac34}(T)}^2\norm{w}_{F^0(T)}^2\\
&+\norm{u_2}_{F^{\frac32+}(T)}\norm{w}_{F^0(T)}^2,
\end{aligned}
\end{equation} 
at $L^2$-level.

Next, we estimate
\[\left|\int_0^{t_k}\wt{E}_{2,2} + \wt{E}_{3,2} \; dt \right|.\]
But, since the total number of derivatives is less than that in $\wt{E}_{2,1}$ and $\wt{E}_{3,1}$ terms, we can easily control those terms and obtain
\[\left|\int_0^{t_k} \wt{E}_{2,2} + \wt{E}_{3,2} \; dt \right| \lesssim \norm{u_2}_{F^0(T)}\sum_{|k-k'|\le5}\norm{P_{k'}w}_{F_{k'}(T)}^2,\]
which implies
\begin{equation}\label{eq:energy2-3.6}
\sum_{k \ge 1}2^{2sk} \sup_{t_k \in [0,T]}\left|\int_0^{t_k} \wt{E}_{2,2} + \wt{E}_{3,2} \; dt \right| \lesssim \norm{u_2}_{F^0(T)}\norm{w}_{F^s(T)}^2,
\end{equation}
for $s \ge 0$.

Lastly, we focus on the cubic and quartic terms given by 
\[\left|\int_0^{t_k}\wt{E}_{2,3} + \wt{E}_{3,3} \; dt \right|.\]
We first estimate 
\begin{equation}\label{eq:energy2-3.7}
\left|\sum_{n,\overline{\N}_{2,n}}\int_0^{t_k}\wh{N}_2(u_2)(n_1)\chi_k(n_2)\frac1{n_2}\wh{w}(n_2)\chi_k(n)\frac1{n}\wh{w}(n) \; dt \right|.
\end{equation}
For $N_{1,1}$ in $N_{1}$, it is enough to estimate
\begin{equation}\label{eq:energy2-3.8}
\norm{u_2}_{F^0(T)}^2\sum_{k_1 \ge 0}2^{k_1}2^{-2k}\left|\sum_{n,\overline{\N}_{2,n}}\int_0^{t_k} \chi_{k_1}(n_1)\wh{u}_2(n_1)\chi_k(n_2)\wh{w}(n_2)\chi_k(n)\wh{w}(n) \; dt\right|.
\end{equation}
Using Lemma \ref{lem:energy2-1}, we obtain
\[\begin{aligned}
\eqref{eq:energy2-3.8} &\lesssim \norm{u_2}_{F^0(T)}^2\sum_{k_1 \le k- 10}2^{k_1/2}\norm{P_{k_1}u_2}_{F_{k_1}(T)}\sum_{|k-k'|\le 5}2^{-3k}\norm{P_{k'}w}_{F_{k'}(T)}^2\\
&+\norm{u_2}_{F^0(T)}^2\sum_{|k-k'|\le 5}2^{-5k/2}\norm{P_{k'}u_2}_{F_{k'}(T)}\norm{P_{k'}w}_{F_{k'}(T)}^2\\
&\lesssim \norm{u_2}_{F^0(T)}^3\sum_{|k-k'|\le 5}\norm{P_{k'}w}_{F_{k'}(T)}^2.
\end{aligned}\]
For $N_{1,2}$ in $N_1$, it suffices to consider
\begin{equation}\label{eq:energy2-3.9}
\sum_{0 \le k_1 \le k_2 \le k_3}2^{k_3}2^{-2k}\left|\sum_{\overline{n} \in \Gamma_5(\Z)}\int_0^{t_k} \prod_{i=1}^{3}\chi_{k_i}(n_i)\wh{u}_2(n_i)\chi_k(n_4)\wh{w}(n_4)\chi_k(n)\wh{w}(n) \; dt\right|.
\end{equation}
We use \eqref{eq:energy2-2.11} to obtain at most
\[\begin{aligned}
\eqref{eq:energy2-3.9} &\lesssim \sum_{\substack{k_3 \le k + 10 \\ 0 \le k_1 \le k_2 \le k_3}}2^{k_1/2}\prod_{i=1}^3\norm{P_{k_i}u_2}_{F_{k_i}(T)}\sum_{|k-k'|\le 5}\norm{P_{k'}w}_{F_{k'}(T)}^2\\
&+\sum_{\substack{k_3 \ge k + 10 \\ 0 \le k_1 \le k_2 \le k_3}}2^{k_3}2^{k_1/2}\prod_{i=1}^3\norm{P_{k_i}u_2}_{F_{k_i}(T)}\sum_{|k-k'|\le 5}\norm{P_{k'}w}_{F_{k'}(T)}^2\\
&\lesssim \norm{u_2}_{F^{\frac12+}(T)}^3\sum_{|k-k'|\le 5}\norm{P_{k'}w}_{F_{k'}(T)}^2.
\end{aligned}\]
For $N_{1,3}$ and $N_{1,4}$ in $N_1$, we need to estimate the following term as the worst term:
\begin{equation}\label{eq:energy2-3.10}
\sum_{0 \le k_1 \le k_2}2^{3k_2}2^{-2k}\left|\sum_{\overline{n} \in \Gamma_4(\Z)}\int_0^{t_k} \prod_{i=1}^{2}\chi_{k_i}(n_i)\wh{u}_2(n_i)\chi_k(n_3)\wh{w}(n_3)\chi_k(n)\wh{w}(n) \; dt\right|.
\end{equation}
We roughly estimate \eqref{eq:energy2-3.10} by using the Cauchy-Schwarz inequality to obtain
\[\begin{aligned}
\eqref{eq:energy2-3.10} &\lesssim \sum_{\substack{k_2 \le k + 10 \\ 0 \le k_1 \le k_2}}2^{k_1/2}2^{3k_2/2}\prod_{i=1}^2\norm{P_{k_i}u_2}_{F_{k_i}(T)}\sum_{|k-k'|\le 5}\norm{P_{k'}w}_{F_{k'}(T)}^2\\
&+\sum_{\substack{k_2 \ge k + 10 \\ |k_1 -k_2|\le 5}}2^{3k_2}\prod_{i=1}^2\norm{P_{k_i}u_2}_{F_{k_i}(T)}\sum_{|k-k'|\le 5}\norm{P_{k'}w}_{F_{k'}(T)}^2\\
&\lesssim \norm{u_2}_{F^{\frac32+}(T)}^2\sum_{|k-k'|\le 5}\norm{P_{k'}w}_{F_{k'}(T)}^2.
\end{aligned}\]
Hence we have
\[\sum_{k \ge 1} 2^{2sk} \sup_{t_k \in [0,T]} \eqref{eq:energy2-3.7} \lesssim \left(\norm{u_2}_{F^{\frac32+}(T)}^2 + \norm{u_2}_{F^{\frac12+}(T)}^3\right) \norm{w}_{F^s(T)}^2,\]
for $s \ge 0$.

For the rest terms in $\wt{E}_{2,3}$ and  $\wt{E}_{3,3}$, by the symmetry of $n_3$ and $n$ variables, it is enough to consider
\begin{equation}\label{eq:energy2-3.11}
\left|\sum_{n,\overline{\N}_{2,n}}\int_0^{t_k}\wh{u}_2(n_1)\chi_k(n_2)\frac1{n_2}\wh{N}_2(w)(n_2)\chi_k(n)\frac1{n}\wh{w}(n) \; dt \right|.
\end{equation}
For $N_{1,1}$ in $N_1$, similarly as the estimation of \eqref{eq:energy2-3.8}, we obtain
\[\norm{u_2}_{F^0(T)}^3\sum_{|k-k'|\le 5}\norm{P_{k'}w}_{F_{k'}(T)}^2.\]
For $N_{1,2}$ in $N_1$, we need to estimate
\begin{equation}\label{eq:energy2-3.12}
\sum_{\substack{0 \le k_1 \le k_2 \le k_3\\k_4 \ge 0}}2^{-k}\left|\sum_{\overline{n} \in \Gamma_5(\Z)}\int_0^{t_k} \prod_{i=1}^{3}\chi_{k_i}(n_i)\wh{u}_2(n_i)\chi_{k_4}(n_4)\wh{w}(n_4)\chi_k^2(n)\wh{w}(n) \; dt\right|.
\end{equation}
If $k = \max(k_1,k_2,k_3,k_4,k)$ and $|k-k_4|\le 5$, we use \eqref{eq:energy2-2.11} to obtain at most
\[\begin{aligned}
\eqref{eq:energy2-3.12} &\lesssim \sum_{0 \le k_1 \le k_2 \le k_3}2^{k_1/2}\prod_{i=1}^3\norm{P_{k_i}u_2}_{F_{k_i}(T)}\sum_{|k-k'|\le 5}\norm{P_{k'}w}_{F_{k'}(T)}^2\\
&\lesssim \norm{u_2}_{F^{\frac16+}(T)}^3\sum_{|k-k'|\le 5}\norm{P_{k'}w}_{F_{k'}(T)}^2.
\end{aligned}\]
If $|k_3-k|\le 5$, similarly, we obtain
\[\begin{aligned}
\eqref{eq:energy2-3.12} &\lesssim \sum_{\substack{0 \le k_1 \le k_2  \\ k_4 \ge 0}}2^{k_1/2}2^{k_2/2}2^{-k_4/2}\prod_{i=1}^2\norm{P_{k_i}u_2}_{F_{k_i}(T)}\norm{P_{k_4}w}_{F_{k_4}(T)}\sum_{|k-k'|\le 5}\norm{P_{k'}u_2}_{F_{k'}(T)}\norm{P_{k'}w}_{F_{k'}(T)}\\
&\lesssim \norm{u_2}_{F^{\frac12}(T)}\norm{w}_{F^0(T)}\sum_{|k-k'|\le 5}\norm{P_{k'}u_2}_{F_{k'}(T)}\norm{P_{k'}w}_{F_{k'}(T)}.
\end{aligned}\]
On the other hand, if $k \neq \max(k_1,k_2,k_3,k_4,k)$, we use \eqref{eq:energy2-2.11} to obtain that
\[\begin{aligned}
\eqref{eq:energy2-3.12} &\lesssim \sum_{k_1,k_3 \ge 0}2^{k_1/2}2^{k_3/2}2^{-k/2}\norm{P_{k_1}u_2}_{F_{k_1}(T)}\norm{P_{k_4}w}_{F_{k_4}(T)}\norm{P_{k}w}_{F_k(T)}\sum_{\substack{k_3 \ge k_4 + 10\\|k_2-k_3|\le 5|}}\norm{P_{k_2}u_2}_{F_{k_2}(T)}^2\\
&+\sum_{0 \le k_1 \le k_2}2^{k_1/2}2^{k_2/2}2^{-k/2}\prod_{i=1}^2\norm{P_{k_i}u_2}_{F_{k_i}(T)}\norm{P_{k}w}_{F_k(T)}\sum_{|k_3-k_4|\le 5}\norm{P_{k_3}u_2}_{F_{k_3}(T)}\norm{P_{k_4}w}_{F_{k_4}(T)}\\
&\lesssim 2^{-(s+\frac14)k}\norm{P_kw}_{F_k(T)}\left(\norm{u_2}_{F^{\frac12}(T)}^3\norm{w}_{F^s(T)} + \norm{u_2}_{F^{\frac12}(T)}^2\norm{u_2}_{F^{s}(T)}\norm{w}_{F^0(T)}\right),
\end{aligned}\]
for $s \ge 0$.

Now we consider $N_{1,3}$ and $N_{1,4}$ portions in $N_1$.
\begin{remark}\label{rem:resonant4}
Similarly as Remark \ref{rem:resonant3}, we need to check carefully the cubic resonant interaction components. From \eqref{eq:energy-nonlinear2-3} and \eqref{eq:energy-nonlinear2-4} and the cubic resonance relation, there are following terms as the cubic resonant terms:
\[\sum_{n_1 \in \Z}\wh{u}_2(n_1)\chi_k(n+n_1)(\wh{w}(-n_1)\wh{u}_1(-n)+\wh{u}_2(-n_1)\wh{w}(-n))\chi_k(n)n\wh{w}(n),\]
\[\sum_{n_1 \in \Z}\wh{u}_2(n_1)\chi_k(n+n_1)\frac{n_1}{n+n_1}(\wh{w}(-n_1)\wh{u}_1(-n)+\wh{u}_2(-n_1)\wh{w}(-n))\chi_k(n)n\wh{w}(n),\]
\[\sum_{n_1 \in \Z}\wh{u}_2(n_1)\chi_k(n+n_1)n_1^2(\wh{w}(-n)\wh{u}_1(-n_1)+\wh{u}_2(-n)\wh{w}(-n_1))\chi_k(n)\frac1n\wh{w}(n)\]
and
\[\sum_{n_1 \in \Z}\wh{u}_2(n_1)\chi_k(n+n_1)\frac{n_1^2}{n+n_1}(\wh{w}(-n)\wh{u}_1(-n_1)+\wh{u}_2(-n)\wh{w}(-n_1))\chi_k(n)\wh{w}(n).\]
Since the worst term
\[|\wh{u}_2(n_1)|^2\chi_k(n+n_1)\chi_k(n)n|\wh{w}(n)|^2\]
is real number, so this term vanishes. For the other terms, we use the Cauchy-Schwarz inequality and embedding property \eqref{eq:small data1.1} to obtain the bound at most
\[\norm{u_1}_{F^0(T)}\norm{u_2}_{F^{s+1}(T)}\norm{w}_{F^0(T)}\norm{w}_{F^s(T)},\]
by performing the summation over $k$, whenever $s \ge 0$.

Hence, in the following cubic estimates, we do not need to consider the resonant case any more.
\end{remark} 
To complete the proof of Proposition \ref{prop:energy2-3}, we need to consider 
\begin{equation}\label{eq:energy2-3.13}
\sum_{\substack{0 \le k_1 \le k_2\\k_3 \ge 0}} 2^{2k_3}2^{-k}\left|\sum_{n,\overline{N}_{3,n}}\int_0^{t_k}\chi_{k_1}(n_1)\wh{u}_2(n_1)\chi_{k_2}(n_2)\wh{u}_2(n_2)\chi_{k_3}(n_3)\wh{w}(n_3)\chi_k^2(n)\wh{w}(n) \; dt \right|,
\end{equation}
\begin{equation}\label{eq:energy2-3.14}
\sum_{\substack{0 \le k_1 \le k_3\\k_2 \ge 0}} 2^{2k_3}2^{-k}\left|\sum_{n,\overline{N}_{3,n}}\int_0^{t_k}\chi_{k_1}(n_1)\wh{u}_2(n_1)\chi_{k_2}(n_2)\wh{w}(n_2)\chi_{k_3}(n_3)\wh{u}_2(n_3)\chi_k^2(n)\wh{w}(n) \; dt \right|,
\end{equation}
\begin{equation}\label{eq:energy2-3.15}
\sum_{\substack{0 \le k_1 \le k_2\\k_3 \ge 0}} 2^{k_2}2^{2k_3}2^{-2k}\left|\sum_{n,\overline{N}_{3,n}}\int_0^{t_k}\chi_{k_1}(n_1)\wh{u}_2(n_1)\chi_{k_2}(n_2)\wh{u}_2(n_2)\chi_{k_3}(n_3)\wh{w}(n_3)\chi_k^2(n)\wh{w}(n) \; dt \right|
\end{equation}
and
\begin{equation}\label{eq:energy2-3.16}
\sum_{\substack{0 \le k_1 \le k_3\\k_2 \ge 0}}  2^{k_2}2^{2k_3}2^{-2k}\left|\sum_{n,\overline{N}_{3,n}}\int_0^{t_k}\chi_{k_1}(n_1)\wh{u}_2(n_1)\chi_{k_2}(n_2)\wh{w}(n_2)\chi_{k_3}(n_3)\wh{u}_2(n_3)\chi_k^2(n)\wh{w}(n) \; dt \right|.
\end{equation}
First we assume that $k = \max(k_1,k_2,k_3,k)$. If $|k-k_3| \le 5$, \eqref{eq:energy2-3.13} and \eqref{eq:energy2-3.14} are dominant, then by using Lemma \ref{lem:energy-cubic}, we obtain
\[\begin{aligned}
\eqref{eq:energy2-3.13} &\lesssim \sum_{|k-k'|\le 5}2^{3k/2}\norm{P_{k'}u_2}_{F_{k'}(T)}^2\norm{P_{k'}w}_{F_{k'}(T)}^2\\
&+\sum_{k_1 \le k-10}2^{k_1/2}\norm{P_{k_1}u_2}_{F_{k_1}(T)}\sum_{|k-k'|\le 5}\norm{P_{k'}u_2}_{F_{k'}(T)}\norm{P_{k'}w}_{F_{k'}(T)}^2\\
&+\sum_{\substack{k_2 \le k - 10\\ |k_1 - k_2| \le 5}}2^{k_1/2}\norm{P_{k_1}u_2}_{F_{k_1}(T)}^2\sum_{|k-k'|\le 5}\norm{P_{k'}w}_{F_{k'}(T)}^2\\
&+\sum_{\substack{k_2 \le k - 10\\ k_1 \le k_2 - 10}}\norm{P_{k_1}u_2}_{F_{k_1}(T)}\norm{P_{k_2}u_2}_{F_{k_2}(T)}\sum_{|k-k'|\le 5}\norm{P_{k'}w}_{F_{k'}(T)}^2\\
&\lesssim \norm{u_2}_{F^{\frac14+}(T)}^2\sum_{|k-k'|\le 5}\norm{P_{k'}w}_{F_{k'}(T)}^2
\end{aligned}\] 
and
\[\begin{aligned}
\eqref{eq:energy2-3.14} &\lesssim \sum_{k_2 \le k-10}2^{k_2/2}\norm{P_{k_2}w}_{F_{k_2}(T)}\sum_{|k-k'|\le 5}\norm{P_{k'}u_2}_{F_{k'}(T)}^2\norm{P_{k'}w}_{F_{k'}(T)}\\
&+\sum_{\substack{k_1, k_2 \le k - 10\\ |k_1 - k_2| \le 5}}2^{k_1/2}\norm{P_{k_1}u_2}_{F_{k_1}(T)}\norm{P_{k_2}w}_{F_{k_2}(T)}\sum_{|k-k'|\le 5}\norm{P_{k'}u_2}_{F_{k'}(T)}\norm{P_{k'}w}_{F_{k'}(T)}\\
&+\sum_{\substack{k_1, k_2 \le k - 10\\ k_1 \le k_2 - 10}}\norm{P_{k_1}u_2}_{F_{k_1}(T)}\norm{P_{k_2}w}_{F_{k_2}(T)}\sum_{|k-k'|\le 5}\norm{P_{k'}u_2}_{F_{k'}(T)}\norm{P_{k'}w}_{F_{k'}(T)}\\
&+\sum_{\substack{k_1, k_2 \le k - 10\\ k_2 \le k_1 - 10}}\norm{P_{k_1}u_2}_{F_{k_1}(T)}\norm{P_{k_2}w}_{F_{k_2}(T)}\sum_{|k-k'|\le 5}\norm{P_{k'}u_2}_{F_{k'}(T)}\norm{P_{k'}w}_{F_{k'}(T)}\\
&\lesssim (\norm{u_2}_{F^{\frac12+}(T)}\norm{w}_{F^{0}(T)}+\norm{u_2}_{F^{\frac12}(T)}\norm{w}_{F^{0+}(T)})\sum_{|k-k'|\le 5}\norm{P_{k'}u_2}_{F_{k'}(T)}\norm{P_{k'}w}_{F_{k'}(T)}.
\end{aligned}\] 
If $k \neq \max(k_1,k_2,k_3,k)$, \eqref{eq:energy2-3.15} and \eqref{eq:energy2-3.16} are dominant. If $|k_1 - k_2| \le 5$ and $|k_2-k_3|\le 5$, we do not distinguish between \eqref{eq:energy2-3.15} and \eqref{eq:energy2-3.16}, and by using \eqref{eq:energy-cubic2},  we obtain that
\[\begin{aligned}
\eqref{eq:energy2-3.15} &\lesssim 2^{-3k/2}\norm{P_kw}_{F_k(T)}\sum_{\substack{k_3 \ge k+10 \\ |k_3 - k'|\le 5}}2^{2k_3}\norm{P_{k'}u_2}_{F_{k'}(T)}^2\norm{P_{k'}w}_{F_{k'}(T)}\\
&\lesssim 2^{-(s+3/2)k}\norm{P_kw}_{F_k(T)}\norm{u_2}_{F^1(T)}^2\norm{w}_{F^s(T)},
\end{aligned}\]
whenever $s \ge 0$.
If $|k_2-k_3| \le 5$ and $k_1 \le k_2 - 10$, we use \eqref{eq:energy-cubic3} and \eqref{eq:energy-cubic4} to obtain that
\[\begin{aligned}
\eqref{eq:energy2-3.15} &\lesssim \sum_{\substack{k_1 \le k_2- 10 \\ |k_1-k|\le 5}}2^{-3k/2}\norm{P_{k_1}u_2}_{F_{k_1}(T)}\norm{P_kw}_{F_k(T)} \sum_{|k_2-k_3|\le 5}2^{2k_3}\norm{P_{k_2}u_2}_{F_{k_2}(T)}\norm{P_{k_3}w}_{F_{k_3}(T)} \\
&+\sum_{\substack{k_1 \le k_2- 10 \\ k_1 \le k - 10}}2^{-2k}\norm{P_{k_1}u_2}_{F_{k_1}(T)}\norm{P_kw}_{F_k(T)} \sum_{|k_2-k_3|\le 5}2^{2k_3}\norm{P_{k_2}u_2}_{F_{k_2}(T)}\norm{P_{k_3}w}_{F_{k_3}(T)}\\
&+\sum_{\substack{k_1 \le k_2- 10 \\ k_1 \ge k + 10}}2^{-2k}\norm{P_{k_1}u_2}_{F_{k_1}(T)}\norm{P_kw}_{F_k(T)} \sum_{|k_2-k_3|\le 5}2^{2k_3}\norm{P_{k_2}u_2}_{F_{k_2}(T)}\norm{P_{k_3}w}_{F_{k_3}(T)}\\
&\lesssim \norm{u_2}_{F^2(T)}\norm{w}_{F^0(T)}\sum_{|k-k'|\le5}\norm{P_{k'}u_2}_{F_{k'}(T)}\norm{P_{k'}w}_{F_{k'}(T)}\\
&+2^{-(s+3/2)k}\norm{P_kw}_{F_k(T)}\norm{u_2}_{F^2(T)}^2\norm{w}_{F^s(T)}.
\end{aligned}\]
Finally, we consider the case when $|k_1 - k_2| \le 5$ and $k_3 \le k_1 - 10$ in \eqref{eq:energy2-3.15} or $|k_1-k_3| \le 5$ and $k_2 \le k_1 - 10$ in \eqref{eq:energy2-3.16}. Since the second case is dominant, we use \eqref{eq:energy-cubic3} and \eqref{eq:energy-cubic4} to obtain that
\[\begin{aligned}
\eqref{eq:energy2-3.16} &\lesssim \sum_{\substack{k_2 \le k_3- 10 \\ |k_2-k|\le 5}}2^{-k/2}\norm{P_{k_2}w}_{F_{k_2}(T)}^2 \sum_{|k_1-k_3|\le 5}2^{k_3}\norm{P_{k_1}u_2}_{F_{k_2}(T)}^2 \\
&+\sum_{\substack{k_2 \le k_3- 10 \\ k_2 \le k - 10}}2^{-k}\norm{P_{k_2}w}_{F_{k_2}(T)}\norm{P_kw}_{F_k(T)} \sum_{|k_1-k_3|\le 5}2^{k_3}\norm{P_{k_1}u_2}_{F_{k_2}(T)}^2\\
&+\sum_{\substack{k_2 \le k_3- 10 \\ k_2 \ge k + 10}}2^{k_2}2^{-2k}\norm{P_{k_2}w}_{F_{k_2}(T)}\norm{P_kw}_{F_k(T)} \sum_{|k_1-k_3|\le 5}2^{k_3}\norm{P_{k_1}u_2}_{F_{k_2}(T)}^2\\
&\lesssim \norm{u_2}_{F^{\frac12}(T)}^2\sum_{|k-k'|\le5}\norm{P_{k'}w}_{F_{k'}(T)}^2\\
&+2^{-(s+1/2)k}\norm{P_kw}_{F_k(T)}\norm{u_2}_{F^{1+}(T)}(\norm{u_2}_{F^s(T)}\norm{w}_{F^0(T)} + \norm{u_2}_{F^{1+}(T)}\norm{w}_{F^s(T)}),
\end{aligned}\]
when $s \ge 0$.

Hence, we have
\[\begin{aligned}
\sum_{k \ge 1} 2^{2sk} \sup_{t_k \in [0,T]} \eqref{eq:energy2-3.11} &\lesssim \norm{u_2}_{F^2(T)}^2\norm{w}_{F^s(T)}^2\\
&+\norm{u_2}_{F^0(T)}\norm{u_2}_{F^{s+1}(T)}\norm{w}_{F^0(T)}\norm{w}_{F^s(T)},
\end{aligned}\]
when $s \ge 0$, and conclude that
\begin{equation}\label{eq:energy2-3.17}
\begin{aligned}
\sum_{k \ge 1} 2^{2sk} \sup_{t_k \in [0,T]}\left|\int_0^{t_k}\wt{E}_{2,3} + \wt{E}_{3,3} \; dt \right| & \lesssim \norm{u_2}_{F^2(T)}^2\norm{w}_{F^s(T)}^2\\
&+\norm{u_2}_{F^0(T)}\norm{u_2}_{F^{s+1}(T)}\norm{w}_{F^0(T)}\norm{w}_{F^s(T)}\\
&+\norm{u_2}_{F^{\frac12+}(T)}^3\norm{w}_{F^s(T)}^2.
\end{aligned}
\end{equation}
Together with \eqref{eq:energy2-3.4}, \eqref{eq:energy2-3.6} and \eqref{eq:energy2-3.17} for $s \ge 2$, and \eqref{eq:energy2-3.5}, \eqref{eq:energy2-3.6} and \eqref{eq:energy2-3.17} for $L^2$-level, we complete the proof of \eqref{eq:energy2-3.1} and \eqref{eq:energy2-3.2}, respectively.
\end{proof}

As a Corollary to Lemma \ref{lem:comparable energy2-2} and Proposition \ref{prop:energy2-3}, we obtain \emph{a priori bound} of $\norm{w}_{E^s(T)}$ for the difference of two solutions.
\begin{corollary}\label{cor:energy2-3}
Let $s \ge 2$ and $T \in (0,1]$. Then, there exists $0 < \delta \ll 1$ such that
\[\begin{aligned}
\norm{w}_{E^0(T)} &\lesssim (1+ \norm{u_{1,0}}_{H^{\frac12+}}+\norm{u_{2,0}}_{H^{\frac12+}})\norm{w_0}_{L_x^2}^2\\
&+(1+ \norm{u_1}_{F^2(T)}+\norm{u_2}_{F^2(T)})(\norm{u_1}_{F^2(T)}+\norm{u_2}_{F^2(T)})\norm{w}_{F^0(T)}^2\\
&+\left(\sum_{1 \le i \le j \le k \le 2}\norm{u_i}_{F^2(T)}\norm{u_j}_{F^2(T)}\norm{u_k}_{F^2(T)}\right)\norm{w}_{F^0(T)}^2.
\end{aligned} \]
and
\[\begin{aligned}
\norm{w}_{E^s(T)} &\lesssim (1+ \norm{u_{2,0}}_{H^{\frac12+}})\norm{w_0}_{H^s}^2\\
&+(\norm{u_1}_{F^{2s}(T)}+\norm{u_2}_{F^{2s}(T)})\norm{w}_{F^0(T)}\norm{w}_{F^s(T)}\\
&+(\norm{u_1}_{F^s(T)}+\norm{u_2}_{F^s(T)})\norm{w}_{F^s(T)}^2\\
&+(\norm{u_1}_{F^s(T)}+\norm{u_2}_{F^s(T)})(\norm{u_1}_{F^{2s}(T)}+\norm{u_2}_{F^{2s}(T)})\norm{w}_{F^0(T)}\norm{w}_{F^s(T)}\\
&+\left(\sum_{1 \le i \le j \le 2}\norm{u_i}_{F^s(T)}\norm{u_j}_{F^s(T)}\right)\norm{w}_{F^s(T)}^2\\
&+\left(\sum_{1 \le i \le j \le k \le 2}\norm{u_i}_{F^s(T)}\norm{u_j}_{F^s(T)}\norm{u_k}_{F^s(T)}\right)\norm{w}_{F^s(T)}^2,
\end{aligned}\] 
for solutions $w \in C([-T,T];H^{\infty}(\T))$ to \eqref{eq:5kdv5} and $u_1, u_2 \in C([-T,T];H^{\infty}(\T))$ to \eqref{eq:5kdv4} satisfying $\norm{u_1}_{L_T^{\infty}H_x^{\frac12+}} < \delta$ and $\norm{u_2}_{L_T^{\infty}H_x^{\frac12+}} < \delta$.
\end{corollary}

\end{document}